\documentclass[12pt]{amsart}
\usepackage{amsmath,amssymb,amsthm}

\DeclareMathOperator{\supp}{supp} 
\DeclareMathOperator{\diag}{diag} 
 
 \DeclareMathOperator{\dist}{d}

\newcommand{\la}{\langle} \newcommand{\ra}{\rangle}

\newcommand{\ls}{\lesssim} \newcommand{\gs}{\gtrsim}
\newcommand{\R}{\mathbb{R}} \newcommand{\C}{\mathbb{C}}
\newcommand{\N}{\mathbb{N}} \newcommand{\Z}{\mathbb{Z}}
\newcommand{\Sp}{\mathbb{S}} 
\newcommand{\al}{\alpha} \newcommand{\om}{\omega}
 \newcommand{\ka}{\kappa}
\newcommand{\F}{\mathcal{F}} 
 \newcommand{\Ka}{\mathcal{K}}
\newcommand{\Th}{\Theta} \newcommand{\lS}{l^2S}

\newtheorem{thm}{Theorem} \newtheorem{cor}[thm]{Corollary}
\newtheorem{pro}[thm]{Proposition} \newtheorem{lem}[thm]{Lemma}

\theoremstyle{remark} \newtheorem{rmk}{Remark}
\theoremstyle{definition} 

\numberwithin{equation}{section} \numberwithin{thm}{section}

\begin{document}

\title[The cubic Dirac equation]{The cubic Dirac equation: Small
  initial data in $H^{\frac12}(\R^2)$}

\author[I.~Bejenaru]{Ioan Bejenaru} \address[I.~Bejenaru]{Department
  of Mathematics, University of California, San Diego, La Jolla, CA
  92093-0112, USA} \email{ibejenaru@math.ucsd.edu}

\author[S.~Herr]{Sebastian Herr} \address[S.~Herr]{Fakult\"at f\"ur
  Mathematik, Universit\"at Bielefeld, Postfach 10 01 31, 33501
  Bielefeld, Germany} \email{herr@math.uni-bielefeld.de}

\begin{abstract}
  Global well-posedness and scattering for the cubic Dirac equation
  with small initial data in the critical space $H^{\frac12}(\R^2)$ is
  established. The proof is based on a sharp endpoint Strichartz
  estimate for the Klein-Gordon equation in dimension $n=2$, which is
  captured by constructing an adapted systems of coordinate frames.
\end{abstract}


\maketitle

\section{Introduction and main results}\label{sect:intro}
In this paper we continue our investigation initiated in \cite{BH}
regarding the full range of Strichartz estimates available for the
Klein-Gordon equation, with the particular goal of providing
$L^2L^\infty$ type estimates. As an application we prove global
well-posedness and scattering for the cubic Dirac equation with small
data in the critical space.

For fixed $m>0$, we consider the (scalar) homogeneous
Klein-Gordon equation
\begin{equation} \label{KG} \Box u + m^2 u =0, \qquad u: \R
  \times \R^n \rightarrow \R \ \mbox{or} \ \C.
\end{equation}
The validity of Strichartz estimates for solutions $u$ of this
equation is a fundamental and well-studied problem.  In the low
frequency regime, the dispersive properties of the Klein-Gordon
equation are similar to the Schr\"odinger equation, i.e.\ the decay
rate of the fundamental solution is $t^{-\frac{n}2}$. In the high frequency regime they are
similar to the wave equation, i.e.\ the decay rate is
$t^{-\frac{n-1}2}$.  In the high frequency regime there is also a
penalized Schr\"odinger-type decay: the fundamental solution localized at frequency $2^k$ decay as
$2^{k} t^{-\frac{n}2}$; the penalization is due to the small curvature
of the characteristic surface.  If one is not concerned with sharp
estimates, in the high frequency regime one could trade regularity for
having access to the better decay $t^{-\frac{n}2}$. Such an approach
severely limits the range of applications, in particular to low
regularity nonlinear problems.

The decay rate of the fundamental solution plays a crucial role in determining the
range of available Strichartz estimates.  It is well-known that the
endpoint Strichartz $L^2_t L^\infty_x$ estimate fails for the wave
equation in dimensions $n=3$ and for the Schr\"odinger equation in
dimension $n=2$, see \cite{M-S,Tao-bil}. For the Klein-Gordon
equation \eqref{KG} in three dimensions, the endpoint $L^2_t
L^\infty_x$ estimate does not fail if one allows for a loss of
regularity, see \cite{MNO}. However, the sharp $L^2_t L^\infty_x$
estimate (dictated by scaling) fails to hold true.  In \cite{BH} we
provided a microlocal replacement of the missing sharp endpoint $L^2_t
L^\infty_x$ Strichartz estimate in dimension $n=3$ by using adapted frames.

In dimension $n=2$ the same problem becomes significantly more
difficult since both endpoint Strichartz
estimates for the wave equation, $L^4_t L^\infty_x$,
and for the Schr\"odinger equation, $L^2_tL^\infty_x$, fail to hold. In this paper we address this problem
by providing $L^2 L^\infty$ estimates in adapted frames. For the Klein-Gordon equation in dimension $n=2$, to our best
knowledge, these estimates are novel in literature.

Throughout the rest of this paper, we fix the physical dimension
$n=2$.  In applications to nonlinear problems, see \cite{MNNO,BH} for
the cubic Dirac equation in three dimensions, the endpoint
Strichartz estimate and the $L^\infty L^2$ energy estimate imply a
bilinear $L^2_{t,x}$ estimate via the toy scheme
\[
\| u \cdot v \|_{L^2_{t,x}} \leq \| u \|_{L^2 L^\infty} \| v
\|_{L^\infty L^2}.
\]
Since the $L^2 L^\infty$ estimate will be established in adapted
frames, energy estimates in similar frames are needed to recover the
above $L^2_{t,x}$ bilinear estimate.  As in dimension $n=3$ in
\cite{BH}, combining the energy and the Strichartz estimate to derive
a uniform $L^2$ estimate is only possible in presence of a null
structure, see Section \ref{sect:setupD}.

The idea to use adapted frames in order to find a replacement for the
missing $L^2 L^\infty$ endpoint Strichartz estimate is due to Tataru
\cite{tat}, and was motivated by the Wave maps problem. In the
context of the Schr\"odinger equation, this was done for solving the
Schr\"odinger Map problem in two dimensions in \cite{bikt}. Naively,
one may expect that by using the structures in \cite{tat} and
\cite{bikt}, one can address the same problem for the Klein-Gordon,
but this is not the case. The reason is two-fold: there are no
straight lines (zero curvature submanifolds) foliating the
characteristic surface so as to emulate the Wave Equation
construction; trading regularity in order to rely only on the
Schr\"odinger equation would provide non-optimal estimates.
 
Instead, our current work builds on ideas from \cite{tat} and
\cite{bikt} and brings new ideas to provide a more complex
construction well-adapted the geometry of the characteristic surface
for the Klein-Gordon equation.

As an application, we study the cubic Dirac equation in dimension
$n=2$ at the critical regularity: Fix $M> 0$. Using the summation
convention, the cubic Dirac equation for the spinor field $\psi: \R
\times \R^2 \to \C^2$ is
\begin{equation} \label{eq:dirac} (-i \gamma^\mu \partial_\mu + M )
  \psi= \la \gamma^0 \psi, \psi \ra \psi,
\end{equation}
where $\gamma^\mu\in \C^{2\times 2}$ are the Dirac matrices
\[
\gamma^0= \left( \begin{array}{cc} 1 & 0 \\ 0 & -1 \end{array}
\right) , \qquad \gamma^1=\left( \begin{array}{cc} 0 & 1 \\
    -1 & 0 \end{array} \right), \qquad \gamma^2=\left( \begin{array}{cc} 0 & -i \\
    -i & 0 \end{array} \right),
\]
and $\langle \cdot, \cdot \rangle$ is the standard inner product on
$\C^2$.

The matrices $\gamma^\mu$ satisfy $ \gamma^\alpha \gamma^\beta +
\gamma^\beta \gamma^\alpha = 2 g^{\alpha \beta} I_2$, where
$(g^{\alpha \beta})= \diag(1,-1)$.  By adapting the set of matrices,
the equation \eqref{eq:dirac} can be written in any spatial
dimension. We refer the reader to \cite{flr,soler} for the physical
background for this equation.

The $n$-dimensional version of \eqref{eq:dirac} becomes critical in
$H^{\frac{n-1}2}(\R^n)$ in the sense that it is (approximately)
invariant under rescaling of solutions. In three dimensions the
equation was studied extensively, see \cite{ev,MNO,MNNO,sv,cv,M88} and
references therein. The global well-posedness for small data in the critical space was established
by the authors in \cite{BH}.

In dimension $n=2$ and $M \ne 0$, we are aware of only two results:
\cite{Pe13,Pe13-corr} where Pecher establishes local well-posedness of
the equation with initial data in $H^s(\R^2)$ for $s > \frac12$ and
\cite{BC} where Bournaveas and Candy establish local well-posedness of
the equation with initial data in $H^\frac12(\R^2)$.  To our best
knowledge, no global well-posedness result is known so far.
The case $M=0$ has been settled in \cite{BC} where Bournaveas and
Candy also prove global well-posedness and scattering for small
initial data in $\dot H^\frac12(\R^2)$, see more commentaries below
about this case.

Our main result in this paper is

\begin{thm}\label{thm:main} Let $M\ne 0$.
  The initial value problem associated to the cubic Dirac equation
  \eqref{eq:dirac} is globally well-posed for small initial data $\psi(0)\in H^\frac12(\R^2)$. Moreover, these solutions scatter to free
  solutions for $t\to \pm \infty$.
\end{thm}

For results in space dimension $n=1$, see \cite{MNT10,C11}.

A special case arises in the massless variant of the cubic Dirac
equation, that is \eqref{eq:dirac} with $M=0$. A recent result of
Bournaveas and Candy \cite{BC} provides the equivalent result of
Theorem \ref{thm:main} for the case $M=0$. Their strategy stems from
the observation that the massless case carries similarities to the
Wave Maps equation. The authors tailor their resolution spaces around
the original ones introduced by Tataru \cite{tat} in the context of
Wave Maps. In order to overcome the Besov space obstacle, the authors
of \cite{BC} used an idea from \cite{BH}, i.e.\ adding a high
modulation nonlinear structure of type $L^p_tL^2_x$ for certain
$p<2$. The authors of \cite{BC} also obtain a local in time result for
$M \ne 0$ by treating the mass term $M\psi$ as a
perturbation. However, from the perspective of obtaining a global in
time result, the above strategy is limited to the case $M=0$ since the
resolution spaces for $M \ne 0$ were not known prior to the work in
the present paper.

Our results here and the one in dimension $n=3$ from \cite{BH} may
seem orthogonal to the work of Bournaveas and Candy \cite{BC}. Indeed,
we do not address directly the problem with $M=0$.  However by
passing to the high frequency limit one can ---at least formally---
recoup the results for $M=0$ since we work in the in the scale
invariant space dictated by the wave part. We do not formalize this
here and note that the approach in \cite{BC} is a more elegant and
easier way to deal with this problem with $M=0$.  It is an instructive
exercise is to check that, on fixed bounded time intervals, our
structures become in the high frequency limit the ones used in
\cite{BC} and originating in the work of Tataru \cite{tat}.

We describe some of the key ideas involved in this paper. The
Klein-Gordon waves travel with speed strictly less than $1$, though in
the high frequency limit the speed converges to $1$. Our frames
 capture the speed variation of these waves as well as their
directions, and this is why we work with two parameters: $\omega$
(angle) and $\lambda$ (speed). Having a precise formulation on how the
range of speed parameter $\lambda$ depends on the frequency plays a
crucial role in the argument.
 
The first system of frames we construct to recover an $L^2 L^\infty$
estimate stems from the one used \cite{BH}.  An additional level of
complexity is required due to the fact that once the high frequency
waves enter the Schr\"odinger regime the decay rate fails to provide
us with a classical $L^2_t L^\infty_x$ estimate. To fix this issue we
need a bi-parameter system of frames which depends both on $\omega$
(angle) and $\lambda$ (speed).

The next problem arises from that the above system is well suited for
most angular interactions, but fails near the parallel interactions
(in fact it works at exact parallel interactions).  Moreover, the null
structure cannot fix this failure as usually is the case.  To remedy
this problem we construct another system of frames which is suited
precisely to those angular scales and highlights a key geometrical
property of wave interactions: waves with distinct frequencies
travel with different speeds in the context of the Klein-Gordon
equation.

The paper is organized is as follows: In the following
subsection we introduce notation.  Section \ref{sect:le} is devoted to
endpoint Strichartz and energy estimates. In Section \ref{sect:setupD}
we recall the null-structure of the cubic Dirac equation. In Section
\ref{fspaces} we construct function spaces for the nonlinear
problem. In Section \ref{sect:bil-est} we prove auxiliary bilinear and
trilinear estimates. In Section \ref{sect:dirac} we prove the crucial
nonlinear estimates and provide a proof of Theorem \ref{thm:main}.

We point out that the notation, setup and general
reductions in the present paper are adopted from \cite{BH}. Also, we will repeatedly refer
to \cite{BH} for arguments which are similar in two and three
dimensions. As indicated above, the analysis in this paper is
significantly more involved, so it might be useful for the reader to
take a look at \cite{BH}, too.

\subsection{Notation}\label{subsect:not}
Here, we repeat the notation from \cite[Subsection 1.1]{BH} and adjust
it to the case $n=2$: We define $A\prec B$ by $A\leq B-c$ for some
absolute constant $c>0$. Also, we define $A\ll B$ by $A\leq d B$
for some absolute small constant $0<d<1$. Similarly, we define $A\ls
B$ to be $A\leq e B$ for some absolute constant $e>0$, and $A \approx
B$ iff $A\ls B \ls A$.

Let $\dist(M_1,M_2)$ denote the euclidean distance between the two sets 
$M_1,M_2\subset \R^2$.

We set $\la \xi \ra_k:=(2^{-2k}+|\xi|^2)^{\frac12}$ for $k \in \Z$ and $\xi \in \R^2$,
and write $\la\xi\ra:=\la\xi\ra_0$.

Throughout the paper, let $\rho\in C^\infty_c(-2,2)$ be a fixed
smooth, even, cutoff satisfying $\rho(s)=1$ for $|s|\leq 1$ and $0\leq
\rho\leq 1$. For $k \in \Z$ we define $\chi_k:\R^2 \to \R$,
$\chi_k(y):=\rho(2^{-k}|y|)-\rho(2^{-k+1}|y|)$, such that
$A_k:=\supp(\chi_k)\subset \{y \in \R^2 \colon 2^{k-1}\leq |y|\leq
2^{k+1}\}$. Let $\tilde{\chi}_k=\chi_{k-1}+\chi_k+\chi_{k+1}$ and
$\tilde{A}_k:=\supp(\tilde{\chi}_k)$.

We denote by $P_k = \chi_k(D)$ and $\tilde{P}_k =
\tilde{\chi}_k(D)$. Note that $P_k\tilde{P}_k=\tilde{P}_kP_k=P_k$.
Further, we define $\chi_{\leq k}=\sum_{l=-\infty}^k \chi_l$, $\chi_{>
  k}=1-\chi_{\leq k}$ as well as the corresponding operators $P_{\leq
  k}=\chi_{\leq k}(D)$ and $P_{> k}=\chi_{> k}(D)$.

We denote by $\mathcal{K}_l$ a collection of spherical arcs (caps) of
diameter $2^{-l}$ which provide a symmetric and finitely overlapping
cover of the unit circle $\mathbb{S}^1$. Let $\om(\ka)$ to be the
``center'' of $\ka$ and let $\Gamma_{\ka}\subset \R^2$ be the cone
generated by $\ka$ and the origin, in particular $\Gamma_{\ka}\cap
\mathbb{S}^1 =\ka$.

Further, let $\eta_{\ka}$ be smooth partition of unity subordinate to
the covering of $\R^2\setminus\{0\}$ with the cones $\Gamma_{\ka}$,
such that each $\eta_\ka$ is supported in $\frac32 \Gamma_{\ka}$ and
is homogeneous of degree zero and satisfies
\[
|\partial_\xi^\beta\eta_\kappa(\xi)|\leq C_\beta
2^{l|\beta|}|\xi|^{-\beta}, \quad
|(\omega(\ka)\cdot\nabla)^N\eta_\kappa(\xi)|\leq C_N |\xi|^{-N}.
\]
Let $\tilde{\eta}_{\ka}$ with similar properties but slightly bigger
support $2 \Gamma_{\ka}$, such that $\tilde{\eta}_{\ka}\eta_{\ka}=1$.
We define $P_{\kappa}=\eta_{\ka}(D)$,
$\tilde{P}_{\kappa}=\tilde{\eta}_{\ka}(D)$.  With $P_{k,\ka}:=
\eta_{\ka}(D) \chi_k(D)$ and $\tilde{P}_{k,\ka}:=
\tilde{\eta}_{\ka}(D) \tilde{\chi}_k(D)$, we obtain the angular
decomposition
\[
P_k = \sum_{\ka \in \mathcal{K}_l} P_{k,\ka}
\]
and $P_{k,\ka}\tilde{P}_{k,\ka}=\tilde{P}_{k,\ka}P_{k,\ka}=P_{k,\ka}$.
We further define $A_{k,\ka}=\supp(\eta_{\ka} \chi_k)$ and
$\tilde{A}_{k,\ka}=\supp(\tilde{\eta}_{\ka} \tilde{\chi}_k)$.

We define $\widehat{Q^\pm_m u}(\tau,\xi)=\chi_m(\tau\mp \la
\xi\ra)\widehat{u}(\tau,\xi)$, and $\widehat{Q^\pm_{\leq m}
  u}(\tau,\xi)=\chi_{\leq m}(\tau\mp \la
\xi\ra)\widehat{u}(\tau,\xi)$. We also define
$\tilde{Q}^\pm_{m}=Q^\pm_{m-1}+Q^\pm_{m}+Q^\pm_{m+1}$.  We set
$B^\pm_{k,m}$ to be the Fourier support of $Q^\pm_m$, and
$\tilde{B}^\pm_{k,m}$ to be the Fourier support of $\tilde{Q}^\pm_m$.
We define $Q^\pm_{\prec m} =\sum_{l=-\infty}^{m-c}
Q^\pm_l$ for a fixed large integer $c>30$, and $Q^\pm_{\succeq
  m}=I-Q^\pm_{\prec m}$.  Given $k \in \Z$, and $\kappa \in
\mathcal{K}_l$ for some $l \in \N$ we set $B^\pm_{k,\kappa}$ to be the
Fourier-support of $Q^\pm_{\prec k-2l} P_{k,\kappa}$. Similarly we
define $\tilde B^\pm_{k,\kappa}$.

Given a pair $(\lambda,\omega)$ with $\lambda \in \R$ and $\om
=(\omega_1,\omega_2) \in \mathbb{S}^1$, we define
$\omega^\perp=(-\omega_2,\omega_1)$ and the directions
\begin{align*}
  \Theta=\Theta_{\lambda,\omega}=&\frac1{\sqrt{1+\lambda^2}}(\lambda
  ,\omega), \\
  \Theta^\perp=\Theta^{\perp}_{\lambda,\omega} =&
  \frac1{\sqrt{1+\lambda^{2}}}(-1,\lambda\omega), \\
  \Theta_{0,\om^\perp}=&(0,\omega^\perp).
\end{align*}
With respect to this basis, understanding the vectors
$\Theta_{\lambda,\om}$, $\Theta_{\lambda,\om}^\perp$,
$\Theta_{0,\om^\perp}$ as column vectors, we introduce the new
coordinates $t_{\Theta},x_{\Theta}$, with
$x_{\Theta}=(x^1_{\Theta},x^2_{\Theta})$, defined by
\begin{equation} \label{corT}
  \begin{pmatrix}
    t_{\Theta} \\ x^1_{\Th} \\ x^2_{\Th} \\
  \end{pmatrix}
  =\begin{pmatrix} \Theta_{\lambda,\om} & \Theta_{\lambda,\om}^\perp &
    \Theta_{0,\om^\perp}
  \end{pmatrix}^t
  \begin{pmatrix}
    t \\ x_1 \\ x_{2} \\
  \end{pmatrix}
\end{equation}
If $\lambda=1$ we obtain the characteristic directions (null
co-ordinates) as in \cite[p.\ 42]{tat} and \cite[p.\
476]{tao}. However, our analysis requires more flexibility in the
choice of the frames. For fixed $k \in \Z$ we define
$\lambda(k)=(1+2^{-2k})^{-\frac12}$.

\section{Linear estimates}\label{sect:le}
As in \cite[Section 2]{BH}, we recall that the decay rates of
solutions to the linear wave equation and Klein-Gordon equation are determined by the principal curvatures of the
characteristic hypersurfaces. This is well-known and we refer the reader to the list of references provided in \cite[page 47, line 22]{BH}
and the detailed discussion in \cite[Section 2.5]{NaSch}.

In \cite{BH} we started investigating the endpoint Strichartz estimate
for the Dirac and Klein-Gordon equations in dimension $n=3$.  In this paper we continue our investigation in that
direction in dimension $n=2$. This requires a far more
delicate theory since we have to deal with a missing endpoint
Strichartz estimate for the Schr\"odinger part as well.

For convenience, we set $m=1$ in the Klein-Gordon equation \eqref{KG},
which extends by rescaling to \eqref{KG} with any $m \neq 0$. In this
case, the solution is given by
\begin{equation}\label{eq:kg-prop}
  u(t) = \frac12 (e^{it \la D \ra} + e^{-it \la D \ra}) u_0 + \frac1{2i}
  (e^{it \la D \ra} - e^{-it \la D \ra}) \frac{u_1}{\la D \ra}.
\end{equation}
where $\la D \ra$ is the Fourier multiplier with symbol $\la \xi \ra$.
Obviously, we need to study the propagator $e^{\pm it \la D \ra}$. For
the sake of the exposition, we work out all the estimates for $e^{it
  \la D \ra}$, the estimates for $e^{-it \la D \ra}$ are then obtained
by simply reversing time in the estimates for $e^{it \la D \ra}$.

\subsection{Endpoint $L^2L^\infty$ type Strichartz estimate.}\label{endpnt}
Our main result in this subsection provides endpoint Strichartz
estimates for functions localized in frequency. The construction of
the frame systems needed to capture these estimates is time-dependent,
but the constants involved in the estimates are time independent.

We fix $r \in \N$, construct spaces that depend on $r$ and provide
uniform estimates on intervals $[-T,T]$ with $2^{r-1} \leq T \leq 2^r$.  For $k
\leq 99$ and $\omega \in \Sp^1$ we define the set
\[
\Lambda_{k,\omega}=\Big\{ i2^{-r}; i \in \Z, |i|\leq
\frac{2^r}{\sqrt{1+2^{-2k-4}}} \Big\}\times\{\omega\}
\]
and
\[\| \phi \|_{\sum_{\Lambda_{k,\omega}}
  L^2_{t_{\Theta}}L^\infty_{x_{\Theta}}}:= \inf_{\phi = \sum_{{\Theta
      \in \Lambda_{k,\om}}} \phi_\Theta} \sum_{\Theta \in
  \Lambda_{k,\om}} \| \phi_\Theta
\|_{L^2_{t_{\Theta}}L^\infty_{x_{\Theta}}}.
\]
Note that if $k_1 \leq k_2 \leq 99$ then $\Lambda_{k_1,\omega} \subset
\Lambda_{k_2,\omega}$.  One could be more precise about
$\Lambda_{k,\omega}$, but this is not needed for low frequencies.
However it is needed for high frequencies and this motivates the
next definition.

For $k \geq 100$, and $\om \in \Sp^1$ we define
\[
  \Lambda_{k,\om}  =\Big\{\frac1{\sqrt{1+m^{-2}}}; m \in 2^{2k-r-20} \Z\cap[2^{k-3},2^{k+3}]  \Big\} \times\{\omega\}
 \]
if $k < r+20$, while if $k \geq r+20$,
\[
  \Lambda_{k,\om}  =\{ \lambda(k) \} \times\{\omega\}
 \]
 Recall that $\lambda(k)=(1+2^{-2k})^{-\frac12}$. 
We also define
\[
  \Omega_{k,\om}  =\{\lambda(k)\}\times \Big\{ R^{i} \omega; i \in
  \Z, |i| \leq 2^{-k-8+r} \Big\},
\]
where $R$ denotes a rotation by $2^{-r}$. Note that the above set reduces to $\Omega_{k,\om} = \{\lambda(k)\}\times \{ \omega \}$
if $k+8 > r$. These multiscale constructions, corresponding to large families of frames, are  needed in the case $2^k \ls T$;
in the case $T \ls 2^k$, single frames suffice.

For $\ka \in \Ka_{k+10}$, we set
$\Lambda_{k,\ka}:=\Lambda_{k, \om(\ka)}$ and
$\Omega_{k,\ka}:=\Omega_{k,\om(\ka)}$.

Using these sets, we define
\begin{align*}
  \| \phi \|_{\sum_{\Lambda_{k,\kappa}}
    L^2_{t_{\Theta}}L^\infty_{x_{\Theta}}}:=& \inf_{\phi =
    \sum_{\Theta \in \Lambda_{k,\ka}} \phi_\Theta} \sum_{\Theta \in
    \Lambda_{k,\ka}}
  \| \phi_\Theta \|_{L^2_{t_{\Theta}}L^\infty_{x_{\Theta}}}\\
  \| \phi \|_{\sum_{\Omega_{k,\ka}}
    L^2_{x^2_{\Theta}}L^\infty_{(t,x^1)_{\Theta}}}:=& \inf_{\phi =
    \sum_{\Theta\in \Omega_{k,\ka}} \phi_{\Theta}} \sum_{\Theta \in
    \Omega_{k,\ka}} \| \phi_{\Theta}
  \|_{L^2_{x^2_{\Theta}}L^\infty_{(t,x^1)_{\Theta}}}
\end{align*}
We are ready to state the main result containing an effective
replacement structure for the missing endpoint Strichartz estimates.

\begin{thm}\label{thm:estr} Let $r > 0$ and $T \in (0,2^r]$.

  {\rm i)} For all $k \leq 99$, $\omega \in \Sp^1$ and $f \in
  L^2(\R^2)$ with $\supp(\widehat{f})\subset \tilde{A}_{\leq
    k}$,
  \begin{equation} \label{eq:lf} \| 1_{[-T,T]}(t) e^{it \la D \ra} f
    \|_{\sum_{\Lambda_{k,\om}} L^2_{t_{\Theta}}L^\infty_{x_{\Theta}}}
    \ls \| f \|_{L^2},
  \end{equation}
  where the implicit constant does not depend on $r$ and $T$.

  {\rm ii)} For all $k \geq 100$, $\ka \in \mathcal{K}_{k+10}$, and $f
  \in L^2(\R^2)$ with $\supp(\widehat{f}) \subset
  \tilde{A}_{k,\ka}$,
  \begin{equation} \label{Sloc1} \| 1_{[-T,T]}(t) e^{it \la D \ra} f
    \|_{\sum_{\Lambda_{k,\ka}} L^2_{t_{\Theta}}L^\infty_{x_{\Theta}}}
    \ls \| f \|_{L^2},
  \end{equation}
  \begin{equation} \label{Sloc2} \| 1_{[-T,T]}(t) e^{it \la D \ra} f
    \|_{\sum_{\Omega_{k,\ka}}
      L^2_{x^2_{\Theta}}L^\infty_{(t,x^1)_{\Theta}}} \ls 2^{\frac{k}2}
    \| f \|_{L^2},
  \end{equation}
  where the implicit constants do not depend on $r$ and $T$.

  {\rm iii)} For all $k \geq 100$, $1 \leq l \leq k$, $ \ka_1 \in
  \mathcal{K}_l$ and $f \in L^2(\R^2)$ with $\supp(\widehat{f})
  \subset \tilde A_{k, \ka_1}$,
  \begin{equation} \label{MStr} \sum_{\ka \in \mathcal{K}_k} \|
    1_{[-T,T]}(t) e^{it \la D \ra} \tilde P_{\ka} f
    \|_{\sum_{\Lambda_{k,\ka}} L^2_{t_{\Theta}}L^\infty_{x_{\Theta}}}
    \ls 2^{\frac{k-l}2} \|f\|_{L^2}.
  \end{equation}
  where the implicit constant does not depend on $r$ and $T$.
\end{thm}

The estimate \eqref{eq:lf} is similar in nature to the corresponding
estimate in \cite[Lemma 3.4]{bikt}.  We highlight the similarities and
the differences. By changing the variables and using that $|\lambda|
\ls 1$ one passes from the frames used in \cite[Lemma 3.4]{bikt} to
the ones used in this paper. We do not need to discriminate between
the low frequencies and in this sense the estimate as listed here is
suboptimal; one could easily restate it with a factor of
$2^{\frac{k}2}$ for functions that are localized at frequency $\approx
2^k, k \leq 99.$ The range of admissible $\lambda$ is more carefully
tracked here and this is why our version of $\Lambda$ differs from the
one used in \cite[Lemma 3.4]{bikt}.

The rest of this subsection is devoted to the proof of Theorem
\ref{thm:estr}. In order to prove \eqref{eq:lf} we consider the kernel
\begin{equation}\label{Kl_kdef}
  K_{\leq k }(t,x)=\int_{\R^2}e^{ix\cdot\xi}e^{it\la \xi \ra}\tilde{\chi}_{\leq k}^2(|\xi|)\,d\xi,
\end{equation}
for $k \leq 99$. The key estimates about this kernel are:
\begin{equation} \label{k99-1} |K_{\leq k}(t,x)| \ls \la t \ra^{-1},
\end{equation}
\begin{equation} \label{k99-2} |K_{\leq k}(t,x)| \ls_N \la x \ra^{-N},
  \quad |x| \geq \frac1{\sqrt{1+2^{-2k-4}}} |t|.
\end{equation}
Indeed, \eqref{k99-1} is the standard decay rate for the Schr\"odinger
kernel in dimension $2$, which applies here because we truncate at low
frequencies.  \eqref{k99-2} is obtained by using stationary phase type
arguments, taking into account that the critical points of the phase
function $\phi(\xi)= x\cdot \xi + t \la \xi \ra$ are contained inside
the cone $|x| \leq \frac1{\sqrt{1+2^{-2k-2}}} |t|$.

For any $\omega \in \Sp^1$, we obtain the bound
\[
|1_{[-T,T]} K_{\leq k}(t,x)| \ls_N \sum_{\Theta \in
  \Lambda_{k,\omega}} K_{\Theta}(t,x), \quad K_{\Theta}(t,x) = 2^{-r}
\la t_{\Theta} \ra^{-N}.
\]
This is obvious from \eqref{k99-2} in the region of fast decay, and
for fixed $(t,x)$ in the region of slow decay we count the number of
$\Theta$ such that $|t_{\Theta}|\ls 1$: If $|t|\ls 1$, every
$\Theta\in \Lambda_{k,\omega}$ satisfies this, so the sum is of the
size $1$ which is ok in view of \eqref{k99-1}. In the case $|t|\gg 1$,
the number of such $\Theta$ is $\approx 2^r t^{-1}$, so the sum is of
size $\la t\ra^{-1}$, which is again fine because of \eqref{k99-1}.

From the expression of $ K_{\Theta}$ we derive
\begin{equation} \label{k99-3} \sum_{\Theta \in \Lambda_{k,\om}} \|
  K_{\Theta} \|_{L^1_{t_{\Theta}}L^\infty_{x_{\Theta}}} \ls 1.
\end{equation}
This suffices to prove \eqref{eq:lf}. Indeed, by the $TT^*$
argument and the duality:
\[
(\bigcap_{\Theta \in \Lambda_{k,\om}}
L^2_{t_{\Theta}}L^1_{x_{\Theta}})^* = \sum_{\Theta \in
  \Lambda_{k,\om}} L^2_{t_{\Theta}}L^\infty_{x_{\Theta}}
\]
the problem is reduced to proving $\| 1_{[-T,T]} K_{\leq k}
\|_{\sum_{\Lambda_{\leq k,\om}} L^1_{t_{\Theta}}L^\infty_{x_{\Theta}}}
\ls 1$, which follows from \eqref{k99-3}. A more complete
formalization of this type of argument can be found in \cite{bikt}.

We continue the more delicate part of the argument, that is the
analysis in high frequency with the aim of proving \eqref{Sloc1},
\eqref{Sloc2} and \eqref{MStr}.  For $k \in \Z, k \geq 100$ we define
\begin{equation}\label{K_kdef}
  K_{k}(t,x)=\int_{\R^2}e^{ix\cdot\xi}e^{it\la \xi \ra}\tilde{\chi}_{k}^2(|\xi|)\,d\xi.
\end{equation}
and record the decay estimate
\begin{equation}\label{eq:bigk}
  |K_k(t,x)| \ls
  2^{2k} (1+2^{k} |(t,x)|)^{-\frac12} \min(1,(1+2^k|(t,x)|)^{-\frac12}2^k)).
\end{equation}
This estimate appears in many places in literature, see for instance
\cite{NaSch}.  We provided a self-contained proof in \cite{BH} for
dimension $3$ which can be replicated almost verbatim for dimension
$2$ to give \eqref{eq:bigk}.

We define localized versions of the above kernel. For fixed $l\geq 1$
and $\ka \in \Ka_l$ we define:
\begin{equation}\label{A_kdef}
  K_{k,\ka} (t,x)=\int_{\R^2}e^{ix\cdot\xi}e^{ it\la \xi \ra}\tilde{\chi}_k^2(|\xi|)\tilde{\eta}_\kappa(\xi)\,d\xi.
\end{equation}
$K_{k,\ka}$ is the part of $K_k$ localized in the angular cap
$\ka$. Also, we define
\begin{equation}\label{A_k-loc-def}
  K^j_{k,\ka} (t,x)=\int_{\R^2}e^{ix\cdot\xi}e^{ it\la \xi \ra}\alpha_j(2^{-k}|\xi|)\tilde{\chi_k}\tilde{\eta}_\kappa(\xi)\,d\xi,
\end{equation}
where $(\alpha_j)$ is a smooth partition of unity with
$\supp\alpha_j\subset \{(j-1)2^{-20}\leq |\xi|\leq
(j+1)2^{-20}\}$. Obviously, we have
\begin{equation} \label{kkj}
  K_{k,\ka}(t,x)=\sum_{j=2^{18}-1}^{2^{22}+1}K^j_{k,\ka}.
\end{equation}
The important decay properties of $K_{k,\ka}$ and $K_{k,\ka}^j$ are
recorded in the following Proposition.

 \begin{pro}\label{pro:ang} 
   For all $k \in \Z, k \geq 100$, and $\ka \in \mathcal{K}_{k+10}$,
   \begin{equation}\label{eq:ang1}
     |K_{k,\ka}(t,x)|\ls2^{k}(1+ 2^{-k} |(t,x)|)^{-1}.
   \end{equation}
   In addition, for $N=1,2$, we have the following:
   \begin{equation} \label{eq:ang3} |K_{k,\ka}(t,x)| \ls 2^{k}( 1+
     |x^2_{k,\ka}|)^{-N}, \text{if } |x^2_{k,\ka}| \geq 2^{-k-9}
     |(t,x)|,
   \end{equation}
   where $x^2_{k,\ka}=x^2_{\Theta_{\lambda(k),\om(\ka)}}$. For
   $2^{18}-1\leq j \leq 2^{22}+1$,
   \begin{equation}\label{eq:ang4}
     |K_{k,\ka}^j(t,x)| \ls 2^{k}( 1+ 2^{k}|t_{\lambda_k^j,\ka}|)^{-N}, \text{if } |t_{\lambda_k^j,\ka}| \geq 2^{-2k-8} |t|,
   \end{equation}
   where $\lambda_k^j=1/\sqrt{1+2^{-2k+40}j^{-2}}$ and
   $t_{\lambda_k^j,\ka}=t_{\Theta_{\lambda^j_k,\om(\ka)}}$.
 \end{pro}

 We remark that \eqref{eq:ang3}-\eqref{eq:ang4} hold with any $N \in
 \N$, but as stated it suffices for our purposes.  Ideally one would
 like to have the estimate \eqref{eq:ang4} for $K_{k,\ka}$ is a
 similar form to \eqref{eq:ang3} and skip the cumbersome $K_{k,\ka}^j$
 kernels. While available, such a formulation is not able to provide a
 strong exponent as above, see the factor $2^{-2k-8}$ in
 \eqref{eq:ang4}, and this would impact a key property of the set
 $\Lambda_{k,\ka}$.

 We now show how \eqref{Sloc1} follows from the above result. Fix $j
 \in [2^{18}-1, 2^{22}+1] \cap \Z$ and define
 \[
 \Lambda_{k,\kappa}^j =\Big\{\frac1{\sqrt{1+m^{-2}}}; m \in 2^{2k-r-20}
 \Z\cap[ (j-1)2^{k-20},(j+1)2^{k-20}] \Big\}
 \times\{\omega(\kappa)\}
 \]
 We first make a few observations. The cardinality of each $ \Lambda_{k,\kappa}^j$ is $2^{r-k+20} \approx 2^{-k}T$.
 This complicated construction is needed only for a certain range of frequencies: $2^k \ls T \approx 2^r$. 
  If $k \geq r+20$, then we simply use a single set 
 \[
 \Lambda_{k,\kappa} =\Big\{\frac1{\sqrt{1+2^{-2k}}} \Big\}
 \times\{\omega(\kappa)\},
 \]
 and the arguments below simplify considerably: the claim \eqref{inde} follows from Proposition \ref{pro:ang}
 and the rest of the argument is identical. 
 
 Thus we focus below on the case $k \leq r+20$.  For each $\Theta \in \Lambda_{k,\ka}^j$ we define
 \[
 K_{\Theta}(t,x)= 2^{2k} T^{-1} (1+2^{k} |t_{\Theta}|)^{-2}
 \]
 and claim that
 \begin{equation} \label{inde} |K^j_{k,\ka}(t,x)| \ls \sum_{\Theta \in
     \Lambda^j_{k,\ka}} K_{\Theta} (t,x).
 \end{equation}
 Since
 \[
   \sum_{\Theta \in \Lambda^j_{k,\ka}} \| K_\Theta
   \|_{L^1_{t_{\Theta}}L^\infty_{x_{\Theta}}}  \ls
   |\Lambda^j_{k,\ka}| \sup_{\Theta \in \Lambda^j_{k,\ka}} \| K_\Theta \|_{L^1_{t_{\Theta}}L^\infty_{x_{\Theta}}} 
    \ls 2^{-k}T \cdot 2^{2k} T^{-1} 2^{-k}\ls 1.
 \] 
 we conclude with
 \[
 \| K^j_{k,\ka} \|_{\sum_{\Lambda^j_{k,\ka}}
   L^1_{t_{\Theta}}L^\infty_{x_{\Theta}} } \ls 1.
 \]
 By noting that $\Lambda_{k,\ka}=\cup_j \Lambda^j_{k,\ka}$, using
 \eqref{kkj} and the fact that $j$ runs in a finite set, we obtain
 \[
 \| K_{k,\ka} \|_{\sum_{\Lambda_{k,\ka}}
   L^1_{t_{\Theta}}L^\infty_{x_{\Theta}} } \ls 1.
 \]
 which implies \eqref{Sloc1} by a $TT^*$ argument similar to the one
 we used in the proof of \eqref{eq:lf}.

 We continue with the argument for \eqref{inde}. We start with a few
 observations, which in fact were the basis for the construction of
 the set $\Lambda^j_{k,\ka}$:

 \bf P1: \rm If $|t_{\lambda_k^j,\ka}| \leq 2^{-2k-2}|(t,x)|$ then
 there exists $\Theta \in \Lambda^j_{k,\ka}$ such that $|t_\Theta|
 \leq 2^{-k+2}$.

 \bf P2: \rm If $|t_{\lambda_k^j,\ka}| \geq 2^{-2k-2}|(t,x)|$ then
 $|t_{\lambda_k^j,\ka}| \gs |t_{\Theta}|$, for all $\Theta \in
 \Lambda^j_{k,\ka}$.

 As a first case, let $(t,x)$ be such that $|t_{\lambda_k^j,\ka} |
 \leq 2^{-2k-2}|(t,x)|$.  From {\bf P1} it follows that for each such
 $(t,x)$ we estimate the number of $\Theta \in \Lambda^j_{k,\ka}$ such
 that $|t_{\Theta}| \leq 2^{-k+2}$. If $\Theta_0=(\lambda_0,\omega)$
 is such a value, then any other such $\Theta=(\lambda,\omega)$ should
 satisfy $|(\lambda - \lambda_0) t| \leq 2^{-k+3}$. There are two
 subcases to consider next:

 If $|t| \leq 2^{k}$, then since all $\Theta=(\lambda,\om) \in
 \Lambda_{k,\ka}^j$ satisfy $|\lambda - \lambda_0| \leq 2^{-2k+6}$ it
 follows that $|(\lambda - \lambda_0) t| \leq 2^{-k+6}$, hence the
 number of such $\Theta $ is $|\Lambda_{k,\ka}|=2^{-k}T$. Thus the sum
 on the right of \eqref{inde} is estimated by $|\Lambda_{k,\ka}| \cdot
 2^{2k} T^{-1} =2^{k}$ and this is the bound we have for the kernel
 $K_{k,\ka}$.

 If $|t| \geq 2^{k}$, we use that the discretization in
 $\Lambda_{k,\ka}^j$ is at scale $2^{-k} T^{-1}$, it follows that the
 number of such $\lambda$ is given by $\approx
 \frac{2^{-k}t^{-1}}{2^{-k}T^{-1}}= t^{-1} T$. The sum on the right of
 \eqref{inde} is then $\gs 2^{2k} T^{-1} t^{-1} T=2^{2k} t^{-1}$ which
 is precisely the bound we have for the kernel $K_{k,\ka}$.

 Next we consider the second case where $|t_{k,\ka} | \geq
 2^{-2k-2}|(t,x)|$. We use \bf P2 \rm: $|t_{k,\ka}| \gs |t_{\Theta}|$,
 for all $\Theta \in \Lambda^j_{k,\ka}$.  Thus $(1+2^{k}
 |t_{\Theta}|)^{-2} \gs (1+2^{k} |t_{k,\ka}|)^{-2}$ and the right hand
 side of \eqref{inde} is $\gs |\Lambda_{k,\ka}^j| \cdot 2^{2k} T^{-1}
 \cdot (1+2^{k} |t_{k,\ka}|)^{-2}=2^{k} (1+2^{k} |t_{k,\ka}|)^{-2}$
 and this is the bound we have on $K^j_{k,\ka}$ from
 \eqref{eq:ang4}. This finishes the proof of \eqref{Sloc1}.

 A similar argument using \eqref{eq:ang3} proves \eqref{Sloc2}. Note
 that the construction of the set $\Omega_{k,\ka}$ was designed
 precisely to fit the corresponding \bf P1 \rm and \bf P2 \rm in this
 context: the angles considered in $\Omega_{k,\ka}$ cover a
 neighborhood of $\om(\ka)$ size $2^{-k-8}$ which is double the size
 of the slow decay neighborhood described by \eqref{eq:ang3}.
 
 Next we show how \eqref{MStr} follows from \eqref{Sloc1}. Since there
 are $\approx 2^{k-l}$ caps $\ka \in \mathcal{K}_{k}$ such that $P_\ka
 f \ne 0$, we obtain from \eqref{Sloc1}
 \begin{align*}
   & \sum_{\ka \in \mathcal{K}_k} \| 1_{[-T,T]}(t) e^{it \la D \ra}
   \tilde P_{\ka} f \|_{\sum_{\Lambda_{k,\ka}}
     L^2_{t_{\Theta}}L^\infty_{x_{\Theta}}} \\
   & \ls{} 2^{\frac{k-l}2}
   \left( \sum_{\ka \in \mathcal{K}_{k}} \|e^{it \la D \ra}  \tilde P_{\kappa}f \|_{\sum_{\Lambda_{k,\ka}} L^2_{t_{\Theta}}L^\infty_{x_{\Theta}}} ^2 \right)^\frac12 \\
   & \ls{} 2^{\frac{k-l}2} \left( \sum_{\ka \in \mathcal{K}_{k}}
     \|\tilde P_{\kappa}f\|_{L^2_x}^2 \right)^\frac12 \ls{}
   2^{\frac{k-l}2} \| f \|_{L^2_x}.
 \end{align*}

 We end this section with the proof of \eqref{MStr}.
 
 \begin{proof}[Proof of Proposition \ref{pro:ang}] The following proof
   is very similar to \cite{BH}. We begin with the proof of
   \eqref{eq:ang1}. If $|(t,x)| \ls 2^k$ the claim follows from the
   fact that the domain of integration has measure $\approx
   2^{2k-l}\approx 2^k$, otherwise the estimate follows from
   \eqref{eq:bigk} and Young's inequality.

   Next, we turn to the proof of \eqref{eq:ang4}. For compactness of
   notation, we write $\lambda=\lambda_k^j$ and
   $\Th=\Th_{\lambda_k^j,\omega(\ka)}$. By rescaling it suffices to
   consider
   \[
   B_{k,\ka}^j(s,y):=\int_{\R^2}e^{iy\cdot \xi + is\la \xi
     \ra_k}\zeta_j(\xi)d\xi,
   \]
   for
   $\zeta_j(\xi)=\alpha_j(|\xi|)\tilde{\chi}_1^2(|\xi|)\tilde{\eta}_\ka(\xi)$,
   and to prove
   \begin{equation} \label{Best} |B_{k,\ka}^j(s,y)| \ls_N
     2^{-k}(1+|s_{\Th}|)^{-N} \text{ if } |s_{\Th}| \geq 2^{-2k-8} |s|
   \end{equation} 
   for $N=1,2$.  If $|s_{\Th}| \ls 1$, the estimate follows from the
   fact that the support of $\zeta_j$ has measure $\approx
   2^{-k}$. Now, we assume $|s_{\Th}| \gg 1$ and write $\phi(s,y,\xi)=
   y \cdot \xi + s \la \xi \ra_k$ Define $\partial_\omega =\omega\cdot
   \nabla_\xi$, $ d_{\phi,\omega} :=\frac{1}{i\partial_\omega
     \phi}\partial_\omega $ and $ d_{\phi,\omega}^\ast
   :=-\partial_\omega \Big(\frac{\cdot}{i \partial_\omega
     \phi}\Big)$. Integration by parts implies
   \begin{equation}\label{eq:ip}
     \int_{\R^2}e^{i\phi(s,y,\xi)}\zeta_j(\xi)d\xi
     =  \int_{\R^2} e^{i\phi(s,y,\xi)} (d^*_{\phi,\omega})^N \zeta_j(\xi)) d\xi.
   \end{equation}
   We will prove
   \begin{equation}\label{eq:diffop}
     |(d^*_{\phi,\om})^N (\zeta_j)(\xi)|\ls_N |s_{\Th}|^{-N}, \qquad N=1,2,
   \end{equation}
   so that \eqref{Best} follows from \eqref{eq:ip} and
   \eqref{eq:diffop}. Indeed, we observe that
   \[
   \partial_\om \phi(s,y,\xi)=s_{\lambda,\om}+s\Big(\frac{\xi\cdot
     \om}{\la \xi \ra_k}-\lambda\Big),
   \]
   and in the domain of integration we have
   \begin{align*}
     &\Big| \frac{\xi\cdot \om}{\la \xi \ra_k}-\lambda\Big|\leq \Big| \frac{1}{\sqrt{1+2^{-2k}|\xi|^{-2}}}-\lambda\Big|+\Big|\cos(\angle(\hat{\xi},\om))-1\Big|\\
     &\leq 2^{-2k-10}+2^{-2k-10}\leq 2^{-2k-9},
   \end{align*}
   where we use that $(j-1)2^{-20}\leq |\xi|\leq (j+1)2^{-20}$ and
   $|\angle(\xi,\om))|\leq 2^{-k-10}$.  This implies
   \[
   |\partial_\om \phi(s,y,\xi)|\geq |s_{\Th}|-|s|2^{-2k-9}\geq 2^{-1}
   |s_{\Th}|.
   \]
   In particular it follows that
   \begin{equation} \label{part1} |\frac{\partial_\omega
       \zeta}{\partial_\omega \phi}| \ls |s_{\Th}|^{-1}.
   \end{equation}
   where we used that $|\partial_\omega \zeta| \ls 1$. In addition, we
   have
   \[
   \partial_\omega^2 \phi(\xi) = \partial_\omega \Big(s \frac{\omega
     \cdot \xi}{\la \xi \ra_k}\Big) = s \left( \frac{\omega \cdot
       \omega}{\la \xi \ra_k} - \frac{(\omega \cdot \xi)^2}{\la \xi
       \ra_k^3}\right) = \frac{s}{\la \xi \ra_k} \left(1- (
     \frac{\omega \cdot \xi}{\la \xi \ra_k})^2 \right)
   \]
   from which, using the above arguments, we conclude that in the
   domain of integration we have $| \partial_\omega^2 \phi | \ls
   2^{-2k} |s|$.  This allows us to estimate
   \[
   |\partial_\omega \Big(\frac1{\partial_\omega\phi}\Big)| \ls
   \frac{2^{-2k}|s|}{|\partial_\omega \phi|^2} \ls
   \frac{2^{-2k}|s|}{|s_\Th|^2} \ls |s_\Th|^{-1}.
   \]
   From this and \eqref{part1} we obtain \eqref{Best} for $N=1$.  Now
   let $N=2$ and compute
   \begin{align*}
     (d^*_{\phi,\om^\perp})^2 \zeta
     =\partial_\omega\Big(\frac{1}{\partial_\omega \phi
     }\partial_\omega\frac{\zeta}{\partial_\omega \phi}\Big)=
     \frac{\partial_\omega^2 \zeta}{(\partial_\omega \phi)^2}-3
     \frac{\partial_\omega \zeta \partial_\omega^2
       \phi}{(\partial_\omega \phi)^3}
     -\frac{\zeta \partial_\omega^3\phi}{(\partial_\omega \phi)^3}
     +3\frac{\zeta (\partial_\omega^2\phi)^2}{(\partial_\omega
       \phi)^4}
   \end{align*}
   We compute
   \[
   \partial_\omega^3 \phi =\frac{3s}{\la \xi\ra_k^5}\Big((\omega \cdot
   \xi)^3-(\omega \cdot \xi) \la
   \xi\ra_k^2\Big)=\mathcal{O}(2^{-2k})|s|.
   \]
   Recalling that $|\partial_\omega\phi | \geq \frac12 |s_\Th| \gg
   2^{-2k}$, $|\partial_\omega^2 \phi|\ls 2^{-2k}$ and
   $|\partial_\omega^N \zeta|\ls_N 1$ we conclude that
   \[
   |(d^*_{\phi,\om^\perp})^N | \ls |s_\Th|^{-2}+ 2^{-2k}|s_{\Th}|^{-3}
   + 2^{-4k} |s_{\Th}|^{-4} \ls |s_{\Th}|^{-2}.
   \]
   This finishes the proof of \eqref{eq:diffop} and, in turn, the
   proof of \eqref{eq:ang4}.

   It remains to prove \eqref{eq:ang3}. We reset the definition of
   $\Th$ to $\Th=\Theta_{\lambda(k),\om(\ka)}$.  As above, by
   rescaling it suffices to prove
   \begin{equation} \label{Best3} |B_{k,\ka}(s,y)| \ls_N
     2^{-k}(1+2^{-k}|y^2_\Th|)^{-N} \text{ if } |y^2_{\Th}| \geq
     2^{-k-8} |(s,y)|
   \end{equation} 
   for $N=1,2$, where we recall that $y^2_\Th=y\cdot \omega^\perp$.
   If $|y^2_\Th| \ls 2^k$, the estimate follows from the fact
   that the size of the support of integration is $\ls 2^{-k}$.
  
   We now consider the case $|y^2_\Th| \gg 2^k$. By replacing $\om$
   with $\om^\perp$ in the above argument (see \eqref{eq:ip}), we
   obtain
   \begin{equation}\label{eq:ip3}
     \int_{\R^2}e^{i \phi(s,y,\xi)}\zeta(\xi)d\xi
     =  \int_{\R^2} e^{i\beta
       \phi(s,y,\xi)} (d^*_{\phi,\omega^\perp})^N \zeta(\xi)) d\xi.
   \end{equation}
   As above, we claim
   \begin{equation}\label{eq:diffop3}
     |(d^*_{\phi,\om^\perp})^N (\zeta)(\xi)|\ls_N \Big(2^{-k} |y^2_\Th|\Big)^{-N}, \qquad N=1,2.
   \end{equation}
   Since the support of $\zeta$ has measure $\approx 2^{-k}$,
   \eqref{Best3} follows from \eqref{eq:ip3} and \eqref{eq:diffop3}.

   We conclude the proof with the argument for \eqref{eq:diffop3}. If
   $\xi$ in the support of $\zeta$ then
   \[
   \frac{\xi}{|\xi|} = (1-c_1) \omega + c_2 \omega^\perp, \qquad |c_1|
   \leq 2^{-2k-18}, |c_2| \leq 2^{-k-10}
   \]
   and
   \[
   |\frac{|\xi|}{\la \xi \ra_k}- \lambda| \leq 2^{-2k+4},
   \]
   We compute
   \begin{equation*}
     \begin{split}
       \partial_{\omega^\perp} \phi & = \omega^\perp \cdot (y + s
       \frac{\xi}{\la \xi \ra_k}) = y^2_\Th + c_2 s \lambda + c_2 s
       (\frac{|\xi|}{\la \xi \ra_k}-\lambda).
     \end{split}
   \end{equation*}
   We have $|y^2_\Th| \geq 2^{-k-9} |s| \geq 2 |c_2 s \lambda |$, as
   well as $|y^2_\Th| \geq 2^{-k-9}|(y,s)| \gg |c_2 s
   (\frac{|\xi|}{\la \xi \ra_k}-\lambda)|$. From these we conclude
   \begin{equation}\label{eq:der-phi}
     |\partial_{\omega^\perp}\phi | \gs |y^2_\Th| \gg 2^{k}
   \end{equation}
   and, using $|\partial_{\omega^\perp} \zeta| \ls 2^k$,
   \begin{equation} \label{der1} |\frac{\partial_{\omega^\perp}
       \zeta}{\partial_{\omega^\perp} \phi}| \ls 2^k |y^2_\Th|^{-1}.
   \end{equation}
   In addition, we have
   \[
   \partial_{\omega^\perp}^2 \phi(\xi) = \partial_{\omega^\perp}
   \Big(r \frac{\omega^\perp \cdot \xi}{\la \xi \ra_k}\Big) = s \left(
     \frac{\omega^\perp \cdot \omega^\perp}{\la \xi \ra_k} -
     \frac{(\omega^\perp \cdot \xi)^2}{\la \xi \ra_k^3}\right) =
   s(1+\mathcal{O}(2^{-k}))
   \]
   within the support of $\zeta$ and we conclude
   \[
   |\partial_{\omega^\perp}
   \Big(\frac1{\partial_{\omega^\perp}\phi}\Big)| \ls
   \frac{|s|}{|\partial_{\omega^\perp} \phi|^2} \ls
   \frac{|s|}{|y^2_\Th|^2} \ls 2^k |y^2_\Th|^{-1}.
   \]
   From this and \eqref{der1} we obtain \eqref{eq:diffop3} for $N=1$.
   Now we consider the case $N=2$ and compute
   \begin{align*}
     (d^*_{\phi,\omega^\perp})^2 \zeta =
     \frac{\partial_{\omega^\perp}^2 \zeta}{(\partial_{\omega^\perp}
       \phi)^2}-3 \frac{\partial_{\omega^\perp}
       \zeta \partial_{\omega^\perp}^2 \phi}{(\partial_{\omega^\perp}
       \phi)^3}
     -\frac{\zeta \partial_{\omega^\perp}^3\phi}{(\partial_{\omega^\perp}
       \phi)^3} +3\frac{\zeta
       (\partial_{\omega^\perp}^2\phi)^2}{(\partial_{\omega^\perp}
       \phi)^4}.
   \end{align*}
   Further,
   \[
   \partial_{\omega^\perp}^3 \phi =\frac{3s}{\la
     \xi\ra_k^5}\Big(({\omega^\perp} \cdot \xi)^3-({\omega^\perp}
   \cdot \xi) \la \xi\ra_k^2\Big)=s \mathcal{O}(2^{-k}).
   \]
   From \eqref{eq:der-phi} and $|\partial_{\omega^\perp}^2 \phi|\ls
   |s|$ and $|\partial_{\omega^\perp}^N \zeta|\ls_N 2^{kN}$ it follows
   that
   \[
   |(d^*_{\phi,\om^\perp})^2 | \ls 2^{2k} |y^2_{\Th}|^{-2}+ 2^{k}
   |y^2_{\Th}|^{-3} + 2^{-k} |y^2_{\Th}|^{-3}+ |y^2_{\Th}|^{-4} \ls
   2^{2k} |y^2_{\Th}|^{-2},
   \]
   which completes the proof of \eqref{eq:diffop3} for $N=2$.
 \end{proof}

\subsection{Energy estimates in the $(\lambda,\omega)$
   frames}\label{Energy}
 Next, we prove energy estimates similar to \cite[Subsection 2.2]{BH},
 but there will be important differences which we will point out
 below.  At the end of the notation section we have introduced frames
 adapted to a pair $(\lambda,\omega)$ with $\lambda \in \R$ and $\om
 \in \mathbb{S}^1$ and the new coordinates $t_{\Theta},x_{\Theta}$.
%
 We denote by $(\tau_{\Th}, \xi_{\Th})$ the corresponding Fourier
 variables which are given by
 \[
 \begin{pmatrix}
   \tau_{\Th} \\ \xi^1_{\Th} \\ \xi^2_{\Th} \\
 \end{pmatrix}
 =\begin{pmatrix} \Theta_{\lambda,\om} & \Theta_{\lambda,\om}^\perp &
   \Theta_{0,\om^\perp}
 \end{pmatrix}
 \begin{pmatrix}
   \tau \\ \xi_1 \\ \xi_{2} \\
 \end{pmatrix}
 \]
We also
 introduce here a fourth vector $\Theta^-=\Theta_{\lambda,-\om}$ for
 reasons which will become apparent in the proof of the Theorem below.
 In the following theorem we set $B_{k,\ka}=B^+_{k,\ka}$ and
 $\tilde{B}_{k,\ka}=\tilde{B}^+_{k,\ka}$.
 
 \begin{thm} \label{thm:Energy} a) Let $99 \leq m=\min(j,k)$, $0 \leq
   l \leq m+10$ and $\ka \in \mathcal{K}_l$.  Let $\Theta
   =\Theta_{\lambda,\om} \in \Lambda_{j,\om}$.  Assume
   $\alpha=\dist(\om, \ka)$ satisfies $2^{-3-l} \leq \alpha \leq
   2^{3-l}$ for $l \leq m+9$ and $\alpha \leq 2^{3-l}$ for $l=m+10$;
   if $j = 99$ then we consider only the last case. Define $\tilde
   \alpha=\max(\alpha,2^{-m})$.

   {\rm i)} If $f \in L^2(\R^2)$ has the property that $\hat f$ is
   supported in $A_{k,\ka}$, the following holds true
   \begin{equation} \label{DH} \tilde \alpha \| e^{ it \la D \ra}
     f\|_{L^\infty_{t_{\Theta}}L^2_{x_{\Theta}}} \ls \| f \|_{L^2},
   \end{equation}
   provided that $l \leq m-10$ or $ l=m+10 \wedge |j-k| \geq 10$, and
   \begin{equation} \label{DH2} \alpha^\frac12 \| e^{ it \la D \ra}
     f\|_{L^\infty_{x^2_{\Theta}}L^2_{(t,x^1)_{\Theta}}} \ls \| f
     \|_{L^2}, \quad l \leq m+9.
   \end{equation}

   {\rm ii)} Consider the inhomogeneous equation
   \begin{equation} \label{inheq} (i \partial_t + \la D \ra) u = g,
     \quad u(0)=0,
   \end{equation}
   where $\hat g$ is assumed to be supported in the set
   $B_{k,\ka}$. If $g \in L^1_{t_{\Theta}}L^2_{x_{\Theta}}$, then the
   solution $u$ satisfies the estimate
   \begin{equation} \label{DIH} \tilde \alpha \| u
     \|_{L^\infty_{t_{\Theta}}L^2_{x_{\Theta}}} \ls \tilde \alpha^{-1}
     \| g \|_{L^1_{t_{\Theta}}L^2_{x_{\Theta}}},
   \end{equation}
   provided that $l \leq m-10$ or $ l=m+10 \wedge |j-k| \geq 10$.
   
   If $g \in L^1_{x^2_{\Theta}}L^2_{(t,x^1)_{\Theta}} $, then the
   solution $u$ satisfies the estimate
   \begin{equation} \label{DIH2} \alpha^\frac12 \| u
     \|_{L^\infty_{x^2_{\Theta}}L^2_{(t,x^1)_{\Theta}}} \ls
     \alpha^{-\frac12} \| g
     \|_{L^1_{x^2_{\Theta}}L^2_{(t,x^1)_{\Theta}}}, \quad l \leq m+9.
   \end{equation}

   {\rm iii)} Under the hypothesis of Part ii) when $g \in
   L^1_{t_{\Theta}}L^2_{x_{\Theta}}$ the solution $u$ can be written
   as
   \begin{equation}
     u(t) =e^{ it \la D \ra} \tilde v_0
     + \int_{-\infty}^\infty u_{s}(t) \chi_{t_{\Theta} > s} ds
   \end{equation}
   where $u_s(t)= e^{ it \la D \ra} v_s$ (homogeneous solution in the
   original coordinates) and
   \begin{equation}
     \| \tilde v_0 \|_{L^2_x} + \int_{-\infty}^\infty \| v_s \|_{L^2_x} ds
     \ls \al^{-1} \| g \|_{L^1_{t_{\Theta}}L^2_{x_{\Theta}}}. 
   \end{equation}
   In addition $\hat v_s$ and $\hat{\tilde v}_0$ are supported in
   $\tilde A_{k,\ka}$.
     
   A similar statement holds true when $g \in
   L^1_{x^2_{\Theta}}L^2_{(t,x^1)_{\Theta}} $.
 \end{thm}

 A few remarks are in place about the statement of the above
 theorem. First, the statement \eqref{DH2} and the corresponding ones
 in part ii) and iii) hold true for all $\alpha$ with $2^{-3-l} \leq
 \alpha \leq 2^{3-l}$, in the sense that we do not need to restrict to
 $l \leq m+9$. The reason we did so in the statement is for the sake
 of conciseness. Nevertheless the statement \eqref{DH} for $l=m+10$
 does not require angular separation, thus covering the ranges skipped
 by the way we state \eqref{DH2}.
   
 What is important to note is that \eqref{DH} fails somewhere in the
 range $m-9 \leq l \leq m+9$ in the sense that the energy estimates in
 the given frames "blow-up" and become useless. This is precisely the
 region where we need to use the estimates \eqref{DH2}.
   
 A careful reading reveals that in the case $|j-k| \leq 9$, and $l =
 m+10$ we did not provide any estimates.  As noted above, one can
 continue estimates of type \eqref{DH2} and \eqref{DIH2} for $l \geq
 m+10$, but these will not be helpful for our purposes.

 \begin{proof} i) \bf Proof of \eqref{DH}. \rm We start with an almost verbatim repetition from \cite[Proof of Theorem 2.4]{BH}: The space-time Fourier
   of $w(t,x)=e^{ i t \la D \ra} f(x)$ is given by the distribution
   $\F w= \hat f d \sigma$ where $d \sigma(\tau,\xi) = \delta_{\tau=
     \sqrt{|\xi|^2+1}}$ is comparable with the standard measure on the
   surface $\tau= \sqrt{|\xi|^2+1}$. We change the variables
   $(\tau,\xi) \rightarrow (\tau_{\Th},\xi_{\Th})$ and rewrite $\hat f
   d \sigma = F \delta_{\tau_{\Th}=h(\xi_{\Th})}$; thus
   \begin{equation} \label{Fest} \| F \|_{L^2_{\xi_{\Th}}} \ls
     (1+\|\nabla h \|_{L^\infty})^\frac12 \| f \|_{L^2}
   \end{equation}
   where the $L^\infty$ norms is taken on the support of $F$.

   We now work out the details. The equation of the characteristic
   surface $\tau= \sqrt{|\xi|^2+1}$ can be rewritten as
   $\tau^2-|\xi|^2-1=0$. In the new frame this takes the form
   \[
   \frac1{\lambda^2+1}(\lambda \tau_{\Th} - \xi^1_{\Th})^2 -
   \frac1{\lambda^2+1} (\tau_{\Th} + \lambda \xi^1_{\Th})^2 -
   |\xi^2_{\Th}|^2 -1 =0.
   \]
   We solve this equation for $\tau_{\Th}$, hence we rewrite it as
   follows
   \begin{equation} \label{quadeq} \frac{\lambda^2-1}{\lambda^2+1}
     (\tau_{\Th})^2 - \frac{4 \lambda}{\lambda^2+1} \tau_{\Th}
     \xi^1_{\Th} + \frac{1-\lambda^2}{\lambda^2+1} (\xi^1_{\Th})^2 -
     |\xi^2_{\Th}|^2 -1 =0.
   \end{equation}
   The solutions of this quadratic equation are given by
   \begin{equation} \label{root}\begin{split} \tau_{\Th}=
       h^\pm(\xi_{\Th}) = \frac{2 \lambda \xi^1_{\Th} \pm
         \sqrt{(\lambda^2+1)^2
           (\xi^1_{\Th})^2+(\lambda^4-1)(|\xi^2_{\Th}|^2+1)}}{\lambda^2-1}.
     \end{split}
   \end{equation}
   We will identify which one of the two solutions is the correct
   one. The positivity of the discriminant
   $\Delta_{\Th}=(\lambda^2+1)^2
   (\xi^1_{\Th})^2+(\lambda^4-1)(|\xi^2_{\Th}|^2+1)$ is implicit, as
   we know a priori that \eqref{quadeq} has at least one solution. We
   will come back shortly to these issues. We continue with the
   following computation:
   \[
   \begin{split}
     \frac{\partial h^\pm}{\partial \xi^1_{\Th}}
     & =  \frac{1}{\lambda^2-1}(2\lambda+\frac{(\lambda^2+1)^2\xi^1_{\Th}}{\pm \sqrt{(\lambda^2+1)^2 (\xi^1_{\Th})^2+(\lambda^4-1)(|\xi^2_{\Th}|^2+1)}})\\
     & = \frac{1}{\lambda^2-1}(2\lambda+ \frac{(\lambda^2+1)^2\xi^1_{\Th}}{(\lambda^2-1) \tau_{\Th} - 2 \lambda \xi^1_{\Th}}) \\
     & = \frac{2\lambda \tau_{\Th} + (\lambda^2-1)\xi^1_{\Th}}{(\lambda^2-1) \tau_{\Th} - 2 \lambda \xi^1_{\Th}}
      = - \frac{\xi^1_{\Th^-}}{\tau_{\Th^-}}
   \end{split}
   \]
   In a similar manner we obtain $ \nabla_{\xi^2_\Th} h^\pm =
   (\lambda^2+1) \frac{\xi^2_{\Th}}{\tau_{\Th^-}}$, from which, using
   \eqref{Fest}, it follows
   \begin{equation} \label{inter} \| e^{ it \la D \ra}
     f\|_{L^\infty_{t_{\Th}}L^2_{x_{\Th}}} \ls \left( 1+ \sup_{\xi \in
         A_{k,\ka}} \frac{2^k}{|\tau_{\Th^-}|} \right)^\frac12 \| f
     \|_{L^2}.
   \end{equation}

   To finish the argument we need a lower bound for
   $|\tau_{\Th^-}|$. We provide below lower bounds for $\Delta_{\Th}$
   and $\tau_{\Th^-}$ for $(\tau,\xi) \in B_{k,\ka}$, as these more
   general bounds are needed in Part ii).

   We need to consider a few cases: $j \leq k-10$, $|j-k| \leq 9$ and
   $j \geq k+10$.  Since the computations are entirely similar, we
   will deal with $j \leq k-10$ in detail.  Here we have to consider
   two more cases: $l \leq j-10$ and $l=j+10$.
   
   \bf Case 1: \rm $l \leq j-10$.  For $(\tau,\xi) \in B_{k,\ka}$ it
   holds that $\tau - \sqrt{|\xi|^2+1} = \epsilon(\tau,\xi)$ with
   $|\epsilon(\tau,\xi)| \leq 2^{k-2l-10}$, hence
   \[
   \begin{split}
     \tau_{\Th^-} = & \lambda \tau - \xi \cdot \omega = \lambda
     \sqrt{|\xi|^2+1} + \lambda \epsilon - \xi \cdot \omega \\
     = & |\xi| \left( (\lambda-1) \sqrt{1+|\xi|^{-2}} +
       \sqrt{1+|\xi|^{-2}} -1+1 - \frac{\xi \cdot \omega}{|\xi|} +
       \frac{\lambda \epsilon}{|\xi|} \right)
   \end{split}
   \]
   We have the following: $|(1-\lambda) \sqrt{1+|\xi|^{-2}}| \leq
   2(1-\lambda) \leq 2^{-2j+6} \leq 2^{-2l-12}$ (since $\lambda \in
   \Lambda_j$), $|\sqrt{1+|\xi|^{-2}}- 1| \leq 2^{-2j-12} \leq
   2^{-2l-20}$, $ 2^{-2l-6} \leq 1- \frac{\xi \cdot \omega}{|\xi|}
   \leq 2^{-2l+6}$ and $|\frac{\lambda \epsilon}{|\xi|} | \leq
   2^{-2l-8}$. From these we conclude that $2^{k-2l-10} \leq
   \tau_{\Th^-} \leq 2^{k-2l+10}$; thus we conclude that $\tau_{\Th^-}
   \approx 2^k \alpha^2$ and $\tau_{\Th^-} \geq 2^{k-20} \alpha^2$.

   In particular, using \eqref{inter} we obtain \eqref{DH}. Since the
   solutions in \eqref{root} can be recast in the form
   $\tau_{\Th^-}=\pm \sqrt{\Delta_{\Th}}$ and we just proved that $
   \tau_{\Th^-} > 0$ in $B_{k,\ka}$, it follows that the solutions
   $h^+$ in \eqref{root} correspond to the choice of the surface $\tau
   = \sqrt{|\xi|^2+1}$.

   We now continue with the more general bounds for $\Delta_{\Th}$ in
   the set $B_{k,\ka}$.  Since $|\tau - \la \xi \ra| \leq
   2^{k-2l-10}$, it follows that $ |\tau^2 - |\xi|^2 - 1| \leq
   2^{2k-2l-8}$ or equivalently, $\tau^2 - |\xi|^2 -
   1=\epsilon(\tau,\xi)$ with $|\epsilon(\tau,\xi)| \leq 2^{2k-2l-8}$.
   We rewrite the equation in characteristic coordinates as above, to
   obtain
   \[
   \tau_{\Th^-}^2= \Delta_{\Th} + (1-\lambda^4)\epsilon
   \]
   We have already shown that $\tau_{\Th^-} \geq 2^{k-2l-10}$ and
   since $|(1-\lambda^4)\epsilon| \leq 2^{2k-2l-8} |1-\lambda| \leq
   2^{2k-2l-8} 2^{-2j+5} \leq 2^{2k-4l-23}$, it follows that
   $\Delta_{\Th} \geq 2^{2k-4l-22} \approx 2^{2k} \alpha^4$ in
   $B_{k,\ka}$. A similar argument proves $\Delta_{\Th} \approx 2^{2k}
   \alpha^4$ in $B_{k,\ka}$.

   \bf Case 2: $l=j+10$ \rm. For $(\tau,\xi) \in B_{k,\ka}$ it holds
   that $\tau - \sqrt{|\xi|^2+1} = \epsilon(\tau,\xi)$ with
   $|\epsilon(\tau,\xi)| \leq 2^{k-2j-20}$, hence
   \[
   \begin{split}
     \tau_{\Th^-} = & \lambda \tau - \xi \cdot \omega = \lambda
     \sqrt{|\xi|^2+1} + \lambda \epsilon - \xi \cdot \omega \\
     = & |\xi| \left( (\lambda-1) \sqrt{1+|\xi|^{-2}} +
       \sqrt{1+|\xi|^{-2}} -1+1 - \frac{\xi \cdot \omega}{|\xi|} +
       \frac{\lambda \epsilon}{|\xi|} \right)
   \end{split}
   \]
   We have the following: $(1-\lambda) \sqrt{1+|\xi|^{-2}} \geq
   1-\lambda \geq 2^{-2j-8}$ (since $\lambda \in \Lambda_j$),
   $|\sqrt{1+|\xi|^{-2}}- 1| \leq 2^{-2j-12}$, $|1- \frac{\xi \cdot
     \omega}{|\xi|}| \leq 2^{-2j-12}$ and $|\frac{\lambda
     \epsilon}{|\xi|} | \leq 2^{-2j-12}$.  From these we conclude that
   $-\tau_{\Th^-} \approx 2^{k-2j}$ and also that $-\tau_{\Th^-} \geq
   2^{k-2j-10}$.

   In particular, using \eqref{inter} we obtain \eqref{DH}. Since the
   solutions in \eqref{root} can be recast in the form
   $\tau_{\Th^-}=\pm \sqrt{\Delta_{\Th}}$ and we just proved that $
   \tau_{\Th^-} < 0$ in $B_{k,\ka}$, it follows that the solutions
   $h^-$ in \eqref{root} correspond to the choice of the surface $\tau
   = \sqrt{|\xi|^2+1}$.

   We now continue with the more general bounds for $\Delta_{\Th}$ in
   the set $B_{k,\ka}$.  Since $|\tau - \la \xi \ra| \leq 2^{k-2j-30}$
   hence $ |\tau^2 - |\xi|^2 - 1| \leq 2^{2k-2j-28}$ or equivalently,
   $\tau^2 - |\xi|^2 - 1=\epsilon(\tau,\xi)$ with
   $|\epsilon(\tau,\xi)| \leq 2^{2k-2j-28}$.  We rewrite the equation
   in characteristic coordinates as above, to obtain
   \[
   \tau_{\Th^-}^2= \Delta_{\Th} + (1-\lambda^4)\epsilon
   \]
   We have already shown that $\tau_{\Th^-} \geq 2^{k-2j-10}$ and
   since $|(1-\lambda^4)\epsilon| \leq 2^{2k-2j-26} |1-\lambda| \leq
   2^{2k-2j-26} 2^{-2j+5}=2^{2k-4j-21}$, it follows that $\Delta_{\Th}
   \geq 2^{2k-2j-21}$ in $B_{k,\ka}$. A similar argument proves
   $\Delta_{\Th} \approx 2^{2k} \tilde \alpha^4$ in $B_{k,\ka}$.
 
   Although we decided to leave out the details of this argument in
   the cases $|j-k| \leq 9$ and $j \geq k+10$, we would like to point
   out a simple fact. If $j=k$, $\xi=2^k \omega$ and $\epsilon=0$, we
   obtain $\tau_{\Th^-}=0$. This highlights the reason why we cannot
   cover the case $l=m+10$ when $|j-k| \leq 9$.
 
   \bf Proof of \eqref{DH2}. \rm We start as in the proof of
   \eqref{DH} but with the goal of writing $\hat f d \sigma = F
   \delta_{\xi^2_{\Th}=h( \tau_{\Th},\xi^1_{\Th})}$. This gives the
   bound
   \begin{equation} \label{Fest2} \| F
     \|_{L^2_{\tau_{\Th},\xi^1_{\Th}}} \ls (1+\|\nabla h
     \|_{L^\infty})^\frac12 \| f \|_{L^2}
   \end{equation}
   where the $L^\infty$ norm of $\nabla h$ is taken on the support of
   $F$.
   
   We use the equation of the characteristic surface in the form
   \eqref{quadeq} and solve this equation for $\xi'_{\lambda,\om}$:
   \begin{equation} \label{root2}\begin{split} \xi^2_{\Th}= \tilde
       h^\pm(\tau_{\Th},\xi^1_{\Th}) = \pm \sqrt{\tilde \Delta_{\Th}}.
     \end{split}
   \end{equation}
   where $\tilde \Delta_{\Th}=\frac1{\lambda^2+1}(\lambda \tau_{\Th} -
   \xi^1_{\Th})^2 - \frac1{\lambda^2+1} (\tau_{\Th} + \lambda
   \xi^1_{\Th})^2-1$. Now,
   \[
   \begin{split}
     \frac{\partial \tilde h^\pm}{\partial \tau_{\Th}} =
     \frac1{\lambda^2+1} \frac{(\lambda^2-1) \tau_{\Th} - 2 \lambda
       \xi^1_{\Th}}{\xi^2_{\Th}} = \frac1{\lambda^2+1}
     \frac{\tau_{\Th^-}}{\xi^2_{\Th}}
   \end{split}
   \]
   In a similar manner we obtain $ \frac{\partial \tilde
     h^\pm}{\partial \xi^1_\Th} = (\lambda^2+1)
   \frac{\xi^1_{\Th^-}}{\xi^2_{\Th}}$, from which, using
   \eqref{Fest2}, it follows
   \begin{equation} \label{inter2} \| e^{ it \la D \ra}
     f\|_{L^\infty_{x^2_\Th}L^2_{(t,x^1)_{\Th}}} \ls \left( 1+
       \sup_{\xi \in A_{k,\ka}} \frac{2^k}{|\xi^2_{\Th}|}
     \right)^\frac12 \| f \|_{L^2}.
   \end{equation}
   To finish the argument we use
$|\xi^2_{\Th}|   = |\xi \cdot \omega^\perp| \approx 2^k \cdot \alpha$.
   As before, a direct computation shows that in the set $B_{k,\ka}$
   we have $|\xi^2_{\Th}| \approx 2^k \cdot \alpha$ and $\tilde
   \Delta_{\Th} \approx (2^k \cdot \alpha)^2$.
 
   ii) and iii) The proofs of these estimates are entirely similar to
   the corresponding ones in \cite{BH}.  The basic idea is that once
   the linear phenomenology is unraveled by \eqref{DH} and
   \eqref{DH2}, obtaining the energy type estimates is done in a
   similar manner: change the coordinates and estimate all quantities
   taking into account the localization in $B_{k,\ka}$.  Note that in
   part i) we upgraded some of our estimates to $B_{k,\ka}$.
 \end{proof}

\section{Reduction and Null structure of the cubic Dirac}\label{sect:setupD}
The cubic Dirac equation \eqref{eq:dirac} has a linear part with
matrix coefficients. Below, we rewrite \eqref{eq:dirac} as a new
system which has two half Klein-Gordon equations as linear parts, see
\eqref{CDsys} below, and we identify a null-structure in the
nonlinearity, similarly to the ideas for the Dirac-Klein-Gordon system
presented in \cite[Section 2 and 3]{dfs2d} and adapted to the Cubic
Dirac equation in dimension $n=2$ in \cite{Pe13}.  However, in
contrast to the above mentioned papers, we keep the mass term inside
the linear operator. The setup here is the two-dimensional equivalent
of \cite[Section 3]{BH} and we repeat the most important aspects.

Multiplying the cubic Dirac equation from the left with $\gamma^0$, we
obtain
\begin{equation} \label{CDmod} -i (\partial_t + \alpha \cdot \nabla +
  i \beta ) \psi= \la \psi, \beta \psi \ra \beta \psi.
\end{equation}
where $\beta=\gamma^0$ and $\alpha^j=\gamma^0 \gamma^j$ and $\alpha
\cdot \nabla = \alpha^j \partial_j$.  The new matrices satisfy
\begin{equation} \label{al} \alpha^j \alpha^k + \alpha^k \alpha^j = 2
  \delta^{jk} I_2, \qquad \alpha^j \beta + \beta \alpha^j=0.
\end{equation}

Following \cite{dfs2d} we decompose the spinor field relative to a
basis of the operator $\alpha \cdot \nabla + i \beta $ with symbol 
$\alpha \cdot \xi + \beta$. Since $(\alpha \cdot \xi + \beta)^2=
(|\xi|^2+1)I$, the eigenvalues are $\pm \la \xi \ra$.  We introduce
the projections $\Pi_{\pm}(D)$ with symbol
\[
\Pi_\pm(\xi)=\frac12 [I \mp \frac{1}{\la \xi \ra} ( \xi \cdot \alpha +
\beta)].
\]
As in \cite{BH}, we slightly deviate from \cite[formula (5)]{dfs2d} by
switching the sign in $\Pi_\pm$ for internal consistency purposes. The
key identity is
\[
-i (\alpha \cdot \nabla + i \beta ) = \langle D \rangle
(\Pi_-(D)-\Pi_+(D))
\]
where $\langle D \rangle$ has symbol $\sqrt{|\xi|^2+1}$. The following
identity, which can be verified easily at the level of the symbols,
will be important in our computations:
\[
\Pi_\pm(D) \beta = \beta (\Pi_\mp(D)\mp\frac{\beta}{\la D \ra}).\]

We then define $\psi_\pm=\Pi_\pm(D) \psi$ and split $\psi=\psi_+ +
\psi_-$. By applying the operators $\Pi_\pm(D)$ to the cubic Dirac
equation we obtain the following system of equations
\begin{equation} \label{CDsys}
  \begin{cases}
    (i\partial_t + \langle D \rangle) \psi_+ = -\Pi_+(D) (\la \psi, \beta \psi \ra \beta \psi) \\
    (i\partial_t - \langle D \rangle) \psi_- = -\Pi_-(D) (\la \psi,
    \beta \psi \ra \beta \psi).
  \end{cases}
\end{equation}
This system will replace \eqref{eq:dirac} as the object of our
research for the rest of the paper. It is obvious from the form of the
operators $\Pi_\pm$ that $\| \psi \|_{X} \approx \| \psi_+ \|_{X}+\|
\psi_- \|_{X}$ for many reasonable function spaces $X$.  In particular
we use it for $X=H^{\frac12}(\R^2)$ so that we conclude that the
initial data for \eqref{CDsys} satisfies $\psi_\pm(0) \in
H^{\frac12}(\R^2)$.

To reveal the null structure, we start with $\la \psi, \beta \psi \ra$
which, in our decomposition, is rewritten as
\[
\begin{split}
  \la \psi, \beta \psi\ra
   ={}& \la \Pi_+(D) \psi_+, \beta \Pi_+(D) \psi_+ \ra + \la \Pi_-(D) \psi_-, \beta \Pi_-(D)) \psi_- \ra \\
  & + \la \Pi_+(D) \psi_+, \beta \Pi_-(D) \psi_- \ra+ \la \Pi_-(D)
  \psi_-, \beta \Pi_+(D) \psi_+ \ra
\end{split}
\]
Next we analyze the symbols of the bilinear operators above.
\begin{lem} The following holds true
  \begin{equation} \label{PiPi}
    \begin{split}
      \Pi_\pm(\xi) \Pi_\mp(\eta) & = \mathcal{O}(\angle (\xi,\eta)) + \mathcal{O}(\la \xi \ra^{-1} + \la \eta \ra^{-1}) \\
      \Pi_\pm(\xi) \Pi_\pm(\eta) & = \mathcal{O}(\angle (-\xi,\eta)) +
      \mathcal{O}(\la \xi \ra^{-1} + \la \eta \ra^{-1})
    \end{split}
  \end{equation}
\end{lem}

Proofs of this result can be found \cite{dfs2d} or \cite{Pe13} modulo
the fact that the operators $\Pi_\pm$ there do not include the $\beta$
factor; but this is accounted by the additional factor of
$\mathcal{O}(\la \xi \ra^{-1} + \la \eta \ra^{-1})$ in the estimate
above, see also \cite[Lemma 3.1]{BH} for the three-dimensional
case. For a detailed explanation why the above result plays the role
of a null structure we refer the reader to \cite[Section 3]{BH}.

\section{Function Spaces}\label{fspaces}
Based on the estimates developed in Section \ref{sect:le} we now
define the function spaces in which we will perform the Picard
iteration for \eqref{CDsys}. The construction here is a significant refinement of
\cite[Section 4]{BH}. Some of the similarities to the function spaces
used in the wave map problem \cite{Kr03,tao,tat} are highlighted by
using a similar notation.

For $1\leq p<\infty$ we define
\[
\|f\|_{V^{p}_{\pm \la D \ra}}=\|f\|_{L^\infty_t
  L^2_x}+\Big(\sup_{(t_\nu) \in \mathcal{Z}} \sum_{\nu \in \N}
\|e^{\mp it_{\nu+1}\la D\ra }f(t_{\nu+1})-e^{\mp it_\nu\la D\ra
}f(t_\nu)\|_{L^2_x}^p\Big)^{\frac{1}{p}},
\]
where the supremum is taken over the set $\mathcal{Z}$ of all
increasing sequences.

For the following, we consider a fixed $r \in \N$ (which is implicit
in the definition, cf.\ Subsection \ref{endpnt}).

For low frequencies, that is for $k \leq 99$, we define
\[
\| f \|_{S^\pm_{k}} = \| f \|_{V^{2}_{\pm \la D \ra}} + \sup_{\om \in
  \mathbb{S}^1} \| f \|_{\sum_{\Lambda_{k,\om}}
  L^2_{t_{\Th}}L^\infty_{x_{\Th}}}.
\]
For the high frequencies, that is $k \geq 100$, the norm has a
multiscale structure. We recall the notation convention that
$\Lambda_{j,\ka_1} = \Lambda_{j ,\om(\ka_1)}$, and similarly for
$\Omega_{j,\ka_1}$. Given $l \leq k+10$, $\ka \in \mathcal{K}_l$ and
$j \geq 89$, we define structures $S^\pm[k,\ka,j]$.

If $89 \leq j = l-10 \leq k-10$ or $l=k+10 \wedge j \geq k+10$, let
\[
\| f \|_{S^\pm[k,\ka,j]} = \sup_{\ka_1 \in \mathcal{K}_{j+10}: \atop
  \dist(\kappa, \kappa_1) \leq 2^{-l+3} } \sup_{\Th \in
  \Lambda_{j,\ka_1}} 2^{-l} \| f
\|_{L^\infty_{t^\pm_{\Th}}L^2_{x^\pm_{\Th}}}.
\]
If $\max(90,l-9) \leq \min(j,k) \leq l+9$, let
\[
\| f \|_{S^\pm[k,\ka,j]} = \sup_{\ka_1 \in \mathcal{K}_{j+10}: \atop
  2^{-l-3}\leq \dist(\kappa, \kappa_1) \leq 2^{-l+3} } \sup_{\Th \in
  \Omega_{j,\ka_1}} 2^{-\frac{l}2} \| f
\|_{L^\infty_{x^{2,\pm}_{\Th}}L^2_{(t,x^1)^\pm_{\Th}}}
\]
If $\max(90,l+10) \leq \min(j,k) $, let
\[
\| f \|_{S^\pm[k,\ka,j]} = \sup_{\ka_1 \in \mathcal{K}_{j+10}: \atop
  2^{-l-3}\leq \dist(\kappa, \kappa_1) \leq 2^{-l+3} } \sup_{\Th \in
  \Lambda_{j,\ka_1}} 2^{-l} \| f
\|_{L^\infty_{t^\pm_{\Th}}L^2_{x^\pm_{\Th}}}
\]
Then for $\ka \in \Ka_l$ we define the cap localized structure as
\[
\| f \|_{S^\pm[k,\ka]} = \| f \|_{L^\infty_t L^2_x} +
\sup_{\max(89,l-10) \leq j} \| f \|_{S^\pm[k,\ka,j]}.
\]
We define the endpoint structure
\[
\| f \|_{END^\pm_k}= \Big(\sum_{\ka \in \mathcal{K}_{k+10}} 2^{-k}\|
P_{\ka} f \|^2_{\sum_{\Omega_{k,\ka}}
  L^2_{x^2_{\Th}}L^\infty_{(t,x^1)_{\Th}}} +\| P_{\ka} f
\|^2_{\sum_{\Lambda_{k,\ka}} L^2_{t^\pm_{\Th}}L^\infty_{x^\pm_{\Th}}}
\Big)^{\frac12}.
\]

Next, for some $\frac43 < p < \frac85$ (any $p$ in this range will
work, see Section \ref{sect:dirac}) we define
\[
\begin{split}
  \| f \|_{S^\pm_k} =& \left( \sum_{\ka \in \Ka_k} \|P_\ka f
    \|^2_{V^{2}_{\pm \la D \ra}} \right)^\frac12
  + 2^{(\frac1p-1)k}\sup_{m} 2^{m} \|Q_{m}^\pm f \|_{L^p_t L^2_x} \\
  & + \| f \|_{END^\pm_k} +\sup_{1\leq l\leq k+10} \Big(\sum_{\ka \in
    \mathcal{K}_l}\| Q_{\prec k-2l}^\pm P_{\ka} f \|^2_{S^\pm[k;
    \ka]}\Big)^{\frac12}
\end{split}
\]

\begin{rmk}\label{rmk:cap-sum}
  If $l_1 \geq l_2$, we have that for each $\ka_1 \in
  \mathcal{K}_{l_1}$ the number of $\ka_2 \in \mathcal{K}_{l_2}$ with
  $\ka_1 \cap \ka_2\ne \emptyset$ is uniformly bounded. As a
  consequence, essential parts of this norm are square-summable with
  respect to caps: For later purposes, we note that for $l \leq l'$,
  \begin{equation*}
    \sum_{\ka\in \mathcal{K}_l}\|P_{\ka}
    f\|_{V^2_{\pm \la D \ra}}^2\ls \sum_{\ka' \in \mathcal{K}_{l'}}\|P_{\ka'}
    f\|_{V^2_{\pm \la D \ra}}^2,
  \end{equation*}
  and, for all $1\leq l \prec k$,
  \begin{equation*}
    \sum_{\ka\in \mathcal{K}_l}\|P_{\ka}
    f\|_{V^2_{\pm \la D \ra}}^2\ls \|f\|_{S^\pm_k}^2.
  \end{equation*}
  Similarly, we have
  \begin{align*}
    &\sum_{\ka' \in \mathcal{K}_l}\Big\{ \sum_{\ka \in
      \mathcal{K}_{k+10}} 2^{-k}\| P_{\ka'}P_{\ka} f
    \|^2_{\sum_{\Omega_{k,\ka}}
      L^2_{x^2_{\Th}}L^\infty_{(t,x^1)_{\Th}}} +\| P_{\ka'}P_{\ka} f
    \|^2_{\sum_{\Lambda_{k,\ka}}
      L^2_{t^\pm_{\Th}}L^\infty_{x^\pm_{\Th}}} \Big\}\\
    \ls & \sum_{\ka \in \mathcal{K}_{k+10}} 2^{-k}\| P_{\ka} f
    \|^2_{\sum_{\Omega_{k,\ka}}
      L^2_{x^2_{\Th}}L^\infty_{(t,x^1)_{\Th}}} +\| P_{\ka} f
    \|^2_{\sum_{\Lambda_{k,\ka}}
      L^2_{t^\pm_{\Th}}L^\infty_{x^\pm_{\Th}}}\ls \|f\|_{END_k^\pm}^2.
  \end{align*}
  For this reason we introduce the norm
  \[
  \begin{split}
    \| f \|_{\lS^\pm_k} =& \left( \sum_{\ka \in \Ka_k} \|P_\ka f
      \|^2_{V^{2}_{\pm \la D \ra}} \right)^\frac12 + \| f
    \|_{END^\pm_k}
  \end{split}
  \]
  which has now the property that for any $1 \leq l \leq k+10$
  \begin{equation} \label{l2S} \sum_{\ka \in \Ka_l} \| P_\ka f
    \|_{\lS^\pm_k}^2 \ls \| f \|_{\lS^\pm_k}^2.
  \end{equation}
  For any $|l-l'|\leq 10$, we also have
  \begin{equation*}\sum_{\ka'\in \mathcal{K}_{l'}} \sum_{\ka \in
      \mathcal{K}_l}\|
    P_{\ka'} Q_{\prec k-2l}^\pm P_{\ka} f \|^2_{S^\pm[k; \ka]}
    \ls\|f\|_{S_k^\pm}^2,
  \end{equation*}
  where we use Part i) of Lemma \ref{stable} below.
\end{rmk}

The space $S^{\pm, \sigma}$ corresponding to regularity at the level
of $H^\sigma(\R^2)$ is the complete subspace of
$L^\infty(\R,H^\sigma(\R^2))$ defined by the norm
\[
\| f \|_{S^{\pm, \sigma}} = \| P_{\leq 89} f \|_{S_{ 89}^\pm} +
\Big(\sum_{k \geq 90} 2^{2k \sigma} \| P_k f
\|^2_{S^\pm_k}\Big)^{\frac12}.
\]
Recall from Subsection \ref{endpnt} that this construction is useful
up to time $2^r$, so for any closed interval $I \subset (-2^r,2^r)$ we
define the space $S^{\pm, \sigma}(I)$ of all functions on $I$ which
have extensions to functions in $S^{\pm, \sigma}$, with norm
\[
\|f\|_{S^{\pm, \sigma}(I)}=\inf_{F \in S^{\pm,
    \sigma}}\{\|F\|_{S^{\pm, \sigma}}: F|_I=f\}.
\]
Note that the space $S^{\pm, \sigma}_{C}(I) :=S^{\pm, \sigma}(I)\cap
C(I,H^\sigma(\R^2))$ is a closed subspace of $S^{\pm, \sigma}(I)$.

Now we construct the space for the nonlinearity. For $1\leq q\leq
\infty$, $b \in \R$, we define
\[\|f\|_{\dot{X}^{\pm,b,q}}=\big\|\big(2^{bm}\|Q_{m}^\pm
f\|_{L^2}\big)_{m \in \Z }\big\|_{\ell^q_m}.\]

For the low frequency part of the nonlinearity we define
\[
\| f \|_{N^{\pm}_{0}} = \inf_{f=f_1+f_2+f_3} \Big\{\|f_1
\|_{\dot{X}^{\pm,-\frac12,1}} + \| f_2 \|_{L^1_tL^2_x} +\| f_3
\|_{L^\frac43_{t,x}} \Big\} +\|f\|_{L^p_tL^2_x}.
\]
Let $(N_{0}^{\pm})^*$ denote the dual of $N_{0}^{\pm}$ and let
$S^{\pm,w}_{0}$ be endowed with the norm
\begin{equation}\label{eq:skweakl}
  \| f \|_{S^{\pm,w}_{0}}= \| f \|_{L^\infty_t L^2_x} + \| f
  \|_{\dot X^{\pm,\frac12,\infty}}. 
\end{equation} 
Then, we observe that for $k \leq 99$,
\begin{equation} \label{duall} S_{k}^\pm \subset (N_{0}^{\pm})^*
  \subset S_{0}^{\pm,w}.
\end{equation}

Next, let $k \geq 100$. For $1 \leq l \leq k+10$ we consider $\ka \in
\mathcal{K}_l$ and consider the following types of atoms:

\bf A1 \rm: If $ 89 \leq j = l-10\leq k-10$ or $l=k+10 \wedge j\geq
k+10$, functions $f_\Theta$ with \[2^{l} \| f_\Theta
\|_{L^1_{t^\pm_{\Theta}}L^2_{x^\pm_{\Theta}}}=1,\] where $\Theta \in
\Lambda_{j,\ka_1}$ and $\ka_1 \in \mathcal{K}_{j+10}$ with
$\dist(\ka_1,\ka) \leq 2^{-l+3}$.

\bf A2 \rm: If $\max(90,l-9) \leq \min(j,k) \leq l+9$, functions
$f_{\Theta}$ with \[2^{\frac{l}2} \| f_\Theta
\|_{L^1_{x^{2,\pm}_{\Theta}}L^2_{(t,x^1)^\pm_{\Theta}}}=1,\] where
$\Theta \in \Omega_{j,\ka_1}$ and $\ka_1 \in \mathcal{K}_{j+10}$ with
$2^{-l-3} \leq \dist(\ka_1,\ka) \leq 2^{-l+3}$,

\bf A3 \rm: If $\max(90,l+10) \leq j \leq \min(j,k)$, functions
$f_\Theta$ with \[2^{l} \| f_\Theta
\|_{L^1_{t^\pm_{\Theta}}L^2_{x^\pm_{\Theta}}}=1,\] where $\Theta \in
\Lambda_{j,\ka_1}$ and $\ka_1 \in \mathcal{K}_{j+10}$ with $2^{-3}
\leq 2^l \dist(\ka_1,\ka) \leq 2^{3}$.

We then define, in the standard way, $N^\pm[k,\ka]$ to be the atomic
space based on the above atoms.

Now, similarly to \cite{BH}, we define the following atomic structure
\begin{equation}\label{eq:n-atom}
  \begin{split}
    \| f \|_{N_k^{\pm,at}}=& \inf_{f=f_1+f_2+\sum_{1 \leq l \leq k+10} g_{l} } \Big\{\|f_1 \|_{\dot{X}^{\pm,-\frac12,1}} + \| f_2 \|_{L^1_tL^2_x} \\
    {}&+ \sum_{1 \leq l \leq k+10} \Big( \sum_{\ka \in \mathcal{K}_l}
    \| P_{\ka} g_{l} \|_{N^\pm[k, \ka]}^2 \Big)^\frac12\Big\}
  \end{split}
\end{equation}
where the atoms $g_l$ in the above decomposition are assumed to be
localized at frequency $2^k$ and modulation $\ll 2^{k-2l}$, more
precisely that $\tilde{Q}_{\prec k-2l}^\pm \tilde{P}_k g_l = g_l$.

The third component in $N_k^{\pm,at}$, i.e. the $\sum_{1 \leq l \leq
  k+10} g_{l}$, will henceforth be called the cap-localized structure.
The atoms $g_l$ are localized in frequency and modulation, while when
they are measured in $N^\pm[k,\ka]$ the atoms $a_\Theta$ in the
decomposition $g_l = \sum_\Theta a_\Theta$ are not assumed to keep
that localization. However, by applying the operator $\tilde{Q}_{\prec
  k-2l}^\pm \tilde{P}_{k,\ka}$ to the decomposition and using
\cite[Lemma~4.1~i)]{BH} (which holds true in dimension $2$ verbatim)
one obtains a new decomposition with similar norm. From now on our
convention is that we assume that the atoms $a_\Theta$ in the atomic
decomposition have the correct frequency and modulation localization.

Let $(N_{k}^{\pm,at})^*$ denote the dual of $N_{k}^{\pm,at}$ and
$S_k^{\pm,w}$ be endowed with the norm
\begin{equation}\label{eq:skweak}
  \| f \|_{S^{\pm,w}_k}= \| f \|_{L^\infty_t L^2_x} + \| f \|_{\dot X^{\pm,\frac12,\infty}} 
  + \sup_{1 \leq l \leq k+10} \Big( \sum_{\ka \in \mathcal{K}_l} \|
  Q_{\prec k-2l}^\pm P_{\ka} f \|^2_{S^\pm[k;\ka]}\Big)^{\frac12}.
\end{equation} 
Then, we record that
\begin{equation} \label{dual} S_k^\pm \subset (N_{k}^{\pm,at})^*
  \subset S_k^{\pm,w},
\end{equation}
with continuous embeddings, i.e.\
\[ \|f\|_{S_k^{\pm,w}}\ls \|f\|_{(N^{\pm,at}_{k})^{*}}\ls
\|f\|_{S_k^\pm}.\]

Now we are in a position to define the space for dyadic pieces of the
nonlinearity for high frequencies by setting
\[
\| f \|_{N^\pm_k} = \| f \|_{N^{\pm,at}_k} + 2^{(\frac1p-1)k} \| f
\|_{L^p_t L^2_x}.
\]
The space for the nonlinearity at regularity $H^\sigma$ is defined via
\[
\| f \|_{N^{\pm, \sigma}} = \| P_{\leq 89} f \|_{N_{\leq 89}^\pm} +
\Big(\sum_{k \geq 90} 2^{2k \sigma} \| P_k f
\|^2_{N^\pm_k}\Big)^{\frac12}.
\]

Now we show why the above structures are relevant for the equations we
study.  We first note a technical result on boundedness properties of
certain frequency and modulation localization operators.

\begin{lem}\label{stable}

  {\rm i)} For all $k \geq 100$ and $m \geq 1$, the operators $ \tilde
  Q_{\leq m}^\pm$ are bounded on $S^\pm_k, N^\pm_k$.
  
  {\rm ii)} For all $k \geq 100, 1\leq l \leq k+10, \ka \in
  \mathcal{K}_l$, and functions $u$ localized at frequency $2^k$, we
  have
  \begin{equation} \label{sta} \| \left( \Pi_\pm(D) - \Pi_\pm(2^{k}
      \omega(\ka)) \right) P_{\ka} u \|_{S} \ls 2^{-l} \| P_{\ka} u
    \|_{S}
  \end{equation}
  for $S \in \{ S_k^\pm, S_k^{\pm,w} \}$.
\end{lem}

\begin{proof}
  i) We start with the boundedness of $ \tilde Q_{\leq m}^\pm$ on the
  components of $S^\pm_k$.  The boundedness of $\tilde Q_{\leq m}^\pm$
  on the $V^2_{\pm \la D \ra}$ is standard, see
  e.g. \cite[Cor.~2.18]{hhk09}.  The boundedness of $\tilde Q_{\leq
    m}^\pm$ on the
  \[
  2^{(\frac1p-1)k}\sup_{m'} 2^{m'} \|Q_{m'}^\pm f \|_{L^p_t L^2_x}
  \]
  structure follows from the commutativity property $Q_{m'}^\pm \tilde
  Q_{\leq m}^\pm = \tilde Q_{\leq m}^\pm Q_{m'}^\pm$ and the
  boundedness of $\tilde Q_{\leq m}^\pm$ on the $L^p_t L^2_x$ type
  spaces.

  Next, we notice that the kernel of $\tilde Q_{\leq m}^\pm \tilde
  P_\ka$ belongs to $L^1_{t,x}$ under the hypothesis $m \geq 1$ and
  $\kappa \in \Ka_{k+10}$. Using that $P_\ka \tilde Q_{\leq m}^\pm =
  \tilde Q_{\leq m}^\pm \tilde P_\ka P_\ka $, this implies the
  boundedness of $\tilde Q_{\leq m}^\pm$ on the
  \[
  \Big(\sum_{\ka \in \mathcal{K}_{k+10}} 2^{-k}\| P_{\ka} f
  \|^2_{\sum_{\Omega_{k,\ka}} L^2_{x^2_{\Th}}L^\infty_{(t,x^1)_{\Th}}}
  +\| P_{\ka} f \|^2_{\sum_{\Lambda_{k,\ka}}
    L^2_{t^\pm_{\Th}}L^\infty_{x^\pm_{\Th}}} \Big)^{\frac12}
  \]
  component of $S_k^\pm$.

  For the boundedness of $\tilde Q_{\leq m}^\pm$ on the $S^\pm[k,\ka]$
  components we use an argument similar to the one used in \cite[Lemma
  4.1]{BH}, part ii). $S^\pm[k,\ka]$ itself has several components and
  we will provide a complete argument for one of them; this will also
  serve as a template for the other ones. With $\ka \in \Ka_l$ for
  some $1 \leq l \leq k+10$, it is enough to consider only the case $m
  \prec k-2l$. We fix the $+$ sign choice, fix $j$ with $\max(90,l+10)
  \leq \min(j,k) $, consider $\ka_1$ with $ 2^{-l-3}\leq \dist(\kappa,
  \kappa_1) \leq 2^{-l+3}$ and $\Th \in \Lambda_{j,\ka_1}$.

  The operator $\tilde Q_{\leq m}^+ \tilde P_{k,\kappa}$ is a Fourier
  multiplier whose symbol \[a_{m,k,\ka}(\tau,\xi)=\tilde \chi_{\leq m}
  (\tau - \la \xi \ra) \tilde \chi_k (\xi) \tilde \eta_\ka(\xi)\]
  satisfies $  |\partial_{\tau_\Th}^\beta a_{m,k,\ka}| \ls (2^{m+2l})^{-\beta}$.
  The inverse Fourier transform of $a_{m,k,\ka}$ with respect to
  $\tau_{j,\ka_1}$ satisfies
  \[
  |K_{l,k,\ka}(t_\Th,\xi_\Th)| \ls_N 2^{m+2l}(1+|t_\Th|2^{m+2l})^{-N},
  \text{ for any }N \in \N.
  \]
  From this we obtain the uniform bound
  \[
  \| K_{l,k,\ka} \|_{L^1_{t_\Th}L^\infty_{\xi_\Th}} \ls 1.
  \]
  On the other hand we have
  \[
  \mathcal{F}_{\xi_\Th} ( \tilde Q^+_{m} \tilde P_{k,\ka} f) =
  K_{l,k,\ka} *_{t_\Th} \mathcal{F}_{\xi_\Th} f,
  \]
  where one performs convolution with respect to $t_\Th$ variable
  only. From the above statements it follows that $\tilde Q_{\leq m}^+
  \tilde P_{k,\kappa}$ is bounded on $L^\infty_{t_\Th}L^2_{x_\Th}$.

  Proving the bounds for $\tilde Q_{\leq m}^\pm$ on the components of $N^\pm_k$
  is done in an entirely similar way.

  ii) The proof is 
similar to \cite[Lemma 4.1]{BH} and therefore omitted.
\end{proof}

We continue with a few preparatory results. In order to later deal
with the $V^2_{\pm \la D \ra}$ structure, we show that the analogue of
the \emph{fungibility estimate} \cite[formula (159)]{stta10} holds in
our spaces, more precisely
\begin{lem}\label{lem:fung}
  For all $g=\tilde{P}_k g$ and any collection $(I_\nu)_{\nu \in \N}$
  of disjoint intervals the estimate
  \begin{equation}\label{eq:fung}
    \sum_\nu \|1_{I_\nu}g\|^2_{N_k^\pm}\ls \|g\|_{N_k^\pm}^2
  \end{equation}
  holds true, uniformly in $k\geq 100$.
\end{lem}
\begin{proof}
  We proceed similarly to \cite[pp.~176-178]{stta10}, the minor
  differences in the following proof are mostly due to the lack of
  scale invariance:

  It suffices to consider the $+$-case. It is obvious for
  $L^1_tL^2_x$-atoms, so we are left with
  $\dot{X}^{+,-\frac12,1}$-atoms and the cap-localized structure.

  a) $\dot{X}^{+,-\frac12,1}$-atoms: We will prove
  \begin{equation}\label{eq:xatoms}
    \sum_{\nu}\|1_{I_\nu}f_1\|^2_{L^1_tL^2_x+\dot{X}^{+,-\frac12,1}}\ls \|f_1\|_{\dot{X}^{+,-\frac12,1}}^2,
  \end{equation}
  for $\tilde{P}_kf_1=f_1$.  By definition, this follows from
  \begin{equation}\label{eq:xatoms-mod}
    \sum_{\nu}\|1_{I_\nu}Q_m f_1\|^2_{L^1_tL^2_x+\dot{X}^{+,-\frac12,1}}\ls 2^{-m}\|Q_mf_1\|_{L^2}^2,
  \end{equation}
  which we establish by proving
  \begin{align}
    \label{eq:xatoms-mod1}\sum_{\nu}\|Q_{\succeq m}(1_{I_\nu}Q_m f_1)\|^2_{L^2}\ls \|Q_mf_1\|_{L^2}^2,\\
    \label{eq:xatoms-mod2}\sum_{\nu}\|Q_{\prec m}(1_{I_\nu}Q_m
    f_1)\|^2_{L^1_tL^2_x}\ls 2^{-m}\|Q_mf_1\|_{L^2}^2.
  \end{align}
  The first one is trivial, since $Q_{\succeq m}$ is bounded in $L^2$,
  so we focus on \eqref{eq:xatoms-mod2}: Let $(J_\nu)$ be the
  subcollection of all intervals in $(I_\nu)$ satisfying
  $|J_\nu|>2^{-m}$ and $(K_\nu)$ all remaining intervals. For the
  short intervals $(K_\nu)$, we obtain
  \[
  \sum_{\nu}\|Q_{\prec m}(1_{K_\nu}Q_m f_1)\|^2_{L^1_tL^2_x}\ls \sum_{\nu}\|1_{K_\nu}Q_m f_1\|^2_{L^1_tL^2_x}
    \ls 2^{-m}\sum_{\nu}\|1_{K_\nu}Q_m f_1\|^2_{L^2}.
  \]
    Concerning the long intervals $(J_\nu)$, we have
  \[
  Q_{\prec m}(1_{J_\nu}Q_m f_1)=Q_{\prec m}((Q_{\sim m}1_{J_\nu})(Q_m
  f_1))
  \]
  and it is easily checked that
  \begin{equation*}
    |Q_{\sim m}1_{[a,b]}(t)|\ls_N \alpha_{[a,b],m}(t)^{-N}, \quad \alpha_{[a,b],m}(t):=1+2^m|t-a|+2^m|t-b|.
  \end{equation*}
  Let $J_\nu=[a_\nu, b_\nu]$. Because of their disjointness and
  $|J_\nu|>2^{-m}$, we have
  \[
  \sum_{\nu} \alpha_{[a_\nu,b_\nu],m}^{-N}(t)\ls \sum_\nu
  (1+2^m|t-a_\nu|+2^m|t-b_\nu|)^{-N}\ls 1 \quad (N>1).
  \]
  Fix $N=2$. We conclude that
  \begin{align*}
    &\sum_{\nu}\|Q_{\prec m}(1_{J_\nu}Q_m f_1)\|^2_{L^1_tL^2_x}\ls \sum_{\nu}\|(Q_{\sim m}1_{J_\nu})(Q_m f_1)\|_{L^1_tL^2_x}^2\\
    \ls{} & \sum_{\nu}\|\alpha_{[a_\nu,b_\nu],m}^{-1}\|_{L^2_t}^2\|\alpha_{[a_\nu,b_\nu],m}^{-1}Q_m f_1\|^2_{L^2_tL^2_x}\\
    \ls{} & 2^{-m} \sum_{\nu}\|\alpha_{[a_\nu,b_\nu],m}^{-1}(t)Q_m f_1\|^2_{L^2_tL^2_x}\\
    \ls{} & 2^{-m} \int_\R
    \sum_{\nu}\alpha_{[a_\nu,b_\nu],m}^{-2}(t)\|Q_m
    f_1(t)\|_{L^2_x}^2dt\ls 2^{-m}\|Q_m f_1\|_{L^2}^2.
  \end{align*}

  b) cap-localized structure: Consider $f_3=\sum_{1\leq l \leq
    k+10}g_l$ satisfying $\tilde{Q}^+_{\prec k-2l} \tilde{P}_k
  g_l=g_l$.  For fixed $1\leq l \leq k+10$, we write
  \[
  1_{\nu} g_l=\tilde{Q}^+_{\succeq k-2l}(1_{\nu}
  g_l)+\tilde{Q}^+_{\prec k-2l} (1_{I_\nu} g_l)
  \]
  By a similar argument as presented in \cite[Proof of Prop.\ 4.2,
  Part 1, Case c)]{BH} it follows that
  \begin{equation*}
    \| P_\ka g_l\|_{L^2_{t,x}}\ls 2^{\frac{k-2l}{2}} \|P_\ka g_l\|_{N^+[k,\ka]}.
  \end{equation*}

  For the first contribution, this implies
  \begin{align*}
    &\sum_{\nu} \|\tilde{Q}^+_{\succeq k-2l} (1_{I_\nu}
    g_l)\|^2_{\dot{X}^{+,-\frac12,1}}\ls 2^{2l-k} \sum_{\ka \in
      \mathcal{K}_l}\sum_{\nu} \|(1_{I_\nu}
    P_\ka g_l)\|^2_{L^2_{t,x}} \\
    &\ls 2^{2l-k} \sum_{\ka \in \mathcal{K}_l}\| P_\ka
    g_l\|^2_{L^2_{t,x}}\ls \sum_{\ka \in \mathcal{K}_l}\|P_\ka
    g_l\|_{N^+[k,\ka]}^2.
  \end{align*}
  For the second contribution we use Lemma 4.1 and the fact that
  \[
  \sum_{\nu}\|1_{I_\nu}h\|_{L^1_{y_1}L^2_{y_2}}^2 \ls
  \|h\|_{L^1_{y_1}L^2_{y_2}}^2
  \]
  for any orthogonal frame $(y_1,y_2)\in \R^{1+2}$ due to Minkowski's
  inequality to deduce that for fixed $\ka\in \mathcal{K}_l$ we have
  \[
  \sum_\nu\|\tilde{Q}^+_{\prec k-2l} (1_{I_\nu} P_\ka
  g_l)\|_{N^+[k,\ka]}^2\ls \sum_\nu\| (1_{I_\nu} P_\ka
  g_l)\|_{N^+[k,\ka]}^2\ls \|P_\ka g_l\|_{N^+[k,\ka]}^2,
  \]
  which we then sum up with respect to $\ka \in \mathcal{K}_l$.
  We obtain
  \begin{align*}
    &\Big(\sum_\nu \|1_{I_\nu} f_3\|_{N^{+,at}_k}^2\Big)^{\frac12}\ls
    \sum_{1\leq l \leq
      k-10}\Big\{\Big(\sum_\nu \|\tilde{Q}^+_{\succeq k-2l} (1_{I_\nu} g_l)\|_{\dot{X}^{+,-\frac12,1}}^2\Big)^{\frac12}\\
    &+\Big(\sum_\nu\sum_{\ka \in \mathcal{K}_l}\|\tilde{Q}^+_{\prec
      k-2l}
    (1_{I_\nu} P_\ka g_l )\|_{N^+[k,\ka]}^2\Big)^{\frac12}\Big\}\\
    &\ls \sum_{1\leq l \leq k+10}\Big(\sum_{\ka \in
      \mathcal{K}_l}\|P_\ka g_l\|_{N^+[k,\ka]}^2\Big)^{\frac12},
  \end{align*}
  and the proof is complete.
\end{proof}

Let $\psi$ be any fixed Schwartz function and
$\psi_T(\cdot)=\psi(\frac{\cdot}{T})$.
\begin{lem}\label{lem:time-loc}
  Fix any $1\leq p\leq 2$. For all $T>0$ we have
  \begin{equation*}
    \sup_{m\in Z} 2^m \|Q_m(\psi_T P_k f)\|_{L^p_t L^2_x}\ls \sup_{m\in \Z}
    2^m \|Q_mP_k f\|_{L^p_t L^2_x}+T^{\frac{1}{p}-1}\|P_k f\|_{L^\infty_t L^2_x}.
  \end{equation*}
  Consequently, there exists $c>0$ such that for any closed interval
  $I\subset (-2^{r-1},2^{r-1})$, we have
  \begin{equation}\label{eq:time-loc-free}
    \|e^{\pm it\la D\ra}\phi\|_{S^{\pm,\sigma}(I)}\leq c \|\phi\|_{H^\sigma(\R^2)}.
  \end{equation}
\end{lem}
\begin{proof}
  Let $f=\tilde{P}_k f$. Obviously, $ \|Q_{\ls m} \psi_T\|_{L^\infty_t}\ls 1$
  and $[Q_m\psi_T](t)=[Q_{T 2^m}\psi](\frac{t}{T})$, hence \[\|Q_{m}
  \psi_T\|_{L^q_t}\ls_N T^{\frac{1}{q}}\la T2^{m}\ra^{-N} \text{ for
    any } N\in \N.
  \]
  We split
  \[
  Q_m(\psi_T f)=Q_m[Q_{\ll m}(\psi_T) f]+Q_m[Q_{\sim m}(\psi_T)
  f]+Q_m[Q_{\gg m}(\psi_T) f].
  \]
  First,
  \[
  2^m \|Q_m[Q_{\ll m}(\psi_T) f]\|_{L^p_t L^2_x}\ls \|Q_{\ll
    m}(\psi_T)\|_{L^\infty_t} 2^m \|Q_m f\|_{L^p_t L^2_x}.
  \]
  Second,
  \begin{align*}
    &2^m \|Q_m[Q_{\sim m}(\psi_T) f]\|_{L^p_t L^2_x}\ls \|Q_{\sim
      m}(\psi_T)\|_{L^p_t} 2^m \| f\|_{L^\infty_t L^2_x}\\
    & \ls T^{\frac{1}{p}}\la T 2^m\ra^{-1} 2^m \|f\|_{L^\infty_t
      L^2_x}\ls T^{\frac{1}{p}-1}\|f\|_{L^\infty_t L^2_x}.
  \end{align*}
  Third,
  \begin{align*}
    2^m \|Q_m[Q_{\gg m}(\psi_T) f]\|_{L^p_t L^2_x}&\ls 2^m \sum_{m_1\gg m} \|Q_{m_1}(\psi_T)Q_{m_1} f\|_{L^p_t L^2_x}\\
    &\ls 2^m \sum_{m_1\gg m} \|Q_{m_1}(\psi_T)\|_{L^\infty_t}
    \|Q_{m_1}
    f\|_{L^p_t L^2_x}\\
    & \ls \sum_{m_1\gg m} 2^{m-m_1} \sup_{m_1}2^{m_1}\|Q_{m_1}
    f\|_{L^p_t L^2_x}.
  \end{align*}
  Concerning the second claim, we define the extension $F=\psi_T
  e^{\pm it\la D\ra}\phi$, where we choose $\psi$ to be equal to $1$
  on $(-1,1)$, to be supported in $(-2,2)$ and $\psi_T$ defined as
  above with $T=2^{r-1}$. The estimate follows from the first claim,
  the results from Section \ref{sect:le} and the fact that
  multiplication with smooth cutoffs is a bounded operation in $V^2$.
\end{proof}
\begin{pro}\label{pro:lin}
  i) For all $g\in N_k^\pm$ and initial data $u_0\in L^2(\R^2)$, both
  localized at (spatial) frequency $2^k$ (in the sense that $\tilde
  P_k g = g, \tilde P_k u_0 =u_0$), $k \geq 100$, the solution $u$ of
  \begin{equation} \label{ng} (i \partial_t \pm \la D \ra) u = g,
    \quad u(0)=u_0,
  \end{equation}
  satisfies $\psi_T u \in S_k^\pm$ for all $1\ls T\ls 2^r$, and
  \begin{equation} \label{eq:sk-bound}\|\psi_T u\|_{S_k^\pm} \ls \| g
    \|_{N_k^\pm} + \| u_0 \|_{L^2}.
  \end{equation}
  
  ii) A similar statement holds true for $90 \leq k \leq 99$. For all
  $g\in N_{\leq 89}^\pm$ and initial data $u_0\in L^2(\R^2)$, both
  localized at (spatial) frequency $\leq 2^{89}$ (in the sense that
  $\tilde P_{\leq 89} g = g, \tilde P_{\leq 89} u_0 =u_0$), the
  solution $u$ of \eqref{ng} satisfies $\psi_T u \in S_{\leq 89}^\pm$
  for all $1\ls T\ls 2^r$, and
  \begin{equation} \label{eq:sk-bound-low}\|\psi_T u\|_{S_{\leq
        89}^\pm} \ls \| g \|_{N_{\leq 89}^\pm} + \| u_0 \|_{L^2}.
  \end{equation}

\end{pro}
\begin{proof}
  i) It suffices to consider the $+$ case. Due to Lemma
  \ref{lem:time-loc} it suffices to consider $u_0=0$.
  Our first claim is that we have the following estimate:
  \begin{equation} \label{help1} \| u\|_{S_k^+ \setminus END_k^+} +
    \|\psi_T u\|_{END_k^+} \ls \| g \|_{N_k^\pm}
  \end{equation}
  where $S_k^+ \setminus END_k^+$ contains all norm components of
  $S_k^\pm$ except the $END_k^+$ one.  The time cut-off in is needed
  to recoup the $END_k^+$ structure.
  Besides the $V^2_{\la D \ra}$ component, the proof of \eqref{help1}
  is analogous to the $3d$ case in \cite[Prop.~4.2]{BH}, which, in
  particular, implies the $L^\infty_t L^2_x$-bound. In what follows we
  provide the estimate for the $V^2_{\la D \ra}$ part of
  \eqref{help1}.

  First, we follow the general strategy of \cite[Prop.~5.4 and Lemma
  5.8]{stta10} to prove the $V^2_{\la D \ra}$-estimate on a fixed cap
  $\kappa \in \mathcal{K}_l$ with $l := k+10$:

  For any interval $[a,b]$ the function
  \[
  w_\ka(t)=P_\ka u(t)-e^{i(t-a)\la D \ra} P_\ka u(a)
  \]
  solves
  \[
  (i \partial_t \pm \la D \ra) w_\ka = P_\ka g ,\quad w_\ka (a)=0,
  \]
  hence we obtain, using the $L^\infty_t L^2_x$-bound,
  \[
  \|P_\ka u(b)-e^{i(b-a)\la D\ra }P_\ka
  u(a)\|_{L^2_x}\ls\|1_{[a,b]}P_\ka g\|_{N_k^+}.
  \]
  For any $(t_\nu)\in \mathcal{Z}$, using \eqref{eq:fung}, we conclude
  \begin{align*}
    \sum_\nu \|e^{-it_{\nu+1} \la D \ra }P_\ka u(t_{\nu+1})-e^{-it_\nu
      \la D \ra }P_\ka u(t_{\nu})\|_{L^2}^2&\ls
    \sum_\nu \|1_{[t_{\nu},t_{\nu+1}]}P_\ka g\|^2_{N_k^+}\\
    &\ls \|P_\ka g\|_{N_k^+}^2,
  \end{align*}
  and finally we take the supremum over $\mathcal{Z}$.

  Second, we sum up the squares: By the estimate above,
  \[
  \Big( \sum_{\ka \in \mathcal{K}_l}\|P_\ka u\|_{V^2_{\la D \ra
    }}^2\Big )^{\frac12}\ls \Big( \sum_{\ka \in \mathcal{K}_l}\|P_\ka
  g\|_{N_k^+}^2\Big )^{\frac12},
  \]
  hence it remains to prove
  \begin{equation}\label{capf}
    \Big( \sum_{\ka \in \mathcal{K}_l}\|P_\ka g\|_{N_k^+}^2\Big
    )^{\frac12}\ls \|g\|_{N_k^+},
  \end{equation}
  uniformly in $1\leq l \leq k+10$. By Minkowski's inequality, this is
  obviously true for the $L^{p}_tL^2_x$-part of the $N_k^+$-norm, and
  also for the $\dot{X}^{+,-\frac12,1}$ and $L^1_tL^2_x$-atoms in
  $N^{+,at}_k$, so it remains to prove it for the cap-localized
  structure. We observe that
  \begin{align*}
    &\Big( \sum_{\ka \in \mathcal{K}_{k+10}}\Big(\sum_{1\leq l' \leq
      k+10}\Big(\sum_{\kappa'\in \mathcal{K}_{l'}}\|P_{\ka'} P_\ka
    g\|_{N^+[k,\ka']}^2\Big)^{\frac12}\Big)^2\Big
    )^{\frac12}\\
    &\ls \sum_{1\leq l' \leq k+10}\Big(\sum_{\kappa'\in
      \mathcal{K}_{l'}} \sum_{\ka \in \mathcal{K}_{k+10}}\|P_{\ka'}
    P_\ka g\|_{N^+[k,\ka']}^2\Big)^{\frac12}.
  \end{align*}

  We now argue why \eqref{capf} holds for the case when $g$ is an atom
  in the cap localized structure.  The only non-trivial case is when
  $g_\Theta= \tilde{Q}_{\prec k-2l'} \tilde{P}_{\ka'} g_\Theta$ where
  $\ka' \in \Ka_{l'}$ and $l' \leq k+10$, while the information we
  have is control on $\| g_\Theta \|_{L^1_{t_\Th}L^2_{x_\Th}}$ or $\|
  g_\Theta \|_{L^1_{x^2_\Th}L^2_{(t,x^1)_\Th}}$ as described in A1 -
  A3 prior to the definition \eqref{eq:n-atom}. Without restricting
  the generality of the argument, consider we have control of the
  first type. The key observation is that the operators $P_\ka
  \tilde{Q}_{\prec k-2l'} \tilde{P}_{\ka'}$ are almost orthogonal with
  respect to $\ka \in \Ka_l$ when acting on $L^2_{x_\Th}$. One way to
  formalize this is through the identity $P_\ka \tilde{Q}_{\prec
    k-2l'} \tilde{P}_{\ka'} = \tilde P(\ka, \xi_\Th) P_\ka
  \tilde{Q}_{\prec k-2l'} \tilde{P}_{\ka'}$ where $\tilde P(\ka,
  \xi_\Th)$ are operators localizing the Fourier variable $\xi_\Th$ in
  almost disjoint cap-type regions. This is a consequence of the
  transversality between the direction $\Th$ and the Fourier support
  of $\tilde{Q}_{\prec k-2l'} \tilde{P}_{\ka'} $.

  Taking advantage of this almost orthogonality, we obtain
  \[
  \sum_{\ka \in \Ka_{k+10}} \|P_\ka g_\Th
  \|^2_{L^1_{t_\Th}L^2_{x_\Th}} \ls \| g_\Th
  \|^2_{L^1_{t_\Th}L^2_{x_\Th}},
  \]
  and this finishes the proof of \eqref{help1}.
  
  Next we show how we derive \eqref{eq:sk-bound} using
  \eqref{help1}. The problem encountered by a direct argument is that
  $\psi_T$ does not commute well with the modulation localizations
  present in the $S^+[k,\kappa]$. $\psi_T u$ solves the following
  equation:
  \begin{equation} \label{help3} (i \partial_t \pm \la D \ra) (\psi_T
    u) = \psi_T g + i \psi_T' u.
  \end{equation}
  with the initial data $\psi_T u(0)=u(0)=0$. Since we have
  \[
  \| i \psi_T' u \|_{L^1_t L^2_x} \ls \| u \|_{L^\infty_t L^2_x} \ls
  \| u \|_{S_k^+} \ls \| g \|_{N_k}
  \]
  and from the proof of Lemma \ref{lem:fung} we easily obtain
  \begin{equation}
    \| \psi_T g  \|_{N_k^+} \ls \| g  \|_{N_k^+}.
  \end{equation}
  We can invoke again \eqref{help1}, this time for the equation
  \eqref{help3}, to obtain
  \[
  \| \psi_T u \|_{S_k^+ \setminus END_k^+} \ls \| g \|_{N_k^+}.
  \]
  This concludes the proof of \eqref{eq:sk-bound}.
 
  ii) The proof of part ii) can be carried over in a similar but
  simpler way, except for the case when $g \in L^\frac43_{t,x}$. A
  complete argument, including the $L^\frac43_{t,x}$ part, can be
  found in \cite[Proposition 7.2]{bikt}.
\end{proof}

\begin{cor} \label{cor:lin} For any $r \in \N$, closed intervals
  $I\subset (-2^{r-1},2^{r-1})$, all $u_0 \in H^\sigma(\R^2)$ and $g
  \in N^{\pm,\sigma}$, there exists a unique solution $u\in
  S^{\pm,\sigma}(I)$ of \eqref{ng}, and the following estimate holds
  true
  \begin{equation} \label{eq:lin} \| u \|_{S^{\pm,\sigma}(I)} \ls \| g
    \|_{N^{\pm,\sigma}(I)} + \| u_0 \|_{H^\sigma}.
  \end{equation}
\end{cor}
\begin{proof}
  By definition of the spaces, it suffices to prove this for frequency
  localized functions which is provided by Proposition \ref{pro:lin}
  above.
\end{proof}

Now, we conclude that we can control all non-endpoint Strichartz norms
in our spaces, see also \cite{Kr03,Kr04,tao,KS12} for other Strichartz
type bounds. We refine the argument from \cite{stta10} in the sense
that we include additional cap-localizations which give stronger
bounds.

\begin{cor}\label{cor:full-str}
  Let $p,q\geq 2$ such that $(p,q)$ is a Schr\"odinger-admissible
  pair, i.e.\
  \[(p,q)\ne (2,\infty), \; \frac{1}{p}+\frac{1}{q}= \frac12,
  \text{and }s=1-\frac{2}{q}\] or a wave admissible pair, i.e.\
  \[
  (p,q)\ne (4,\infty), \; \frac{2}{p}+\frac{1}{q}\leq \frac12,
  \text{and }s=1-\frac{2}{q}-\frac{1}{p}
  \]

  {\rm i)} Then, we have
  \begin{equation}\label{eq:full-str}
    \|P_k u\|_{L^p_t(\R;L^q_x(\R^2))} \ls 2^{k s} \|P_ku\|_{S^\pm_k}.
  \end{equation}

  {\rm ii)} Moreover, we have
  \begin{equation}\label{eq:full-str-loc}
    \sup_{1\leq l \leq k+10} \Big(\sum_{\kappa \in \mathcal{K}_l} \|P_k P_\kappa u\|_{L^p_t(\R;L^q_x(\R^2))}^2 \Big)^{\frac{1}{2}} \ls 2^{ks} \|P_ku\|_{S^\pm_k}.
  \end{equation}
\end{cor}
\begin{proof}
  It suffices to prove ii). The estimate holds for $P_k P_\kappa u$ in
  the atomic space $U^p_{\pm\la D \ra }$ because it is true for free
  solutions, which follows from $TT^*$ argument and \eqref{eq:bigk},
  hence it holds for $U^p$-atoms. Now, by changing $P_k P_\kappa u$ on
  a set of measure zero, we may assume that $u$ is right-continuous,
  hence the claim follows from $\|P_k P_\kappa u\|_{U^p_{\pm\la D \ra
    }}\ls \|P_k P_\ka u\|_{V^2_{\pm\la D \ra }}$, which holds for any
  $p>2$, see \cite[formula (189)]{stta10}, and \cite[Section 2]{hhk09}
  for more details on these spaces. The claim follows from the
  definition of $\|\cdot \|_{S^\pm_k}$ and
  \[
  \sup_{1\leq l \leq k+10} \sum_{\ka \in \mathcal{K}_l}\|P_\ka f
  \|_{V^2_{\pm \la D \ra }}^2 \ls \sum_{\ka \in \mathcal{K}_{k}}\|
  P_\ka f \|_{V^2_{\pm \la D \ra }}^2,
  \]
  which is obvious.
\end{proof}
Clearly, one can also interpolate the estimates provided by Corollary
\ref{cor:full-str} to obtain all Klein-Gordon admissible pairs (up to
endpoints).
\section{Bilinear and trilinear estimates}\label{sect:bil-est}
In this section we provide the crucial bilinear $L^2_{t,x}$-type estimates for
functions in our spaces. For technical reasons, we also provide some
trilinear estimates at the end of the section.

We use the same convention as in \cite[Section 5]{BH} throughout the
rest of the paper, namely that $u$'s denote scalar-valued
functions $u: \R \times \R^2 \rightarrow \C$, while $\psi$'s
denote vector-valued functions $\psi : \R \times \R^2 \rightarrow
\C^2$. As before, a function $f$ is said to be localized at frequency
$2^k$ if $f=\tilde{P}_kf$ if $k\geq 90$ or $f=P_{\leq 90}f$ if
$k=89$. The first main result in this section is
\begin{pro}\label{Pbil} {\rm i)} For all $k_1\geq 89$ and $k_2\geq
  100$ with $10 \leq |k_1-k_2|$ and $\psi_j \in S^\pm_{k_j}$ localized
  at frequency $2^{k_j}$ for $j=1,2$, the following holds true:
  \begin{equation} \label{bil2} \big\| \la \Pi_\pm(D) \psi_1 , \beta
    \Pi_\pm(D) \psi_2 \ra \big\|_{L^2} \ls 2^{\frac{k_1}2} \| \psi_1
    \|_{S_{k_1}^\pm} \| \psi_2 \|_{S_{k_2}^{\pm,w}}
  \end{equation}

  {\rm ii)} If in addition $l \leq \min(k_1,k_2)+10$, then
\begin{equation} \label{bil3}
     \Bigg\| \sum_{\ka_1,\ka_2 \in \mathcal{K}_{l}: \atop \dist(\pm
        \ka_1, \pm \ka_2)
        \ls 2^{-l}} \la \Pi_\pm(D) \tilde P_{\ka_1} \psi_1, \beta \Pi_\pm(D) \tilde P_{\ka_2}  \psi_2 \ra \Bigg\|_{L^2} 
      \ls 2^{\frac{k_1-l}2} \| \psi_1 \|_{S_{k_1}^\pm} \| \psi_2
      \|_{S_{k_2}^{\pm,w}}.
  \end{equation}

  In both \eqref{bil2} and \eqref{bil3} the sign of each $\pm \ka$ and
  $\Pi_\pm$ is chosen to be consistent with the one of the
  corresponding $S^\pm$.

  iii) In the case $|k_1-k_2| \leq 10$ the above
  \eqref{bil2}-\eqref{bil3} hold true provided the following \emph{parallel
    interaction} term  is subtracted:
  \[
  \sum_{\ka_1, \ka_2 \in \mathcal{K}_{k_2}: \atop \dist(\pm \ka_1, \pm
    \ka_2) \leq 2^{-k_2+3}} \la \Pi_\pm(D) \tilde P_{\ka_1} \psi_1,
  \beta \Pi_\pm(D) \tilde P_{\ka_2} \psi_2 \ra.
  \]

  iv) If $S_{k_2}^{\pm,w}$ is replaced with $S_{k_2}^{\pm}$, then
  \eqref{bil2} and \eqref{bil3} improve as follows:

  - the factor becomes $2^{\frac{\min(k_1,k_2)}2}$, respectively,
  $2^{\frac{\min(k_1,k_2)-l}2}$;

  - they hold for all $k_1,k_2 \geq 89$ (in particular, no terms need
  to be subtracted in the case $|k_1-k_2| \leq 10$).

\end{pro}

\begin{proof}[Proof of Proposition \ref{Pbil}] 
  To make the exposition easier, we choose to prove all the estimates
  for the $+$ choice in all terms.  A careful examination of the
  argument reveals that the other choices follow in a similar manner.
  
  We consider $k_1\geq 89$ and $k_2\geq 100$ and distinguish the
  following three cases: $k_1 \leq k_2-10, |k_1-k_2| \leq 10$ and $k_1
  \geq k_2+10$. We will work out in detail the first case, that is for
  $k_1 \leq k_2-10$. One should also note the close relation between
  these ranges and the ones given by the energy estimates in
  Theorem \ref{thm:Energy}.
    
  We will reduce \eqref{bil2} and \eqref{bil3} to the following claim:
  For all $u_1,u_2$ localized at frequencies $2^{k_1}$, respectively
  $2^{k_2}$, and $l \leq k_1+10$ the following estimate holds true:
  \begin{equation} \label{bas1} \sum_{\ka_1,\ka_2 \in
      \mathcal{K}_{l}:*} \| \tilde P_{\ka_1}u_1 \tilde P_{\ka_2} u_2
    \|_{L^2} \ls 2^{\frac{k_1+l}2} \|u_1 \|_{S^+_{k_1}} \|u_2
    \|_{S_{k_2}^{+,w}},
  \end{equation}
  where $*$ means that the above sum is restricted to the range
  $2^{-l-2} \leq \dist(\ka_1, \ka_2) \leq 2^{-l+2}$ or $\dist( \ka_1,
  \ka_2) \leq 2^{-l+2}$ in the case $l=k_1+ 10$.

  We rely on the following estimate:
  \begin{equation*}
    \sum_{\ka_1,\ka_2 \in \mathcal{K}_{l}: *} \| \tilde P_{\ka_1} u_1 \cdot \tilde P_{\ka_2} u_2  \|_{L^2} 
    \leq A_1+A_2,
  \end{equation*}
  where
  \begin{align*}
    A_1:=& \sum_{\ka_1,\ka_2 \in \mathcal{K}_{l}:*} \| \tilde
    P_{\ka_1} u_1 \|_{L^\infty} \| \tilde P_{\ka_2}
    Q_{\succeq k_2-2l} u_2 \|_{L^2} \\
    \ls{} & 2^{\frac{2k_1-l}2} \Big(\sum_{\ka_1\in \mathcal{K}_{l}} \|
    \tilde P_{\ka_1} u_1 \|_{L^\infty_t L^2_x}^2
    \Big)^{\frac12}\Big(\sum_{\ka_2 \in \mathcal{K}_{l}} \| \tilde
    P_{\ka_2} Q_{\succeq k_2-2l}
    u_2 \|_{L^2}^2 \Big)^{\frac12} \\
    \ls{} & 2^{\frac{2k_1-l}2} \Big(\sum_{\ka_1\in \mathcal{K}_{l}} \|
    \tilde P_{\ka_1} u_1 \|_{L^\infty_t L^2_x} \Big)^{\frac12}
    2^{-\frac{k_2-2l}2} \| Q_{\succeq k_2-2l}
    u_2 \|_{\dot{X}^{+,\frac12,\infty}}\\
    \ls{}& 2^{\frac{k_1+l}2} \| u_1 \|_{S^+_{k_1}} \| u_2
    \|_{S_{k_2}^{+,w}}.
  \end{align*}
  The second term $A_2$, corresponding to the interaction $ \tilde
  P_{\ka_1} u_1 \cdot Q_{\prec k_2-2l} \tilde P_{\ka_2} u_2$, needs
  particular attention. We distinguish three particular scenarios $l
  \leq k_1-11$, $k_1 -10 \leq l \leq k_1+9$ and $l=k_1+10$ and each of
  them is dealt with one of the three energy in frames components in
  the definition of $S^+[k_2,\ka_2]$.
  
  If $l \leq k_1-11$, then we estimate as follows
  \begin{align*}
    A_2:=&  \sum_{\ka_1,\ka_2 \in \mathcal{K}_{l}:*} \sum_{\ka \in \mathcal{K}_{k_1+10}} \| \tilde P_{\ka}  \tilde P_{\ka_1} u_1 \|_{\sum_{ \Lambda_{k_1,\ka}}L^2_{t_{\Th}}L^\infty_{x_{\Th}}}\\ &\cdot \sup_{\Th \in \Lambda_{k_1,\ka}} \| Q_{\prec k_2-2l} \tilde P_{\ka_2} u_2 \|_{L^\infty_{t_{\Th}}L^2_{x_{\Th}}}\\
    \ls{}&\Big(\sum_{\ka_2 \in \mathcal{K}_{l}}\sup_{\ka \in
      \mathcal{K}_{k_1+10}: \atop \ka \cap \ka_1 \ne \emptyset}
    \sup_{\Th \in \Lambda_{k_1,\ka}} \| Q_{\prec k_2-2l} \tilde
    P_{\ka_2} u_2
    \|_{L^\infty_{t_{\Th}}L^2_{x_{\Th}}}^2\Big)^{\frac12}\\
    &\cdot \Big(\sum_{\ka_1 \in \mathcal{K}_{l}} \Big(\sum_{\ka \in \mathcal{K}_{k_1+10}} \| P_{\ka}  \tilde P_{\ka_1} u_1 \|_{\sum_{ \Lambda_{k_1,\ka}}L^2_{t_{\Th}}L^\infty_{x_{\Th}}} \Big)^2\Big)^{\frac12}  \\
    \ls{}& 2^{\frac{k_1-l}2} \| u_1 \|_{S^+_{k_1}} 2^{l} \| u_2
    \|_{S_{k_2}^{+,w}}.
  \end{align*}
  If $k_1 -10 \leq l \leq k_1+9$, then
  \begin{align*}
    A_2:=& \sum_{\ka_1,\ka_2 \in \mathcal{K}_{l}:*} \sum_{\ka \in
      \mathcal{K}_{k_1+10}} \| P_{\ka} \tilde P_{\ka_1} u_1
    \|_{\sum_{\Omega_{k_1,\ka}}L^2_{x^2_{\Th}}L^\infty_{(t,x^1)_{\Th}}}\\
    & \cdot
    \sup_{\Th \in \Omega_{k_1,\ka}} \| Q_{\prec k_2-2l} \tilde P_{\ka_2} u_2 \|_{L^\infty_{t_{\Th}}L^2_{x_{\Th}}}\\
    \ls{}&\Big(\sum_{\ka_2 \in \mathcal{K}_{l}} \sup_{\ka \in
      \mathcal{K}_{k_1+10}: \atop \ka \cap \ka_1 \ne \emptyset}
    \sup_{\Th \in \Omega_{k_1,\ka}} \| Q_{\prec k_2-2l} \tilde
    P_{\ka_2} u_2
    \|_{L^\infty_{t_{\Th}}L^2_{x_{\Th}}}^2\Big)^{\frac12} \\ & \cdot \Big(\sum_{\ka_1 \in \mathcal{K}_{l}}\Big(\sum_{\ka \in \mathcal{K}_{k_1+10}} \| P_{\ka} \tilde P_{\ka_1} u_1 \|_{\sum_{\Omega_{k_1,\ka}}L^2_{t_{\Th}}L^\infty_{x_{\Th}}} \Big)^2\Big)^{\frac12}\\
    \ls{}& 2^{\frac{k_1}2} \| \tilde P_{\ka_1} u_1 \|_{S^+_{k_1}}
    2^{\frac{l}2} \| \tilde P_{\ka_2} u_2 \|_{S_{k_2}^{+,w}}.
  \end{align*}
  If $ l = k_1+10$, we repeat the argument of the first case without
  the additional localization to caps of size $2^{k_1+10}$, and obtain
  \begin{align*}
    A_2:=& \sum_{\ka_1,\ka_2 \in \mathcal{K}_{l}:*} \|\tilde P_{\ka_1} u_1 \|_{\sum_{\Th \in \Lambda_{k_1,\ka_1}}L^2_{t_{\Th}}L^\infty_{x_{\Th}}} \sup_{\Th \in \Lambda_{k_1,\ka_1}} \| Q_{\prec k_2-2l} \tilde P_{\ka_2} u_2 \|_{L^\infty_{t_{\Th}}L^2_{x_{\Th}}}\\
    \ls{}& \| u_1 \|_{S^+_{k_1}} 2^{k_1} \| u_2 \|_{S_{k_2}^{+,w}}.
  \end{align*}

  Obviously, \eqref{bas1} implies
  \begin{equation} \label{bas1a} \sum_{\ka_1,\ka_2 \in
      \mathcal{K}_{l}:*} \| \tilde P_{\ka_1}u_1 \overline{\tilde
      P_{\ka_2} u_2} \|_{L^2} \ls 2^{\frac{k_1+l}2} \| u_1
    \|_{S^+_{k_1}} \| u_2 \|_{S_{k_2}^{+,w}}.
  \end{equation}

  Now, we turn to the proof of \eqref{bil3}.  Using \eqref{bas1a} we
  claim the following
  \begin{equation} \label{basnul}\begin{split} &\sum_{\ka_1,\ka_2 \in
        \mathcal{K}_{l}: *}
      \| \la \Pi_+(D)  P_{\ka_1} \psi_1, \beta \Pi_+(D) P_{\ka_2} \psi_2 \ra \|_{L^2}\\
      \ls{}& 2^{\frac{k_1-l}2} \| \Pi_+(D)\psi_1 \|_{S^+_{k_1}} \|
      \Pi_+(D) \psi_2 \|_{S_{k_2}^{+,w}}.
    \end{split}
  \end{equation} To prove \eqref{basnul}, we linearize
  the operator $\Pi_+(D)$ as follows
  \[
  \Pi_+(D) = \Pi_+(2^{k_j} \omega(\ka_j)) + \Pi_+(D) - \Pi_+(2^{k_j}
  \omega(\ka_j))
  \]
  where $j=1,2$. Taking into account \eqref{bas1a} and \eqref{PiPi} we
  obtain
  \begin{align*}
    &\sum_{\ka_1,\ka_2 \in \mathcal{K}_{l}: *} \| \la \Pi_+(2^{k_1}
    \omega(\ka_1)) P_{\ka_1} \psi_1 , \beta \Pi_+(2^{k_2}
    \omega(\ka_2)) P_{\ka_2} \psi_2 \ra
    \|_{L^2} \\
    \ls{}& 2^{\frac{k_1-l}2} \| \psi_1 \|_{S^+_{k_1}} \| \psi_2
    \|_{S_{k_2}^{+,w}}
  \end{align*}
  where we have used $|\angle(\omega(\ka_1),\omega(\ka_2)) | \ls
  2^{-l}$ and that $  \mathcal{O}(2^{-k_1} + 2^{-k_2}) \ls 2^{-k_1} \ls 2^{-l}$.
  
  The estimate for the remaining terms follows from using
  \eqref{bas1a} and \eqref{sta}.  Now, we use
  \begin{align*}
    &\| \la \Pi_+(D) \psi_1 , \beta \Pi_+(D) \psi_2 \ra \|_{L^2} \\
    \ls{}& \sum_{1\leq l \leq k_1+10} \sum_{\ka_1,\ka_2 \in
      \mathcal{K}_{l}:*} \| \la P_{\ka_1} \Pi_+(D) \psi_1 , \beta
    P_{\ka_2} \Pi_+(D) \psi_2 \ra \|_{L^2},
  \end{align*}
  and \eqref{basnul} and observe that the summation with respect to
  $l$ is performed using the factor of $2^{-\frac{l}2}$.
  
  This finishes the proof of i) and ii) in the case $k_1 \leq
  k_2-10$. The proof of \eqref{bas1} in the case $k_1 \geq k_2+10$ is
  similar in the case $l \leq k_2-11$ and $l=k_2+10$, and also in the
  case $k_2-10\leq l \leq k_2+9$ for the contributions $A_1$. In the
  case of $A_2$, we modify the argument as in \cite[Prop.~5.1]{BH}: We
  decompose
  \[
  \tilde P_{\ka_1} u_1 = \sum_{\ka \in \mathcal{K}_{k_1+10}} P_{\ka}
  \tilde P_{\ka_1} u_1
  \]
  and note that the interactions $P_{\ka} \tilde P_{\ka_1} u_1 \tilde
  P_{\ka_2} u_2$ are almost orthogonal with respect to $\ka \in
  \mathcal{K}_{k_1+10}$, which follows from the fact that both
  $P_{\ka} \tilde P_{\ka_1} u_1$ and $ \tilde P_{\ka_2} u_2$ have
  Fourier-support of size $\approx 1$ in the direction orthogonal to
  $\omega(\ka_2)$. As a consequence
  \[
    \| \tilde P_{\ka_1} u_1 \cdot \tilde P_{\ka_2} Q_{\prec
      k_2-2l} u_2 \|_{L^2}^2
    \ls{} \sum_{\ka \in \mathcal{K}_{k_1+10}} \| P_{\ka} \tilde
    P_{\ka_1} u_1 \cdot \tilde P_{\ka_2} Q_{\prec k_2-2l} u_2
    \|_{L^2}^2
  \]
  and we can proceed as before.

  The proof in the case $|k_1-k_2| \leq 10$ is similar, except that
  there there is no mechanism to deal with the \emph{parallel
    interactions}
  \[
  \sum_{\ka_1,\ka_2 \in \mathcal{K}_{k_2}: \atop \dist(\ka_1,\ka_2)
    \leq 2^{-k_2+3}} \la \tilde P_{\ka_1} u_1, \tilde P_{\ka_2} u_2
  \ra
  \]
  in \eqref{bas1}. This is the reason we cannot estimate this term and
  claim only the equivalent of \eqref{bil2}-\eqref{bil3} which
  excludes it.

  Finally, the improvement in iv) is justified as follows: Since both
  terms are in $S_k^+$ type spaces, by symmetry reasons we can replace
  $2^{\frac{k_1}2}$ by $2^{\frac{\min(k_1,k_2)}2}$ in \eqref{bil2} and
  similarly in \eqref{bil3}. If $89\leq k_1,k_2\leq 100$ we simply use
  the $L^4$-Strichartz bound on both functions. In the other cases
  where $|k_1-k_2| \leq 10$ we use the fact that in $S_k^+$ we have
  access to the full family of Strichartz estimates for both terms and
  we estimate the \emph{parallel interactions} term as follows:
  \begin{align*}
     \Big\| \sum_{\ka_1,\ka_2 \in \mathcal{K}_{k_2}: \atop
      \dist(\ka_1,\ka_2) \leq 2^{-k_2+3}}
    \la  \tilde P_{\ka_1} u_1, \tilde P_{\ka_2}  u_2 \ra \Big\| 
    &\ls  \sum_{\ka_1,\ka_2 \in \mathcal{K}_{k_2}: \atop
      \dist(\ka_1,\ka_2) \leq 2^{-k_2+3}}
    \| \tilde P_{\ka_1} u_1\|_{L^4} \| \tilde P_{\ka_2}  u_2 \|_{L^4} \\
   & \ls  \big( \sum_{\ka_1 \in \mathcal{K}_{k_2}} \| \tilde P_{\ka_1}
    u_1\|_{L^4}^2 \big)^\frac12
    \big( \sum_{\ka_2 \in \mathcal{K}_{k_2}} \|  \tilde P_{\ka_2}  u_2 \|^2_{L^4} \big)^\frac12 \\
    & \ls  2^{k_2} \| u_1 \|_{S^+_{k_1}} \| u_2 \|_{S^+_{k_2}}.
  \end{align*}
  This matches the numerology claimed in \eqref{bas1} and adds up
  correctly with the other angular interactions to give \eqref{bil2}
  and \eqref{bil3}.
\end{proof}

\begin{rmk}\label{rmk:interp}
  The estimates of Proposition \ref{Pbil} can be interpolated with the
  trivial estimate
  \[
  \||\psi_1||\psi_2|\|_{L^\infty_tL^2_x}\ls 2^{k_1}
  \|\psi_1\|_{L^\infty_tL^2_x}\|\psi_2\|_{L^\infty_tL^2_x}
  \]
  obtain by the Bernstein inequality.  In particular, for $2\leq r
  \leq \infty$ we obtain
  \begin{equation}\label{eq:interp-bil2}
    \big\| \la \Pi_\pm(D) \psi_1 , \beta
    \Pi_\pm(D) \psi_2 \ra \big\|_{L^r_t L^2_x} \ls 2^{k_1(1-\frac{1}{r})} \| \psi_1
    \|_{S_{k_1}^\pm} \| \psi_2 \|_{S_{k_2}^{\pm}}.
  \end{equation}
\end{rmk}

We finish this section with
two trilinear estimates.

\begin{lem}\label{lem:TRI1} Assume $k_1 \leq k_2 \leq k_3$ and each $\psi_i$ is supported at frequency $2^{k_i}, i=1,2,3$. 
  The following estimate holds true for any $\frac43 < p \leq 2$ and
  any choice of signs $s_i \in \{\pm\}, i = 1,2,3$:
  \begin{equation} \label{TRI1}
    \begin{split}
      & 2^{(\frac1p - \frac12)k_3}
      \|  \la \Pi_{s_1}(D) \psi_1, \beta \Pi_{s_2}(D) \psi_2 \ra \beta \Pi_{s_3}(D) \psi_3 \|_{L^p_t L^2_x} \\
      \ls & 2^{(\frac38 - \frac1{2p})(k_1-k_2)}
      2^{(1-\frac1p)(k_2-k_3)} \prod_{j=1}^3 2^{\frac{k_j}2} \| \psi_j
      \|_{S^{s_j}_{k_j}}.
    \end{split}
  \end{equation}
\end{lem}

\begin{proof} The strategy is to recombine $\psi_1$ and $\psi_3$ or
  $\psi_2$ and $\psi_3$ and provide an $L^2$ type estimate as in
  \eqref{Pbil}.
A careful analysis reveals that one can
  still extract gains from the null structure when recombining terms.

  We provide a complete argument for the $\Pi_+(D)$ part of each term,
  that is we assume $\psi_i = \Pi_+(D) \psi_i, \forall i \in
  \{1,2,3\}$.  A similar argument works for the other combinations.
  Fix $0 \leq l \leq k_1+10$ and write
  \[
  I=\sum_{\ka_1, \ka_2 \in \Ka_l: *} \la \tilde P_{\ka_1} \psi_1,
  \beta \tilde P_{\ka_2} \psi_2 \ra \beta \psi_{3}
  \]
  where $*$ indicates that we consider the range $2^{-l+3} \leq
  \dist(\ka_1,\ka_2) \leq 2^{-l+6}$, if $l < k_1+10$, or $\dist(
  \ka_1, \ka_2) \leq 2^{-l+6}$ in the case $l=k_1+ 10$.
 
  Let $l < k_1+10$. Fix $\ka_1,\ka_2 \in \Ka_l$ subject to $*$. We
  explain now how to take advantage of the null condition in this
  context. For $j=1,2$ we decompose
  \[
  \Pi_+(D) = \Pi_+(2^{k_j} \om(\ka_j)) + \Pi_+(D) - \Pi_+(2^{k_j}
  \om(\ka_j))
  \]
  and use \eqref{PiPi} and \eqref{sta} to extract a factor of $2^{-l}$
  from the expression $\la \tilde P_{\ka_1} \psi_1, \beta \tilde
  P_{\ka_2} \psi_2 \ra$ in all the computations below. To keep things
  simple in the estimates below, we skip the step where each $\psi_j,
  j=1,2$ goes through the above decomposition and simply just book the
  factor of $2^{-l}$.
 
  We start with the high modulation component of $\psi_3$ which we
  estimate as follows
  \begin{equation} \label{highp3}
    \begin{split}
      & \| \la \tilde P_{\ka_1}  \psi_1, \beta \tilde P_{\ka_2} \psi_2 \ra \beta Q_{\succeq k_3-2l} \psi_{3} \|_{L^{p}_t L^2_x} \\
      \ls & \ 2^{-l} \|  \tilde P_{\ka_1}  \psi_1 \|_{L^\infty_{t,x}} \| \tilde P_{\ka_2} \psi_2 \|_{L_t^{\frac{2p}{2-p}} L^\infty_x} \| Q_{\succeq k_3-2l} \psi_{3} \|_{L^2_{t,x}} \\
      \ls & \ 2^{\frac{2k_1-l}2} \| \tilde P_{\ka_1} \psi_1
      \|_{\lS^+_{k_1}} 2^{(1+\frac{p-2}{2p})k_2} \| \tilde P_{\ka_2}
      \psi_2 \|_{\lS^+_{k_2}} 2^{-\frac{k_3}2} \| \psi_3 \|_{V^2_{+
          \la D\ra }}.
    \end{split}
  \end{equation}
  For the low modulation component, we decompose
  \begin{equation} \label{decp3}
    \begin{split}
      \psi_{3} = \sum_{l' < l-8} \sum_{\ka_3 \in \Ka_{l'}^{**}}
      P_{\ka_3} \psi_{3} + \sum_{\ka_3 \in \Ka_{l}^{***}} P_{\ka_3}
      \psi_{3}
    \end{split}
  \end{equation}
  where if $\ka_3 \in \Ka_{l'}^{**}$, $d(\ka_3,\ka_1) \approx
  d(\ka_3,\ka_2) \approx 2^{-l'}$, while if $\ka_3 \in \Ka_{l}^{***}$,
  $d(\ka_3,\ka_1) + d(\ka_3,\ka_2) \leq 2^{-l+10}$.  Fix $l' <
  l-8$. Using \eqref{bas1} we estimate
  \begin{align*}
    & \| \la \tilde P_{\ka_1} \psi_1, \beta \tilde P_{\ka_2} \psi_2
    \ra \beta
    \sum_{\ka_3 \in \Ka_{l'}^{**}} P_{\ka_3} Q_{\prec k_3-2l}  \psi_{3} \|_{L^p_tL^2_x}  \\
    \ls & 2^{-l} \| \tilde P_{\ka_1}
    \psi_1\|_{L^{\frac{2p}{2-p}}_tL^\infty_x} \cdot \Big\| \tilde P_{\ka_2} \psi_2  \sum_{\ka_3 \in \Ka_{l'}^{**}}  P_{\ka_3} Q_{\prec k_3-2l}  \psi_3 \Big\|_{L^2}\\
    \ls & 2^{-l} 2^{(1+\frac{p-2}{2p})k_1} \| \tilde P_{\ka_1} \psi_1
    \|_{\lS^+_{k_1}} 2^{\frac{k_2+l'}{2}} \| \tilde P_{\ka_2} \psi_2
    \|_{\lS^+_{k_2}} \| \psi_3 \|_{S^+_{k_3}},
  \end{align*}
  since it follows from the proof of \eqref{bas1} that the operator
  $Q_{\prec k_3-2l}$ is disposable and we only need the $l^2S_{k_2}$
  component for $\tilde P_{\ka_2}\psi_2$.

  For the second sum, where $\ka_3 \in \Ka_l^{***}$, the key property
  is that $2^{-l+1} \leq \dist(\ka_3,\ka_1) + \dist(\ka_3,\ka_2) \ls
  2^{-l}$.  Thus we can split the set $\Ka_l^{***}=S_1 \cup S_2, S_1
  \cap S_2 = \emptyset$ such that $\ka_3 \in S_1$ satisfies
  $\dist(\ka_3,\ka_1) \geq 2^{-l}$, while $\ka_3 \in S_2$ satisfies
  $\dist(\ka_3,\ka_2) \geq 2^{-l}$.

  The part of the sum with $\ka_3 \in S_2$ is estimated as above with
  $l'=l$, thus leading to
  \begin{align*}
    & \| \la \tilde P_{\ka_1} \psi_1, \beta \tilde P_{\ka_2} \psi_2
    \ra \beta
    \sum_{\ka_3 \in  S_2} P_{\ka_3} Q_{\prec k_3-2l}  \psi_{3} \|_{L^p_tL^2_x} \\
    \ls & 2^{-l} 2^{(1+\frac{p-2}{2p})k_1} \| \tilde P_{\ka_1} \psi_1
    \|_{\lS^+_{k_1}} 2^{\frac{k_2+l}{2}} \| \tilde P_{\ka_2} \psi_2
    \|_{\lS^+_{k_2}} \| \psi_3 \|_{S^+_{k_3}}.
  \end{align*}

  The part of the sum with $\ka_3 \in S_1$ is estimated as follows
  \begin{align*}
    & \| \la \tilde P_{\ka_1} \psi_1, \beta \tilde P_{\ka_2} \psi_2
    \ra \beta
    \sum_{\ka_3 \in S_1} P_{\ka_3} Q_{\prec k_3-2l}  \psi_{3} \|_{L^p_tL^2_x}  \\
    \ls & 2^{-l} \| \tilde P_{\ka_2} \psi_2
    \|_{L^{\frac{8p}{4-p}}_tL^\infty_x}
    \cdot \Big\| \tilde P_{\ka_1} \psi_1  \sum_{\ka_3 \in S_1}  P_{\ka_3} Q_{\prec k_3-2l}  \psi_3 \Big\|_{L^{\frac{8p}{4+p}}_tL^2_x}\\
    \ls & 2^{-l} 2^{(\frac98-\frac1{2p})k_2} \| \tilde P_{\ka_2}
    \psi_2 \|_{\lS^+_{k_2}} 2^{(\frac1p+\frac14)\frac{k_1+l}{2}}
    2^{(\frac34 - \frac1p)k_1} \| \tilde P_{\ka_1} \psi_1
    \|_{\lS^+_{k_1}} \| \psi_3 \|_{S^+_{k_3}}.
  \end{align*}
  The last inequality was obtained by interpolating between the two
  estimates
  \[
  \| \tilde P_{\ka_1} \psi_1 \sum_{\ka_3 \in S_1} P_{\ka_3} Q_{\prec
    k_3-2l} \psi_3 \Big\|_{L^2_tL^2_x} \ls 2^{\frac{k_1+l}2} \| \tilde
  P_{\ka_1} \psi_1 \|_{\lS^+_{k_1}} \| \psi_3 \|_{S^+_{k_3}},
  \]
  \[
  \| \tilde P_{\ka_1} \psi_1 \sum_{\ka_3 \in S_1} P_{\ka_3} Q_{\prec
    k_3-2l} \psi_3 \Big\|_{L^\infty_t L^2_x} \ls 2^{k_1} \| \tilde
  P_{\ka_1} \psi_1 \|_{\lS^+_{k_1}} \| \psi_3 \|_{S^+_{k_3}},
  \]
  where the first one follows from \eqref{bas1} and its proof, while
  the second one follows from the trivial estimate $\| \tilde
  P_{\ka_1} \psi_1 \|_{L^\infty_{t,x}} \ls 2^{k_1} \| \tilde P_{\ka_1}
  \psi_1 \|_{L^\infty_t L^2_x}$.

  Bringing together the two inequalities we obtain:
  \begin{align*}
    & \| \la \tilde P_{\ka_1} \psi_1, \beta \tilde P_{\ka_2} \psi_2
    \ra \beta
    \sum_{\ka_3 \in \Ka_l^{***}} P_{\ka_3} Q_{\prec k_3-2l}  \psi_{3} \|_{L^p_tL^2_x}  \\
    \ls & 2^{-\frac{l}2} 2^{(\frac78 - \frac1{2p})k_1} \| \tilde
    P_{\ka_1} \psi_1 \|_{\lS^+_{k_1}} 2^{(\frac98-\frac1{2p})k_2} \|
    \tilde P_{\ka_2} \psi_2 \|_{\lS^+_{k_2}} \| \psi_3 \|_{S^+_{k_3}},
  \end{align*}
  At this time we can perform the summation with respect to the
  decomposition of $\psi_3$ in \eqref{decp3} to obtain:
  \begin{align*}
    & \|  \la \tilde P_{\ka_1}  \psi_1, \beta \tilde P_{\ka_2} \psi_2 \ra \beta  Q_{\prec k_3-2l}  \psi_{3} \|_{L^p_tL^2_x}  \\
    \ls & 2^{-\frac{l}2} 2^{(\frac78 - \frac1{2p})k_1} \| \tilde
    P_{\ka_1} \psi_1 \|_{\lS^+_{k_1}} 2^{(\frac98-\frac1{2p})k_2} \|
    \tilde P_{\ka_2} \psi_2 \|_{\lS^+_{k_2}} \| \psi_3 \|_{S^+_{k_3}},
  \end{align*}
  To this estimate we add the high modulation component estimate in
  \eqref{highp3} to conclude with
  \begin{align*}
    & \|  \la \tilde P_{\ka_1}  \psi_1, \beta \tilde P_{\ka_2} \psi_2 \ra \beta  \psi_{3} \|_{L^p_tL^2_x}  \\
    \ls & 2^{-\frac{l}2} 2^{(\frac78 - \frac1{2p})k_1} \| \tilde
    P_{\ka_1} \psi_1 \|_{\lS^+_{k_1}} 2^{(\frac98-\frac1{2p})k_2} \|
    \tilde P_{\ka_2} \psi_2 \|_{\lS^+_{k_2}} \| \psi_3 \|_{S^+_{k_3}},
  \end{align*}

  The cap summation with respect to $\ka_1,\ka_2 \in \Ka_l:*$ is
  performed using the $l^2$ property of the $l^2S_k$ spaces
  \eqref{l2S}:
  \begin{align*}
    & \|  \sum_{\ka_1,\ka_2 \in \Ka_l:*} \la \tilde P_{\ka_1}  \psi_1, \beta \tilde P_{\ka_2} \psi_2 \ra \beta  \psi_{3} \|_{L^p_tL^2_x}  \\
    \ls & 2^{-\frac{l}2} 2^{(\frac78 - \frac1{2p})k_1} \| \psi_1
    \|_{\lS^+_{k_1}} 2^{(\frac98-\frac1{2p})k_2} \| \psi_2
    \|_{\lS^+_{k_2}} \| \psi_3 \|_{S^+_{k_3}},
  \end{align*}
  Recall that up to this point we have used that $l < k_1+10$. If
  $l=k_1+10$, then one proceeds as above up to the point where we
  split the set $\Ka_l^{***}=S_1 \cup S_2$. The modification in this
  case is that we simply retain only the $S_1$ component which is now
  characterized by $d(\ka_3,\ka_1) \ls 2^{-k_1}$ and estimate as above
  to obtain
  \begin{align*}
    & \|  \sum_{\ka_1,\ka_2 \in \Ka_l:*} \la \tilde P_{\ka_1}  \psi_1, \beta \tilde P_{\ka_2} \psi_2 \ra \beta  \psi_{3} \|_{L^p_tL^2_x}  \\
    \ls & 2^{-\frac{l}2} 2^{(\frac78 - \frac1{2p})k_1} \| \psi_1
    \|_{S^+_{k_1}} 2^{(\frac98-\frac1{2p})k_2} \| \psi_2
    \|_{S^+_{k_2}} \| \psi_3 \|_{S^+_{k_3}},
  \end{align*}
  where $l=k_1+10$. Finally, the summation with respect to $l$ is done
  using the factor $2^{-\frac{l}{2}}$:
  \begin{align*}
    \| \la \psi_1, \beta \psi_2 \ra \beta \psi_{3} \|_{L^p_tL^2_x}
    & \ls  2^{(\frac78 - \frac1{2p})k_1} \| \psi_1 \|_{S_{k_1}}   2^{(\frac98-\frac1{2p})k_2} \| \psi_2 \|_{S_{k_2}}   \| \psi_3 \|_{S_{k_3}} \\
    & \ls 2^{(\frac38 - \frac1{2p})k_1} 2^{(\frac58-\frac1{2p})k_2}
    2^{-\frac{k_3}2} \prod_{j=1}^3 2^{\frac{k_j}2} \| \psi_j
    \|_{S^{+}_{k_j}}
  \end{align*}
  from which \eqref{TRI1} follows.
\end{proof}

\begin{lem}\label{lem:TRI2} Assume $ k_1 \leq \min(k_2,k_3)$ and each
  $\psi_i$ is supported at frequency $2^{k_i}, i=1,2,3$. For any $2
  \leq p \leq \infty$ and any choice of signs $s_i \in \{\pm\}, i =
  1,2,3$, the following estimate holds true:
   \begin{equation} \label{tris2} \begin{split} &\| \Pi_{s_1}(D) \psi_1
      \la \Pi_{s_2}(D) \psi_2,\beta \Pi_{s_3}(D) \psi_3 \ra
      \|_{L^p_t L^1_x} \\
      \ls{}& 2^{(1-\frac1p)k_1} \| \psi_1 \|_{S^{s_1}_{k_1}} \| \psi_2
      \|_{S^{s_2}_{k_2}} \| \psi_3 \|_{S_{k_3}^{s_3,w}}.
    \end{split}
  \end{equation}

  \begin{proof} Note that \eqref{tris2} follows
    from 
    \begin{equation} \label{tris}
     \| \Pi_{s_1}(D) \psi_1 \la \Pi_{s_2}(D) \psi_2,\beta \Pi_{s_3}(D) \psi_3 \ra
        \|_{L^2_t L^1_x}
        \ls 2^{\frac{k_1}2} \| \psi_1 \|_{S^{s_1}_{k_1}} \| \psi_2
        \|_{S^{s_2}_{k_2}} \| \psi_3 \|_{S_{k_3}^{s_3,w}},
    \end{equation}
    by interpolating with the trivial estimate:
    \begin{align*}
      \| \psi_1 \la \psi_2,\beta\psi_3 \ra \|_{L^\infty_t L^1_x} & \ls
      \| \psi_1 \|_{L^\infty_{t,x}} \| \psi_2 \|_{L^\infty_t L^2_x} \|
      \psi_3
      \|_{L^\infty_t L^2_x} \\
      & \ls 2^{k_1} \| \psi_1 \|_{S^{s_1}_{k_1}} \| \psi_2
      \|_{S^{s_2}_{k_2}} \| \psi_3 \|_{S^{s_3,w}_{k_3}}.
    \end{align*}
    Therefore the rest of this proof is concerned with
    \eqref{tris}. The argument carries some similarities with the one
    used in Lemma \ref{lem:TRI1}.  In particular we extract the gains
    from the null condition as explained in the body of that proof and
    skip the formalization here. We provide a complete argument for
    the $\Pi_+(D)$ part of each term, that is we assume $\psi_i =
    \Pi_+(D) \psi_i, \forall i \in \{1,2,3\}$.  A similar argument
    works for the other combinations.
    We decompose
    \begin{equation} \label{dec15} \psi_1 \la \psi_2, \beta \psi_3 \ra
      = \sum_{0 \leq l \leq k+10} \sum_{\ka_1 \in \Ka_l} P_{\ka_1}
      \psi_1 \sum_{i=2}^3 \sum_{\ka_2, \ka_3 \in \Ka^2_{l}(\ka_1,i)}
      \la P_{\ka_2} \psi_2, \beta P_{\ka_3} \psi_3 \ra
    \end{equation}
    where $\Ka_{l}^2(\ka_1,2)=\{(\ka_2,\ka_3) \in \Ka_l \times \Ka_l:
    2^{-l+3} \leq d(\ka_1,\ka_2) \leq 2^{-l+6}, d(\ka_1,\ka_3) \leq
    2^{-l+6} \}$ and $\Ka_{l}^2(\ka_1,3)=\{(\ka_2,\ka_3) \in \Ka_l
    \times \Ka_l: 2^{-l+3} \leq d(\ka_1,\ka_3) \leq 2^{-l+6},
    d(\ka_1,\ka_2) \leq 2^{-l+6} \}$ for $l < k_1+10$ while for
    $l=k_1+10$ we pick $\Ka^2_{l}(\ka_1,2)=\Ka^2_{l}(\ka_1,3)=
    \{(\ka_2,\ka_3) \in \Ka_l \times \Ka_l: d(\ka_1,\ka_3) \leq
    2^{-l+6}, d(\ka_1,\ka_2) \leq 2^{-l+6} \}$. As defined, these sets
    are not disjoint, so we (implicitly) remove elements which are
    counted multiple times.
 
    We fix $0 \leq l < k_1+10$, $\ka_1 \in \Ka_l$ and aim to estimate
    \[
    \sum_{\ka_1 \in \Ka_l} P_{\ka_1} \psi_1 \sum_{\ka_2, \ka_3 \in
      \Ka^2_{l}(\ka_1,2)} \la P_{\ka_2} \psi_2, \beta P_{\ka_3} \psi_3
    \ra
    \]
    Notice that, given the structure of the set $\Ka^2_{l}(\ka_1,2)$,
    for all $\ka_2,\ka_3 \in \Ka^2_{l}(\ka_1,2)$ we have
    $d(\ka_2,\ka_3) \ls 2^{-l}$ and this allows us to book the gain of
    $2^{-l}$ from the null condition as explained in Lemma
    \ref{lem:TRI1}. Combining this with the fact that in the above sum
    we have $2^{-l+3} \leq d(\ka_1,\ka_2) \leq 2^{-l+6}$ we invoke
    \eqref{bas1} to obtain
    \[
    \begin{split}
      & \| \sum_{\ka_1 \in \Ka_l} P_{\ka_1} \psi_1 \sum_{\ka_2, \ka_3
        \in \Ka^2_{l}(\ka_1,2)} \la P_{\ka_2} \psi_2, \beta
      P_{\ka_3} \psi_3 \ra \|_{L^2_t L^1_x} \\
      \ls & 2^{-l} 2^{\frac{k_1+l}2}
      \| \psi_1 \|_{S^+_{k_1}}  \| \psi_2 \|_{S^+_{k_2}}  \sup_{\ka_3} \| P_{\ka_3} \psi_3 \|_{L^\infty_t L^2_x} \\
      \ls & 2^{\frac{k_1-l}2} \| \psi_1 \|_{S^+_{k_1}} \| \psi_2
      \|_{S^+_{k_2}} \| \psi_3 \|_{S_{k_3}^{+,w}}.
    \end{split}
    \]
    A similar argument gives
    \[
      \| \sum_{\ka_1 \in \Ka_l} P_{\ka_1} \psi_1 \sum_{\ka_2, \ka_3
        \in \Ka^2_{l}(\ka_1,3)} \la P_{\ka_2} \psi_2, \beta
      P_{\ka_3} \psi_3 \ra \|_{L^2_t L^1_x} 
      \ls  2^{\frac{k_1-l}2} \| \psi_1 \|_{S^+_{k_1}} \| \psi_2
      \|_{S^+_{k_2}} \| \psi_3 \|_{S_{k_3}^{+,w}}.
    \]
    If $l=k_1+10$ then we proceed as above in the case of
    $\Ka^2_{l}(\ka_1,2)$ since $\psi_2$ comes with the stronger
    structure $S^+_{k_2}$.

    To conclude with \eqref{tris} we need to perform the summation
    with respect to $l$ in \eqref{dec15}; this is trivially done using
    the power of $2^{-\frac{l}2}$.
  \end{proof}
  
\end{lem}

 \section{The Dirac nonlinearity}\label{sect:dirac}
 The main result of this section is the following
 \begin{thm} \label{thm:thnon} Choose $s_1,s_2,s_3,s_4 \in \{ +,
   -\}$. Then, for all $\psi_k \in S^{s_k,\frac12}$ satisfying
   $\psi_k=\Pi_{s_k}(D) \psi_k$ for $k=1,2,3$, we have
   \begin{equation} \label{cun} \| \Pi_{s_4}(D) ( \la \psi_1, \beta
     \psi_2 \ra \beta \psi_3) \|_{N^{s_4,\frac12}} \ls \|
     \psi_1\|_{S^{s_1,\frac12}} \| \psi_2 \|_{S^{s_2,\frac12}} \|
     \psi_3 \|_{S^{s_3,\frac12}}.
   \end{equation}
 \end{thm}
 The rest of this section is devoted to the proof of Theorem
 \ref{thm:thnon} and the proof of our main result Theorem
 \ref{thm:main}, which is organized similarly to \cite[Section 6]{BH}.
 The estimate \eqref{cun} will be derived from similar estimates for
 frequency localized functions. Our aim will be to identify a function
 $G(\mathbf{k}) : \N^4_{\geq 89} \rightarrow (0,\infty)$ such that
 \begin{equation} \label{G} \sum_{k_1,k_2,k_3,k_4 \in \N_{\geq 89}}
   G(\mathbf{k}) a_{k_1} b_{k_2} c_{k_3} d_{k_4} \ls \| a \|_{l^2} \|
   b \|_{l^2} \| c \|_{l^2} \| d \|_{l^2}
 \end{equation}
 for all sequences $a=(a_j)_{j \in \N_{\geq 89}}$, etc, in $l^2$.
 Here, we set $\N_{\geq 89}=\{ n \in \N : n \geq 89 \}$ and write
 $\mathbf{k}=(k_1,k_2,k_3,k_4)$.

 With these notations, the result of Theorem \ref{thm:thnon} follows
 from
 \begin{pro} \label{pro:thnon} There exists a function $G$ satisfying
   \eqref{G} such that if $\psi_j$ are localized at frequency
   $2^{k_j}$, $k_j \geq 89$ and $\psi_j = \Pi_{s_j}(D) \psi_j$ for
   $j=1,2,3$, then the following holds true
   \begin{equation} \label{cunn} 2^{\frac{k_4}2} \| P_{k_4}
     \Pi_{s_4}(D) ( \la \psi_1, \beta \psi_2 \ra \beta \psi_3)
     \|_{N^{s_4}_{k_4}} \ls G(\mathbf{k}) \prod_{j=1}^3
     2^{\frac{k_j}2} \| \psi_j \|_{S^{s_j}_{k_j}},
   \end{equation}
   for any choice of sign $s_1,s_2,s_3,s_4 \in \{ +, -\}$.
 \end{pro}
 We break this down into two building blocks:
 \begin{lem}\label{lem:cunn1}
   Under the assumptions of Proposition \ref{pro:thnon} the following
   estimate holds true for any $\frac43 < p \leq 2$:
   \begin{equation} \label{cunn1} 2^{(\frac1p - \frac12)k_4} \|
     P_{k_4}\Pi_{s_4}(D) ( \la \psi_1, \beta \psi_2 \ra \beta \psi_3)
     \|_{L^p_t L^2_x} \ls G(\mathbf{k}) \prod_{j=1}^3 2^{\frac{k_j}2}
     \| \psi_j \|_{S^{s_j}_{k_j}}.
   \end{equation}
 \end{lem}

\begin{lem}\label{lem:cunn23}
  Under the assumptions of Proposition \ref{pro:thnon} (including now
  that $\psi_4$ are localized at frequency $2^{k_4}$ and $\psi_4 =
  \Pi_{s_4}(D) \psi_4$) the following estimate hold true:
  \begin{equation} \label{cunn2}
      \Big| \int  \la  \psi_1, \beta \psi_2 \ra \cdot \la \psi_3, \beta \psi_4 \ra dx dt\Big|
      \ls G(\mathbf{k})\prod_{j=1}^3 2^{\frac{k_j}2} \| \psi_j
      \|_{S^{s_j}_{k_j}} \cdot 2^{-\frac{k_4}2} \| \psi_4
      \|_{S^{s_4,w}_{k_4}}.
  \end{equation}
  \end{lem}

Next, we show how Lemmas \ref{lem:cunn1} and \ref{lem:cunn23} imply
Proposition \ref{pro:thnon}.
\begin{proof}[Proof of Prop. \ref{pro:thnon}]
  The estimate \eqref{cunn1} provides the $L^p_t L^2_x$ part of
  \eqref{cunn}. Next, we explain why \eqref{cunn2} implies the atomic
  part of \eqref{cunn}. The nonlinearity \[\mathcal{N}= P_{k_4}
  \Pi_{s_4}(D) ( \la \psi_1, \beta \psi_2 \ra \beta \psi_3)\]
  satisfies $\mathcal{N}=\tilde{P}_{k_4}\Pi_{s_4}(D) \mathcal{N}$ and
  has to be estimated in $N^{s_4}_{k_4}$. Using the duality
  \eqref{dual}, it suffices to test $\mathcal{N}$ against $\psi_4\in
  S^{s_4,w}_{k_4}$ and to prove the estimate
  \begin{equation} \label{dualin}\Big| \int \la P_{k_4}\Pi_{s_4}(D)
    \mathcal{N}, \psi_4 \ra dx dt \Big|\ls G(\mathbf{k}) \prod_{j=1}^3
    2^{\frac{k_j}2} \| \psi_j \|_{S^{s_j}_{k_j}} \cdot
    2^{-\frac{k_4}2} \| \psi_4 \|_{S^{s_4,w}_{k_4}}.
  \end{equation}
  We have
  \begin{align*}
    \int \la \mathcal{N}, \psi_4 \ra dx dt
    &  = \int \la \la  \psi_1, \beta \psi_2 \ra \beta \psi_3, \Pi_{s_4}(D) P_{k_4}\psi_4 \ra dx dt \\
    & = \int \la \psi_1, \beta \psi_2 \ra \la \psi_3, \beta
    \Pi_{s_4}(D) P_{k_4}\psi_4 \ra dx dt.
  \end{align*}
  Now, we split $\psi_j = \Pi_{+}(D) \psi_j + \Pi_{-}(D) \psi_j$, and
  each contribution to \eqref{dualin} is bounded by \eqref{cunn2}.
\end{proof}

\begin{proof}[Proof of Lemma \ref{lem:cunn1}] We will use the
  notation:
  \[
  TR= 2^{\frac{k_1}2} \| \psi_1 \|_{S^{s_1}_{k_1}} 2^{\frac{k_2}2} \|
  \psi_2 \|_{S^{s_2}_{k_2}} 2^{\frac{k_3}2} \| \psi_3
  \|_{S^{s_3}_{k_3}}.
  \]
  The argument is symmetric with respect to $k_1,k_2$, hence we can
  simply assume that $k_1 \leq k_2 $.
  
  We first consider the case $k_3 \leq k_1+20$, in which case $k_4
  \leq k_2+30$ or else the l.h.s.\ of \eqref{cunn1} vanishes. Using
  Strichartz and Prop.\ \ref{Pbil}, we obtain
  \[
  \begin{split}
    \| \la \psi_1, \beta \psi_2 \ra \beta \psi_3 \|_{L^p_t L^2_x} &
    \ls \| \la \psi_1, \beta \psi_2 \ra \|_{L^2}
    \| \psi_3 \|_{L^{\frac{2p}{2-p}}_t L^\infty_x} \\
    & \ls 2^{\frac{k_1}2} \| \psi_1 \|_{S^{s_1}_{k_1}}  \| \psi_2 \|_{S^{s_2}_{k_2}}   2^{(1+\frac{p-2}{2p})k_3}  \| \psi_3 \|_{S^{s_3}_{k_3}} \\
    & \ls 2^{\frac{p-2}{2p} k_4} 2^{\frac{k_3-k_2}2} 2^{\frac{p-2}{2p}
      (k_3-k_4)} TR
  \end{split}
  \]
  which is acceptable given that $0 \leq \frac{2-p}{2p} < \frac12$.
  
  If $k_1+20 \leq k_3 \leq k_2+20$ we use \eqref{eq:interp-bil2} and
  obtain
  \[
  \begin{split}
    \| \la \psi_1, \beta \psi_2 \ra \beta \psi_3 \|_{L^{p}_t L^2_x} &
    \ls \| \la \psi_1, \beta \psi_2 \ra \|_{L^{\frac{8p}{p+4}}_t
      L^2_x}
    \| \psi_3 \|_{L^{\frac{8p}{4-p}}_t L^\infty_x} \\
    & \ls 2^{(\frac78 - \frac1{2p})k_1} \| \psi_1 \|_{S^{s_1}_{k_1}}
    \| \psi_2 \|_{S^{s_2}_{k_2}}
    2^{(\frac98-\frac1{2p})k_3}  \| \psi_3 \|_{S^{s_3}_{k_3}} \\
    & \ls 2^{\frac{p-2}{2p} k_4} 2^{\frac{2-p}{2p} (k_4-k_2)}
    2^{(\frac38 - \frac1{2p})(k_1-k_2)}
    2^{(\frac58-\frac1{2p})(k_3-k_2)} TR
  \end{split}
  \]
  which is acceptable given that $\frac43 < p \leq 2$.
  
  Next we consider the case $k_2+20 \leq k_3$, in which case $k_4 \leq
  k_3+10$ or else the l.h.s.\ of \eqref{cunn1} vanishes. In this case
  the estimate \eqref{TRI1} gives the desired bound provided that
  $\frac43 < p \leq 2$.
\end{proof}
  
It remains to prove Lemma \ref{lem:cunn23}. Before we start to do so,
we analyze the modulation of a product of two waves as in \cite{BH}. We consider two
functions $\psi_1, \psi_2 \in S^{+}$ where their native modulation is
with respect to the quantity $|\tau - \la \xi \ra|$. However, for $\la
\psi_1 , \beta \psi_2 \ra$ we quantify the output modulation with
respect to $||\tau|-\la \xi \ra|$. We recall from \cite{BH} the
following lemma which contains the modulation localization claim which
will be used several times in the argument.

\begin{lem} \label{lem:mod} i) Let $k,k_1k_2\geq 100$ and $l\prec
  \min(k_1,k_2)$, and let $\kappa_1,\kappa_2 \in \mathcal{K}_{l}$,
  with $\dist(\kappa_1,\kappa_2)\approx 2^{-l}$, and assume that
  $u_j=\tilde{P}_{k_j,\kappa_j}\tilde{Q}^{s_j}_{\prec m}u_j$, where
  \[m=k_1+k_2-k-2l.\] Then, if $s_1=s_2$,
  \[\widehat{P_k(u_1\overline{u_2})}(\tau,\xi)=0 \text{ unless }
  ||\tau|-\la \xi \ra |\approx 2^{m}.\] ii) Using the same setup as in
  part i) but with $s_1=-s_2$ and $\dist( \kappa_1, -\kappa_2)\approx
  2^{-l}$, the same result applies with
  \[
  m=\min(k_1,k_2)-2l.
  \]
\end{lem}

\begin{proof}
  i) The proof of the same result in \cite{BH} (where we worked in
  dimension $3$) does not involve the dimension of the physical space,
  thus it carries over verbatim to dimension $2$ for $s_1=s_2=+$. The
  argument $s_1=s_2=-$ is entirely similar.

  ii) Since the modulation of the inputs are much less than the
  claimed modulation of the output it is enough to prove the argument
  for free solutions. Let $(\xi_1, \la \xi_1 \ra)$ be in the support
  of $\hat{u}_1$ and $(-\xi_2, \la \xi_2 \ra)$ be in the support of
  $\overline{\hat{u}_2}$. Then, the angle between $\xi_1$ and $\xi_2$
  is $\approx 2^{-l}$. Let $\xi=\xi_1-\xi_2$ be of size $2^{k}$ and
  $\tau=\la \xi_1 \ra -\la \xi_2\ra$. Our aim is to prove that
  \[|\la \xi_1-\xi_2\ra-|\la \xi_1 \ra +\la \xi_2\ra||\approx 2^{m}.\]
  The claim follows from
  \begin{align*}
    &\la \xi_1-\xi_2\ra-|\la \xi_1 \ra + \la \xi_2\ra|=\frac{\la \xi_1-\xi_2\ra^2-(\la \xi_1\ra  + \la \xi_2\ra)^2}{\la \xi_1-\xi_2\ra+|\la \xi_1\ra  + \la \xi_2 \ra|}\\
    =& \frac{2 |\xi_1||\xi_2|(1+\cos(\angle(\xi_1,\xi_2)))}{\la
      \xi_1-\xi_2\ra+|\la \xi_1\ra +\la \xi_2 \ra|}
    +\mathcal{O}(2^{-\max(k_1,k_2)})
    \\
    \approx &2^{\min(k_1,k_2)}\angle(\xi_1,\xi_2)^2
  \end{align*}
  because by assumption we have $2^{\min(k_1,k_2)-2l}\gg
  2^{-\max(k_1,k_2)}$.
\end{proof}

\begin{proof}[Proof of Lemma \ref{lem:cunn23}]
  Without restricting the generality of the argument we prove
  \eqref{cunn2} for the $+$ choice in all terms. Once we finish the
  argument for the $+$ choice in all terms, we indicate how the other
  cases are treated.  Thus, for now, we drop all the $\pm$ and simply
  consider $\psi_j \in S^+_{k_j}$ and write $S_{k_j}=S^+_{k_j}$
  instead.

  For brevity, we denote the l.h.s.\ of \eqref{cunn2} as
  \begin{equation*}
    I:=\Big| \int \la \psi_1, \beta \psi_2 \ra \cdot \la \psi_3, \beta
    \psi_4 \ra dx dt\Big|
  \end{equation*}
  and the standard factor on the r.h.s.\ as
  \begin{equation*}
    J:=\prod_{j=1}^3 2^{\frac{k_j}2} \| \psi_j \|_{S_{k_j}} \cdot 2^{-\frac{k_4}2} \| \psi_4 \|_{S^{w}_{k_4}}.
  \end{equation*}
  Since the expression $I$ computes the zero mode of the product $\la
  \psi_1, \beta \psi_2 \ra \cdot \la \psi_3, \beta \psi_4 \ra$, it
  follows that $\la \psi_1, \beta \psi_2 \ra$ and $\la \psi_3, \beta
  \psi_4 \ra$ need to be localized at frequencies and modulations of
  comparable size, where the modulation is computed with respect to
  $||\tau|-\la \xi \ra|$. This will be repeatedly used in the argument
  below along with the convention that the modulations of $\psi_k,
  k=1,\ldots,4$ are with respect to $|\tau-\la \xi \ra|$, while the
  modulations of $\la \psi_1, \beta \psi_2 \ra$ and $\la \psi_3, \beta
  \psi_4 \ra$ are with respect to $||\tau|-\la \xi \ra|$.
  
  We also agree that by the angle of interaction in, say, $\la \psi_1,
  \beta \psi_2 \ra$ we mean the angle made by the frequencies in the
  support of $\hat{\psi}_1$ and $\hat{\psi}_2$, where we consider only
  the supports that bring nontrivial contributions to $I$.
  
  We organize the argument based on the size of the frequencies.
  There are a two easy cases we can easily dispose of.

  \underline{Case 1:  $\max(k_1,k_2,k_3,k_4) \leq 200$.}
  In this case we estimate
  \begin{align*}
    I \ls{}& \| \psi_1 \|_{L^3_tL^6_x} \| \psi_2  \|_{L^3_tL^6_x}  \| \psi_3 \|_{L^3_tL^6_x}  \| \psi_4 \|_{L^\infty_t L^2_x} \\
    \ls{}&  \| \psi_1 \|_{S_{k_1}} \| \psi_2 \|_{S_{k_2}}  \| \psi_3 \|_{S_{k_3}} \| \psi_4 \|_{S_{k_4}^w} \\
    \ls{}& J
  \end{align*}

  \underline{Case 2: $k_4< 100$.} Using \eqref{bil2} in the context of
  part iv) of Proposition \ref{Pbil} we obtain:
  \begin{align*}
    I \ls{}& \| \la \psi_1 , \beta \psi_2 \ra \|_{L^2}  \| \psi_3 \|_{L^4}  \| \psi_4 \|_{L^4} \\
    \ls{}& 2^{\frac{\min(k_1,k_2)}2} \| \psi_1 \|_{S_{k_1}} \| \psi_2 \|_{S_{k_2}}  2^{\frac{k_3}2} \| \psi_3 \|_{S_{k_3}} \| \psi_4 \|_{S_{k_4}^w} \\
    \ls{}& 2^{-\frac{\max(k_1,k_2)}2} J
  \end{align*}
  Given that, in order to account for nontrivial outputs, we need to
  consider only the case when $k_3 \prec \max(k_1,k_2)$, the above
  estimate suffices.

  We continue with the more delicate cases. In light of Case 2, from
  now on we work under the hypothesis that $k_4\geq 100$.

  \underline{Case 3: $k_4 \leq \min(k_1,k_2,k_3) + 10$.}
  If $k_3 \geq k_4+10$, then we use \eqref{bil2} and \eqref{bil3} to
  obtain
  \[
  |I| \ls 2^{\frac{\min(k_1,k_2)}2} \| \psi_1 \|_{S_{k_1}} \| \psi_2
  \|_{S_{k_2}} 2^{\frac{k_3}2} \| \psi_3 \|_{S_{k_3}} \| \psi_4
  \|_{S_{k_4}^w} \ls 2^{\frac{k_4-\max(k_1,k_2)}2} J.
  \]
  which is acceptable given that $k_3 \leq \max(k_1,k_2) +10$ (or else
  $I=0$).

  If $k_4 -10 \leq k_3 \leq k_4+9$ the above argument covers most of
  $I$ except
  \[
  I_{par}:=\Big| \sum_{\ka_3,\ka_4 \in \mathcal{K}_{k_4}: \atop
    \dist(\ka_3,\ka_4) \leq 2^{-k_4+3}} \int \la \psi_1, \beta \psi_2
  \ra \cdot \la \tilde P_{\ka_3} \psi_3, \beta \tilde P_{\ka_4} \psi_4
  \ra dx dt\Big|
  \]
  If $k_1,k_2 \leq k_4 +15$ this is estimated as follows:
  \[
  \begin{split}
    I_{par} \ls{}&  \| \psi_1 \|_{L^3_t L^6_x} \| \psi_2 \|_{L^3_t L^6_x}  2^{-k_4} \sum_{\ka_3,\ka_4 \in \mathcal{K}_{k_4}: \atop \dist(\ka_3,\ka_4) \leq 2^{-k_4+3}} \| \tilde P_{\ka_3} \psi_3\|_{L^3_t L^6_x} \| \tilde P_{\ka_4}  \psi_4 \|_{L^\infty_t L^2_x} \\
    \ls{}& 2^{\frac{2(k_1+k_2)}3 -k_4} \| \psi_1 \|_{S_{k_1}} \| \psi_2
    \|_{S_{k_2}}  \left( \sum_{\ka_3 \in \mathcal{K}_{k_4}}
      \| \tilde P_{\ka_3} \psi_3\|^2_{L^3_t L^6_x} \right)^\frac12
    \left(  \sum_{\ka_4 \in \mathcal{K}_{k_4}} \| \tilde P_{\ka_4} \psi_4\|^2_{L^\infty_t L^2_x}  \right)^\frac12 \\
     \ls{}& 2^{\frac{2(k_1+k_2+k_3)}3 -k_4} \| \psi_1 \|_{S_{k_1}} \| \psi_2 \|_{S_{k_2}} \| \psi_3 \|_{S_{k_3}}  \| \psi_4 \|_{S_{k_4}^w} \\
     \ls{}& J
  \end{split}
  \]
   where we have used that $|k_i-k_4| \leq 15$, for $i \in \{1,2,3\}$.

  If $k_1 \geq k_4 +15$, then $k_2 \geq k_4+10$. In addition, since
  $\la \psi_1, \beta \psi_2 \ra$ is supported at frequency $\ls
  2^{k_4}$, it follows that only the interactions between $\psi_1$ and
  $\psi_2$ making an angle $\ls 2^{k_4-k_1}$ have nontrivial
  contribution to $I$. Therefore we need to consider only
  \[
  I_{par}:=\Big| \sum_{\ka_1,\ka_2 \in \mathcal{K}_{k_1-k_4}: \atop
    \dist(\ka_1,\ka_2) \ls 2^{k_4-k_1}} \sum_{\ka_3,\ka_4 \in
    \mathcal{K}_{k_4}: \atop \dist(\ka_3,\ka_4) \leq 2^{-k_4+3}} \int
  \la \tilde P_{\ka_1} \psi_1, \beta \tilde P_{\ka_2} \psi_2 \ra \cdot
  \la \tilde P_{\ka_3} \psi_3, \beta \tilde P_{\ka_4} \tilde \psi_4
  \ra dx dt\Big|
  \]
  Now we use a similar argument to the one when $k_1,k_2 \leq k_4
  +15$:
  \[
  \begin{split}
    I_{par}  \ls{} & 2^{k_4-k_1} \sum_{\ka_1,\ka_2 \in
      \mathcal{K}_{k_1-k_4}: \atop
      \dist(\ka_1,\ka_2) \ls 2^{k_4-k_1}} \| \tilde P_{\ka_1} \psi_1 \|_{L^3_t L^6_x} \| \tilde P_{\ka_2} \psi_2 \|_{L^3_t L^6_x}  \\
    & \cdot  2^{-k_4} \sum_{\ka_3,\ka_4 \in \mathcal{K}_{k_4}: \atop \dist(\ka_3,\ka_4) \leq 2^{-k_4+3}} \| \tilde P_{\ka_3} \psi_3\|_{L^3_t L^6_x} \| \tilde P_{\ka_4}  \psi_4 \|_{L^\infty_t L^2_x} \\
    \ls{} & 2^{-k_1} \left( \sum_{\ka_1 \in \mathcal{K}_{k_1-k_4}} \|
      \tilde P_{\ka_1} \psi_1 \|^2_{L^3_t L^6_x} \right)^\frac12
    \left( \sum_{\ka_2 \in \mathcal{K}_{k_1-k_4}} \| \tilde
      P_{\ka_2} \psi_2 \|^2_{L^3_t L^6_x} \right)^\frac12 \\
   & \cdot  \left( \sum_{\ka_3 \in \mathcal{K}_{k_4}} \| \tilde
      P_{\ka_3} \psi_3\|^2_{L^3_t L^6_x} \right)^\frac12
    \left(  \sum_{\ka_4 \in \mathcal{K}_{k_4}} \| \tilde P_{\ka_4} \psi_4\|^2_{L^\infty_t L^2_x}  \right)^\frac12 \\
    \ls{} & 2^{\frac{2(k_1+k_2+k_3)}3 -k_1} \| \psi_1 \|_{S_{k_1}} \| \psi_2 \|_{S_{k_2}} \| \psi_3 \|_{S_{k_3}}  \| \psi_4 \|_{S_{k_4}^w} \\
    \ls{}& 2^{\frac23(k_4-k_1)} J
  \end{split}
  \]
  which suffices.

  \underline{Case 4: there are exactly two $i \in \{ 1,2,3 \}$ such that $k_4 \leq k_i + 10$.}\\
  {\it Case 4 a)} Assume that $k_3 \geq k_4-10$. Since the argument is
  symmetric in $k_1$ and $k_2$, it is enough to consider the scenario
  $k_1 < k_4-10 \leq k_2$. Note that $|k_2-k_3| \leq 12$.

  To streamline the argument we ignore for a moment that in the case
  $|k_3-k_4| \leq 9$ the proof below does not cover the estimate for
  $I_{par}$. We will explain at the end how to estimate this term.

  We claim that either the angle of interactions in $\la \psi_3,\beta
  \psi_4 \ra$ is $\ls 2^{\frac{k_1-k_4}{16}}$ or at least one factor
  $\psi_j, j=1,\ldots,4$ has modulation $\gs 2^{\frac{k_1+7k_4}8}$. To
  see this, suppose that the claim is false. Then, the modulation of
  $\la \psi_1, \beta \psi_2 \ra$ is $\ls 2^{\frac{k_1+7k_4}8}$ while
  it follows from part i) of Lemma \ref{lem:mod} that the modulation
  of $\la \psi_3, \beta \psi_4 \ra$ is $\gg
  2^{\frac{k_1+7k_4}8}$. This is not possible, hence the claim is
  true. Note that in using Lemma \ref{lem:mod} we are assuming that
  $k_3,k_4 \geq 100$. If this is not the case, that is $k_3=99$, then
  $k_1,k_2,k_3,k_4 \leq 200 $ and this is covered under Case 1.
	
  In the first subcase, where the angle of interaction in $\la
  \psi_3,\beta \psi_4 \ra$ is smaller than $2^{\frac{k_1-k_4}{16}}$,
  we use \eqref{bil2} and \eqref{bil3} to estimate
  \[
    I  \ls 2^{\frac{k_1-k_4}{16}} 2^{\frac{k_1}2} 2^{\frac{k_3}2} \| \psi_1 \|_{S_{k_1}}  \| \psi_2 \|_{S_{k_2}}  \| \psi_3 \|_{S_{k_3}} \| \psi_4 \|_{S^w_{k_4}} 
     \ls 2^{\frac{k_1-k_4}{16}} 2^{\frac{k_4-k_2}2}J
  \] 
   which is acceptable.

  We now consider the second subcase, in which the modulation of the
  factor $\psi_j$ is $\gs 2^{\frac{k_1+7k_4}{8}} \gs
  2^{\frac{k_1+k_4}2}$ for some $j\in \{1,2,3,4\}$.

  $j=1$: Since $\psi_1$ has modulation $\gs 2^{\frac{k_1+k_4}2}$, we
  use the Sobolev embedding for $\psi_1$ to obtain
  \begin{align*}
    I \ls{}& \| \psi_1 \|_{L^2_tL^\infty_x} \| \psi_2 \|_{L^\infty_t L^2_x} 2^{\frac{k_3}2} \| \psi_3 \|_{S_{k_3}} \| \psi_4 \|_{S_{k_4}^w} \\
    \ls{}& 2^{k_1} \| \psi_1 \|_{L^2}  \| \psi_2 \|_{S_{k_2}}  2^{\frac{k_3}2} \| \psi_3 \|_{S_{k_3}} \| \psi_4 \|_{S^w_{k_4}} \\
    \ls{}& 2^{\frac{k_1-k_4}4} 2^{\frac{k_4-k_2}2} J.
  \end{align*}

  $j=2$: Since $\psi_2$ has modulation $\gs 2^{\frac{k_1+k_4}2}$,
  Sobolev embedding for $\psi_1$ and \eqref{bil2} yields
  \begin{align*}
    I \ls{}& \| \psi_1 \|_{L^\infty} \| \psi_2 \|_{L^2} 2^{\frac{k_3}2} \| \psi_3 \|_{S_{k_3}} \| \psi_4 \|_{S_{k_4}^w} \\
    \ls{}& 2^{k_1} \| \psi_1 \|_{L^\infty_t L^2_x}
    2^{-\frac{k_1+k_4}4} \| \psi_2 \|_{S_{k_2}}
    2^{\frac{k_3}2} \| \psi_3 \|_{L^4} \| \psi_4 \|_{S^w_{k_4}} \\
    \ls{}& 2^{\frac{k_1-k_4}4} 2^{\frac{k_4-k_2}2} J.
  \end{align*}
 
  $j=3$: We use \eqref{eq:interp-bil2} and estimate as follows
  \begin{align*}
    I \ls{}& \| \la \psi_1, \beta \psi_2 \ra \|_{L^{\frac{p}{p-1}}_t L^2_x}  \| \psi_3 \|_{L^p_t L^2_x} \| \psi_4 \|_{L^\infty } \\
    \ls{}& 2^{\frac{k_1}{p}} \| \psi_1 \|_{S_{k_1}} \| \psi_2
    \|_{S_{k_2}}
    2^{(1-\frac1p) k_3} 2^{-\frac{k_1+7k_4}8} \| \psi_3 \|_{S_{k_3}} 2^{k_4} \| \psi_4 \|_{S^w_{k_4}} \\
    \ls{}& 2^{(\frac{1}p-\frac58)(k_1-k_3)} 2^{\frac{5k_4-4k_2-k_3}8}
    J.
  \end{align*}
  which is acceptable provided we choose a $\frac43 < p < \frac85$.
  
  $j=4$: We \eqref{eq:interp-bil2} and estimate as follows:
  \begin{align*}
    I \ls{}& \| \la \psi_1, \beta \psi_2 \ra \|_{L^r_t L^2_x} \| \psi_3 \|_{L^{\frac{2r}{r-2}}_t L^\infty_x}  \| \psi_4 \|_{L^2} \\
    \ls{}& 2^{(1-\frac1r)k_1} \| \psi_1 \|_{S_{k_1}} \| \psi_2
    \|_{S_{k_2}}
    2^{(\frac12+\frac1r)k_3}  \| \psi_3 \|_{S_{k_3}} 2^{-\frac{k_1+7k_4}{16}} \| \psi_4 \|_{S_{k_4}^w} \\
    \ls{}& 2^{(\frac{7}{16}-\frac1r)(k_1-k_3)} 2^{\frac{k_4-k_3}{16}}
    J
  \end{align*}
  and this is acceptable provided we pick $4 > r > \frac{16}7$.
   
  The argument is complete, except that we owe an estimate for
  $I_{par}$ in the case $|k_3-k_4| \leq 9$. Note that, in this case we
  also have $k_2 \leq k_4 +15$.  By recombining $\psi_1$ with
  $\psi_4$, $\psi_2$ with $\psi_3$ (at the cost of having no null
  structure) and using \eqref{bas1}, we estimate
  \[
  \begin{split}
    I_{par} \ls 2^{-k_4} 2^{k_1} \| \psi_{k_1} \|_{S_{k_1}} \|
    \psi_{k_4} \|_{S^w_{k_4}} 2^{k_2} \| \psi_{k_2} \|_{S_{k_2}} \|
    \psi_{k_3} \|_{S_{k_3}} \ls 2^{\frac{k_1-k_4}2} J.
  \end{split}
  \]

  {\it Case 4 b)} Assume now that $k_3 \leq k_4-10$, hence $k_1,k_2
  \geq k_4-10$ and $|k_1-k_2| \leq 12$. Here we claim that either the
  angle of interactions in $\la \psi_1,\beta \psi_2 \ra$ is $\ls
  2^{\frac{k_3-k_4}{16}} 2^{k_4-k_2}$ or at least one factor $\psi_j,
  j=1,\ldots,4$ has modulation $\gs 2^{\frac{k_3+7k_4}8}$. Indeed, if
  the claim is false, it follows from Lemma \ref{lem:mod}, part i),
  that the modulation of $\la \psi_1, \beta \psi_2 \ra$ is $\gg
  2^{\frac{k_3+7k_4}8}$ while the modulation of $\la \psi_3, \beta
  \psi_4 \ra$ is $\ll 2^{\frac{k_3+7k_4}8}$. This is not possible,
  hence the claim is true. Note that in using Lemma \ref{lem:mod} we
  are assuming that $k_1,k_2 \geq 100$. If this is not the case, that
  is either $k_1=99$ or $k_2=99$, then $k_1,k_2,k_3,k_4 \leq 200$ the
  argument is provided in Case 1.
   
  In the first subcase the angle of interaction in $\la \psi_1,\beta
  \psi_2 \ra$ is smaller than $2^{\frac{k_3-k_4}{16}}
  2^{k_4-k_2}$. Then, we use \eqref{bil3} to estimate the contribution
  of $\la \psi_1,\beta \psi_2 \ra$ and \eqref{bil2} to estimate the
  contribution of $\la \psi_3,\beta \psi_4 \ra$. This gives $I \ls
  2^{\frac{k_3-k_4}{16}} 2^{k_4-k_2} J$ which is acceptable.

  In the second subcase, where at least one modulation is high, one
  proceeds in a similar manner to Case 2a) above. We indicate the
  starting point in each case and leave the details to the reader.
  
  $j=1$: We proceed as in the case $j=4$, Case 2a):
  \begin{align*}
    I \ls \| \psi_1 \|_{L^2} \| \psi_2 \|_{L^{\frac{2p}{p-2}}_t
      L^\infty_x} \| \la \psi_3 , \beta \psi_4 \ra \|_{L^p_t L^2_x}.
  \end{align*}

  $j=2$: Identical to the case $j=1$.

  $j=3$: We proceed as in the case $j=1$, Case 2a):
  \begin{align*}
    I \ls{}& 2^{\frac{k_1}2} \| \psi_1 \|_{S_{k_1}} \| \psi_2
    \|_{S_{k_2}} \| \psi_3 \|_{L^2_t L^\infty_x} \| \psi_4
    \|_{L^\infty_t L^2_x }.  \end{align*}
   
  $j=4$: We proceed as in the case $j=2$, Case 2b):
  \begin{align*}
    I \ls{} 2^{\frac{k_1}2} \| \psi_1 \|_{S_{k_1}} \| \psi_2
    \|_{S_{k_2}} \| \psi_3 \|_{L^\infty} \| \psi_4 \|_{L^2 }.
  \end{align*}
 
  \underline{Case 5: $| k_2 -k_4 | \leq 2$ and $k_1,k_3 \leq k_4
    -10$.}  Without restricting the generality of the argument, we may
  assume that $k_1 \leq k_3$.
 
  We claim that either the angle of interaction in $ \la \psi_3, \beta
  \psi_4 \ra$ is $\ls 2^{\frac{k_1-k_3}{16}}$ or one factor $\psi_j,
  j=1,\ldots,4$ has modulation $\gs 2^{\frac{k_1+7k_3}8}$. Indeed, if
  all modulations of the functions involved are $\ll
  2^{\frac{k_1+7k_3}8}$, then $ \la \psi_1, \beta \psi_2 \ra$ is
  localized at modulation $\ls 2^{\frac{k_1+7k_3}8}$.  This forces $
  \la \psi_3, \beta \psi_4 \ra$ to be localized at modulation $\ls
  2^{\frac{k_1+7k_3}8}$, hence the angle of interaction is $\ls
  2^{\frac{k_1-k_3}{16}}$ by Lemma \ref{lem:mod}, part i). Note that
  in using Lemma \ref{lem:mod} we are assuming that $k_3,k_4 \geq
  100$. If this is not the case, that is $k_3=99$, then $k_1=99$ and
  the estimate $ I \ls J$ suffices.

  In the first subcase, when the angle of interaction in $ \la \psi_3,
  \beta \psi_4 \ra$ is $\ls 2^{\frac{k_1-k_3}{16}}$, we use
  \eqref{bil3} to obtain
  \[
  I \ls 2^{\frac{k_1-k_3}{16}} 2^{\frac{k_1}2} \| \psi_1 \|_{S_{k_1}}
  \| \psi_2 \|_{S_{k_2}} 2^{\frac{k_3}2} \| \psi_3 \|_{S_{k_3}} \|
  \psi_4 \|_{S_{k_4}^w} \ls 2^{\frac{k_1-k_3}{16}} J.
  \]

  Next, we consider the second subcase when the factor $\psi_j$ has
  modulation $\gs 2^{\frac{k_1+7k_3}8} \gs 2^{\frac{k_1+3k_3}4} $ for
  some $j \in \{1,2,3,4\}$:
 
  $j=1$: The modulation of $\psi_1$ is $\gs 2^{\frac{k_1+3k_3}4}$, so
  we use Sobolev embedding for $\psi_1$ and \eqref{bil2} for $\la
  \psi_3, \beta \psi_4 \ra$ to obtain
  \begin{align*}
    I \ls{}& \| \psi_1 \|_{L^2_t L^\infty_x} \| \psi_2 \|_{L^\infty_t
      L^2_x} 2^{\frac{k_3}2}
    \| \psi_3 \|_{S_{k_3}} \| \psi_4 \|_{S_{k_4}^w} \\
    \ls{}& 2^{k_1} \| \psi_1 \|_{L^2} \| \psi_2 \|_{S_{k_2}}
    2^{\frac{k_3}2}
    \| \psi_3 \|_{S_{k_3}} \| \psi_4 \|_{S_{k_4}^w} \\
    \ls{}& 2^{k_1} 2^{-\frac{k_1+3k_3}8} \| \psi_1 \|_{S_{k_1}} \|
    \psi_2 \|_{S_{k_2}} 2^{\frac{k_3}2}
    \| \psi_3 \|_{S_{k_3}} \| \psi_4 \|_{S_{k_4}^w} \\
    \ls{}& 2^{\frac{k_1-k_3}8} J.
  \end{align*}

  $j=2$: Here, the modulation of $\psi_2$ is $\gs
  2^{\frac{k_1+3k_3}4}$ and we proceed as above to obtain
  \begin{align*}
    I \ls{}& \| \psi_1 \|_{L^\infty} \| \psi_2 \|_{L^2}
    2^{\frac{k_3}2}
    \| \psi_3 \|_{S_{k_3}} \| \psi_4 \|_{S_{k_4}^w} \\
    \ls{}& 2^{k_1} \| \psi_1 \|_{L^\infty_t L^2_x}
    2^{-\frac{k_1+3k_3}8} \| \psi_2 \|_{S_{k_2}} 2^{\frac{k_3}2}
    \| \psi_3 \|_{S_{k_3}} \| \psi_4 \|_{S_{k_4}^w} \\
    \ls{}& 2^{\frac{k_1-k_3}8} J.
  \end{align*}

  $j=3$: The modulation of $\psi_3$ is $\gs 2^{\frac{k_1+7k_3}8}$, we
  use the Sobolev embedding for $\psi_3$ to obtain
  \begin{align*}
    I \ls{}& \| \la \psi_1, \beta \psi_2 \ra \|_{L^{\frac{p}{p-1}}_t L^2_x}  \| \psi_3 \|_{L^{p}_t L^\infty_x} \| \psi_4 \|_{L^\infty_t L^2_x } \\
    \ls{} & \| \la \psi_1, \beta \psi_2 \ra \|_{L^{\frac{p}{p-1}}_t L^2_x}  2^{k_3} \| \psi_3 \|_{L^p_t L^2_x} \| \psi_4 \|_{L^\infty_t L^2_x } \\
    \ls{}& 2^{ \frac{k_1}{p}} \| \psi_1 \|_{S_{k_1}} \| \psi_2
    \|_{S_{k_2}}
    2^{(2-\frac1p) k_3} 2^{-\frac{k_1+7k_3}8} \| \psi_3 \|_{S_{k_3}}  \| \psi_4 \|_{S^w_{k_4}} \\
    \ls{}& 2^{(\frac{1}p -\frac58)(k_1-k_3)} J.
  \end{align*}
  which is acceptable provided we choose a $\frac43 < p < \frac85$.
  
  $j=4$: Since the modulation of $\psi_4$ is $\gs
  2^{\frac{k_1+7k_3}8}$, we estimate as follows
  
  \begin{align*}
    I \ls{}& \| \la \psi_1, \beta \psi_2 \ra \|_{L^p_t L^2_x} \| \psi_3 \|_{L^{\frac{2p}{p-2}}_t L^\infty_x}  \| \psi_4 \|_{L^2} \\
    \ls{}& 2^{(1-\frac1p)k_1} \| \psi_1 \|_{S_{k_1}} \| \psi_2
    \|_{S_{k_2}}
    2^{(\frac12+\frac1p)k_3}  \| \psi_3 \|_{S_{k_3}} 2^{-\frac{k_1+7k_3}{16}} \| \psi_4 \|_{S_{k_4}^w} \\
    \ls{}& 2^{(\frac{7}{16}-\frac1p)(k_1-k_3)} J
  \end{align*}
  and this is acceptable provided we pick $4 > p > \frac{16}7$.
  
  \underline{Case 6: $| k_1 -k_4 | \leq 2$ and $k_2,k_3 \leq k_4
    -10$.} By switching the roles of $\psi_1$ and $\psi_2$, this case
  is entirely similar to Case 5.
  
  \underline{Case 7: $|k_3-k_4| \leq 2$ and $k_1,k_2 \leq k_4 -
    10$}. Without loss of generality we assume $k_1 \leq k_2$.
  Since $|k_3-k_4| \leq 2$ there will be a problem with estimating
  $I_{par}$.  We estimate this term the same way we did in Case 3 (see
  $k_1,k_2 \leq k_4+15$ part there) to obtain:
  $
  I_{par} \ls 2^{\frac{k_1+k_2-2k_4}6} J$
  and this is fine. As a consequence, in the rest of the argument we
  can tacitly ignore that the estimates we provide do not work for the
  $I_{par}$ part of $I$.

  The key observation is that either the angle of interaction between
  $\psi_3$ and $\psi_4$ is $\ls 2^{\frac{k_1-k_2}{16}} 2^{k_2-k_3}$ or
  at least one factor has modulation $\gs
  2^{\frac{k_1+7k_2}8}$. Indeed, if all modulations are $\ll
  2^{\frac{k_1+7k_2}8}$, then the modulation of $\la
  \psi_1,\beta\psi_2\ra$ is $\ls 2^{\frac{k_1+7k_2}8}$ and part i) of
  Lemma \ref{lem:mod} implies the claim. Note that in using Lemma
  \ref{lem:mod} we are assuming that $k_3,k_4 \geq 100$. If this is
  not the case, that is $k_3=99$, then $k_1,k_2,k_3,k_4 \leq 200$ and
  the argument is provided in Case 1.

  We consider the first subcase, when the angle of interaction between
  $\psi_3$ and $\psi_4$ is $\ls 2^{\frac{k_1-k_2}{16}}
  2^{k_2-k_3}$. Using \eqref{bil2} and \eqref{bil3} we estimate
  \[
  I \ls 2^{\frac{k_1}2} \| \psi_1 \|_{S_{k_1}} \| \psi_2 \|_{S_{k_2}}
  2^{\frac{k_3}2} 2^{\frac{k_1-k_2}{32}} 2^{\frac{k_2-k_3}2} \| \psi_3
  \|_{S_{k_3}} \| \psi_4 \|_{S_{k_4}^w} \ls 2^{\frac{k_1-k_2}{32}} J.
  \]
  
  In the second subcase, $\psi_j$ has modulation $\gs
  2^{\frac{k_1+7k_2}8}$ for some $j \in \{1,2,3,4\}$.
  
  $j=1$: The modulation of $\psi_1$ is $\gs
  2^{\frac{k_1+7k_2}{8}}$. Using \eqref{tris2} with $p=2$ for $\psi_2,
  \psi_3, \psi_4$, and the Sobolev embedding for $\psi_1$ we estimate
  \begin{align*}
    I \ls{}& \| \psi_1 \|_{L^2_t L^\infty_x} 2^{\frac{k_2}2} \| \psi_2 \|_{S_{k_2}} \| \psi_3 \|_{S_{k_3}} \| \psi_4 \|_{S^w_{k_4}} \\
    \ls{}& 2^{k_1} \| \psi_1 \|_{L^2}  2^{\frac{k_2}2} \| \psi_2 \|_{S_{k_2}}  \| \psi_3 \|_{S_{k_3}} \| \psi_4 \|_{S_{k_4}^w} \\
    \ls{}& 2^{ k_1} 2^{-\frac{k_1+7k_2}{16}} \| \psi_1 \|_{S_{k_1}}
    2^{\frac{k_2}2} \| \psi_2 \|_{S_{k_2}}  \| \psi_3 \|_{S_{k_3}} \| \psi_4 \|_{S_{k_4}^w} \\
    \ls{}& 2^{\frac{7}{16}(k_1-k_2)} J .
  \end{align*}
  
  $j=2$: Using \eqref{tris2} for $\psi_1, \psi_3, \psi_4$ and the
  Sobolev embedding for $\psi_2$ we proceed as follows:
  \begin{align*}
    I \ls{}& 2^{(1-\frac1q)k_1}\| \psi_1 \|_{S_{k_1}} \| \psi_2 \|_{L^{\frac{q}{q-1}}_t L^\infty_x} \| \psi_3 \|_{S_{k_3}} \| \psi_4 \|_{S^w_{k_4}} \\
    \ls{}& 2^{(1-\frac1q)k_1} \| \psi_1 \|_{S_{k_1}}  2^{k_2} \| \psi_2 \|_{L^{\frac{q}{q-1}}_t L^2_x}  \| \psi_3 \|_{S_{k_3}} \| \psi_4 \|_{S_{k_4}^w} \\
    \ls{}& 2^{(1-\frac1q)k_1} \| \psi_1 \|_{S_{k_1}}
    2^{k_2} 2^{\frac{k_2}q} 2^{-\frac{k_1+7k_2}8} \| \psi_2 \|_{S_{k_2}}  \| \psi_3 \|_{S_{k_3}} \| \psi_4 \|_{S_{k_4}^w} \\
    \ls{}& 2^{(\frac38 - \frac1q)(k_1-k_2)} J.
  \end{align*}
  which is acceptable as long as $p=\frac{q}{q-1} \in
  (\frac43,\frac85)$ and $\frac1q<\frac38$, which is both satisfied as
  long as $\frac83<q<4$.

  $j=3$ and $j=4$: Here we assume that $\psi_3$ and $\psi_4$ have
  modulation $\gs 2^{\frac{k_1+7k_2}8}$. In this case we estimate
  \begin{align*}
    I \ls{}& \| \psi_1 \|_{L^\infty} \| \psi_2 \|_{L^\infty} \|\psi_3\|_{L^2} \|\psi_4\|_{L^2} \\
    \ls{}& 2^{k_1+k_2} \| \psi_1 \|_{S_{k_1}} \| \psi_2 \|_{S_{k_2}}
    2^{-\frac{k_1+7k_2}{8}}\| \psi_3 \|_{S_{k_3}} \| \psi_4 \|_{S_{k_4}^w} \\
    \ls{}& 2^{\frac38(k_1-k_2)} J.
  \end{align*}

  $j=3$ (only): The modulation of $\psi_3$ is $\gs
  2^{\frac{k_1+7k_2}8}$ and all the other terms have modulation $\ll
  2^{\frac{k_1+7k_2}8}$. In this case we note that the angle of
  interaction between $\psi_2$ and $\psi_4$ is $\gs
  2^{\frac{k_1-k_2}{16}}$ or else their interaction has modulation
  $\ll 2^{\frac{k_1+7k_2}8}$ and this cannot be changed by $\psi_1$ to
  match the modulation of $\psi_3$. Thus combine $\psi_2$ and
  $\psi_4$, use \eqref{bas1} to obtain
  
  \begin{align*}
    I \ls{}& \| \psi_1 \|_{L^\infty}  2^{\frac{k_2}2} 2^{-\frac{k_1-k_2}{32}} \| \psi_2 \|_{S_{k_2}}  \| \psi_4 \|_{S_{k_4}} \|\psi_3\|_{L^2}  \\
    \ls{}& 2^{k_1} \| \psi_1 \|_{S_{k_1}} 2^{\frac{k_2}2}
    2^{-\frac{k_1-k_2}{32}} \| \psi_2 \|_{S_{k_2}}
    2^{-\frac{k_1+7k_2}{16}}\| \psi_3 \|_{S_{k_3}} \| \psi_4 \|_{S_{k_4}^w} \\
    \ls{}& 2^{(\frac7{16}-\frac1{32})(k_1-k_2)} J.
  \end{align*}

  $j=4$ (only): We change the role of $\psi_3$ and $\psi_4$ in the
  above argument.

  \vspace{.2in}
 
  We are now done with the analysis of \eqref{cunn2} in the case
  $s_1=s_2=s_3=s_4=+$. It is obvious that the same argument works for
  $s_1=s_2=s_3=s_4=-$. Next we indicate how the other sign choices can
  be dealt with, by highlighting the similarities and differences from
  the choice $s_1=s_2=s_3=s_4=+$. We do this by going over each case.
 
  No changes are needed in the easy cases: Case 1 and Case 2.

  \underline{Case 3: $k_4 \leq \min(k_1,k_2,k_3) + 10$.}
  Here the only part that needs to be adjusted is the last scenario
  when $k_4-10 \leq k_3 \leq k_4+9 , k_1 > k_4+15, k_2 > k_4+10$ and
  $s_1=-s_2$. As already argued there, only the interactions between
  $\psi_1$ and $\psi_2$ making an angle $\ls 2^{k_4-k_1}$ have
  nontrivial contribution to $I$, that is only pairs $\la \tilde
  P_{\ka_1} \psi_1, \beta \tilde P_{\ka_2} \psi_2 \ra$ with
  $d(\ka_1,\ka_2) \ls 2^{k_4-k_1}$. But this implies $d(\ka_1,-\ka_2)
  \approx 1$, and we claim that at least one factor has modulation
  $\gs 2^{k_1}$. Indeed, otherwise all factors have modulations $\ll
  2^{k_1}$ from which we obtain two contradictory results: $\la
  \psi_1, \beta \psi_2\ra$ has modulation $\approx 2^{k_1}$ (on behalf
  of part ii) of Lemma \ref{lem:mod}) while $\la \psi_3, \beta \psi_4
  \ra$ has modulation $\ll 2^{k_1}$.
  
  Now it is an easy exercise to establish the desired estimate, given
  that at least one factor has modulation $\gs 2^{k_1}$.

  \underline{Case 4: there are exactly two $i \in \{ 1,2,3 \}$ such that $k_4 \leq k_i + 10$.}\\
  {\it Case 4 a)} Assume that $k_3 \geq k_4-10$. The argument is the
  same if $s_3=s_4$.  If $s_3=-s_4$ then the new claim is: either the
  angle of interactions in $\la \psi_3,\beta \psi_4 \ra$ is
  $\pi+\alpha$ with $|\alpha| \ls 2^{\frac{k_1-k_4}{16}}$ or at least
  one factor $\psi_j, j=1,\ldots,4$ has modulation $\gs
  2^{\frac{k_1+7k_4}8}$. This claim is proved in a similar manner,
  just that now we invoke part ii) of Lemma \ref{lem:mod}. Then the
  rest of the argument is carried in a similar manner.
  
  {\it Case 4 b)} Assume that $k_3 \leq k_4-10$, hence $k_1,k_2 \geq
  k_4-10$ and $|k_1-k_2| \leq 12$.  If $s_1=s_2$ the proof is the
  same.

  If $s_1=-s_2$ and $k_1,k_2 \leq k_4+10$, then the claim there is
  modified as follows: either the angle of interactions in $\la
  \psi_1,\beta \psi_2 \ra$ is $\pi+\alpha$ with $|\alpha| \ls
  2^{\frac{k_3-k_4}{16}}$ or at least one factor $\psi_j,
  j=1,\ldots,4$ has modulation $\gs 2^{\frac{k_3+7k_4}8}$.  This is
  proved using part ii) of Lemma \ref{lem:mod}. Then the rest of the
  argument follows in a similar manner.

  If $s_1=-s_2$ and $\max(k_1,k_2) \geq k_4+11$, in which case
  $k_1,k_4 \geq k_4+6$, then only interactions at angle $\ls 1$ in
  $\la \psi_1,\beta \psi_2 \ra$ contribute to $I$ given that the
  output $\la \psi_1,\beta \psi_2 \ra$ is localized at much lower
  frequency. Using part ii) of Lemma \ref{lem:mod} we conclude that at
  least one factor $\psi_j$ has modulation $\gs 2^{k_1}$ and then the
  argument becomes easier.

  \underline{Case 5: $| k_2 -k_4 | \leq 2$ and $k_1,k_3 \leq k_4
    -10$.}  Without restricting the generality of the argument, we may
  assume that $k_1 \leq k_3$.
  
  No modification is needed if $s_3=s_4$. If $s_3=-s_4$ then the claim
  is modified to: either the angle of interaction in $ \la \psi_3,
  \beta \psi_4 \ra$ is $\pi+\alpha$ with $|\alpha| \ls
  2^{\frac{k_1-k_3}{16}}$ or one factor $\psi_j, j=1,\ldots,4$ has
  modulation $\gs 2^{\frac{k_1+7k_3}8}$. This is done using part ii)
  of Lemma \ref{lem:mod}.  The rest of the argument is similar.
  
  \underline{Case 6: $| k_1 -k_4 | \leq 2$ and $k_2,k_3 \leq k_4
    -10$.} By switching the roles of $\psi_1$ and $\psi_2$, this case
  is entirely similar to Case 5.
  
  \underline{Case 7: $|k_3-k_4| \leq 2$ and $k_1,k_2 \leq k_4 -
    10$}. Without loss of generality we assume $k_1 \leq k_2$.
  No modification is needed if $s_3=s_4$. If $s_3=-s_4$ then only
  interactions at angle $\ls 1$ in $\la \psi_3,\beta \psi_4 \ra$
  contribute to $I$ given that the output $\la \psi_3,\beta \psi_4
  \ra$ is localized at much lower frequency. Using part ii) of Lemma
  \ref{lem:mod} we conclude that at least one factor $\psi_j$ has
  modulation $\gs 2^{k_4}$ and then the argument becomes easier.
\end{proof}

Based on Theorem \ref{thm:thnon} we can now prove Theorem
\ref{thm:main} concerning the global well-posedness and scattering of
the cubic Dirac equation for small data.
\begin{proof}[Proof of Theorem \ref{thm:main}]
  In Section \ref{sect:setupD} we reduced the study of the cubic Dirac
  equation to the study of the system \eqref{CDsys}.  In the
  nonlinearity of \eqref{CDsys} we split the functions into
  $\psi=\psi_+ + \psi_-$ where $\psi_\pm = \Pi_\pm \psi$ and note that
  $\psi_\pm = \Pi_\pm \psi_\pm$. Using the nonlinear estimate in
  Theorem \ref{thm:thnon} and the linear estimates in Corollary
  \ref{cor:lin}, a standard fixed point argument in a small ball in
  the space $S^{+,\frac12}_C(I) \times S^{-,\frac12}_C(I)$ gives local
  existence on every time interval $I$ containing $0$, uniqueness and
  Lipschitz continuity of the flow map for small initial data
  $(\psi_+(0),\psi_-(0))\in H^\frac12(\R^2)\times
  H^\frac12(\R^2)$. Since all the bounds are independent on the size
  of $I$, this implies global existence, uniqueness and Lipschitz
  continuity of the flow map for small initial data
  $(\psi_+(0),\psi_-(0))\in H^\frac12(\R^2)\times H^\frac12(\R^2)$.

  Concerning scattering, we simply use the fact that $\psi_\pm \in
  V^2_\pm H^\frac12$: this is obtained first on every time interval
  $I$ with bounds independent of the size of $I$ which then implies
  the global bound on $\R$.
\end{proof}

\subsection*{Acknowledgement}
The first author was supported in part by NSF grant
  DMS-1001676.  The second author acknowledges support from the German
  Research Foundation, Collaborative Research Center 701.

\bibliographystyle{plain} \bibliography{cubic-dirac-2d-refs}

\end{document}